\documentclass[12pt]{amsart}
\usepackage{amsfonts}
\usepackage{amsmath}
\usepackage{amssymb}
\usepackage[margin=1.2in]{geometry}
\usepackage{verbatim}
\usepackage{amsfonts}
\usepackage{amsmath}
\usepackage{amssymb}
\usepackage{amsthm}
\usepackage{xcolor}
\newtheorem{theorem}{Theorem}[section]

\newtheorem{lemma}{Lemma}[section]

\newtheorem{remark}[theorem]{Remark}

\usepackage{bm}

\newcommand{\nrm}[1]{\left\| #1 \right\|}
\newcommand{\abs}[1]{\left\lvert #1 \right\rvert}

\newcommand {\scal}[2]{\left(#1,#2\right)}

\DeclareMathOperator{\di}{d\hspace{-1.5pt}}

\newcommand {\dt}{ \di t}
\newcommand{\ds}{\di s}

\newcommand{\dX}{\di \X}
\newcommand {\NN } {{\mathbb N}}

\newcommand {\I}{[0,T]}
\newcommand {\Iopen}{(0,T)}

\newcommand {\domain}{\Omega}

\newcommand {\lp}[1]{\Leb^{#1} (\domain )}
\newcommand {\Lp}[1]{{\bf L}^{#1} (\domain )}

\DeclareMathOperator{\Leb}{L}

\DeclareMathOperator{\Cont}{C}
\DeclareMathOperator{\Hi}{H}
\DeclareMathOperator{\HHi}{\bf H}
\DeclareMathOperator{\Id}{Id}
\DeclareMathOperator{\argmin}{argmin}
\newcommand {\hk}[1]{\Hi^{#1}(\domain )}
\newcommand {\hko}[1]{\Hi^{#1}_0(\domain )}

\newcommand {\Hk}[1]{{\bf H}^{#1}(\domain)}
\newcommand {\Hko}[1]{{\bf H}_0^{#1}(\domain)}

\newcommand {\lpIlp}{\Leb^2\left(\Iopen,\lp{2}\right)}

\newcommand {\lpkIX}[2]{\Leb^{#1}\left(\Iopen,#2\right)}

\newcommand {\cIX}[1]{\Cont\left(\I,#1\right)}

\def \vector#1{\mathbf{#1}}

\newcommand {\X}{\vector{x}}

\newcommand{\weakto}{\rightharpoonup}
\usepackage{comment}
\usepackage[shortlabels]{enumitem}

\usepackage{tabularx} 
\usepackage{caption}
\usepackage{subcaption}
\usepackage{graphicx}
\usepackage{multirow}

\makeatletter
\@namedef{subjclassname@2020}{%
  \textup{2020} Mathematics Subject Classification}
\makeatother

\usepackage{tabularx} 
\usepackage{caption}
\usepackage{subcaption}
\usepackage{graphicx}
\usepackage{multirow}

\usepackage{xcolor}
\usepackage[hidelinks]{hyperref}
\hypersetup{
    colorlinks = true,
    linkcolor = red,
    anchorcolor = red,
    citecolor = green,
    filecolor = red,
    urlcolor = red,
    linktoc=page
    }

\usepackage{cleveref}
\begin{document}

\title[Numerical Algorithms for ISPs in TE]{Numerical Algorithms for the Reconstruction of Space-Dependent Sources in Thermoelasticity}
	
	\author[F.~Maes]{Frederick Maes$^1$}

	\author[K.~Van~Bockstal]{Karel Van Bockstal$^2$} 
	\thanks{The work of K.~Van Bockstal was supported by the Methusalem programme of Ghent University Special Research Fund (BOF) (Grant Number 01M01021)} 
	
	\address[1]{Research Group NaM$^2,$ Department of Electronics and Information Systems\\ Ghent University \\ Krijgslaan 281 \\ B 9000 Ghent \\ Belgium}
	\email{frederick.maes@UGent.be}
	\address[2]{Ghent Analysis \& PDE center, Department of Mathematics: Analysis, Logic and Discrete Mathematics\\ Ghent University\\
		Krijgslaan 281\\ B 9000 Ghent\\ Belgium} 
	\email{karel.vanbockstal@UGent.be}


 \begin{abstract}
      This paper investigates the inverse problems of determining a space-dependent source for thermoelastic systems of type III under adequate time-averaged or final-in-time measurements and conditions on the time-dependent part of the sought source. Several numerical methods are proposed and examined, including a Landweber scheme and minimisation methods for the corresponding cost functionals, which are based on the gradient and conjugate gradient method. A shortcoming of these methods is that the values of the sought source are fixed ab initio and remain fixed during the iterations. The Sobolev gradient method is applied to overcome the possible inaccessibility of the source values at the boundary. Numerical examples are presented to discuss the different approaches and support our findings based on the implementation on the FEniCSx platform.
 \end{abstract}
 \maketitle 
 \tableofcontents

\section{Introduction and problem formulation}\label{sec:intro}
The theory of thermoelasticity aims to combine the theories of elasticity and heat conduction by coupling the mechanical and thermal variables of the system as described in the work of Biot \cite{Biot1956}. The material response to thermal phenomena and the impact of deformation on the thermal properties was characterised by Green and Naghdi \cite{Green1991} and was labelled as type-{I}, type-{II} and type-{III} (generalised) thermoelasticity, see also \cite{chandrasekharaiah1986, Chandrasekharaiah1998,GurtinPipkin1968, Povstenko2014} for an in-depth discussion of related models.  Noninvasive measurements for the thermoelastic model are discussed in \cite{Krouskop1987-pb}. This model has applications ranging from the biomedical problem of fitting the socket of a prosthetic limb \cite{Mak2001-iu, Tonuk2003-ov}, thermoforming and moulding processes of polymer sheets \cite{Alphonse2022, Warby2003}, and solidification of metals on a mould \cite{Demir2017}.

Over the last decades, there has been a surge in the literature on inverse problems (IPs). For those related to thermoelasticity, see, for example  \cite{Bellassoued2011, Dhaeyer2014, Hu2019, Ikehata2018, Karageorghis2014,Kozlov2009, Lesnic2016,fmkvb2022, Marin2016, Tanaka2005, VanBockstal2017b,VanBockstal2022, VanBockstal2014, Wu2012, Yu1999} and references therein for further discussions and valuable works about IPs. 
In relation to the ISPs in this work, we refer to \cite{Bellassoued2011,fmkvb2022,VanBockstal2014,Cao2019,VanBockstal2017, Wu2013} and the research summaries \cite{ MaesVanBockstalMPiPA,VanBockstalMWCAPDE2023}. The inverse problem of determining unknown positive thermal conductivity components based on transient temperature measurements at fixed locations in a two-dimensional setting was studied in \cite{Cao2019}. The analysis was performed by minimising the associated objective functional via the conjugate gradient method, for which the sensitivity and adjoint problems were written down. The obtained $\Leb^2$-gradient was regularised by solving an elliptic boundary-value problem to obtain the Sobolev gradient, see also \cite{Cao2018} for the IP of estimating the perfusion coefficient.  The Sobolev gradient method \cite{Neuberger2010}, was already used in \cite{JinZou2010} to estimate a Robin coefficient on an inaccessible part of the boundary from Cauchy data on the accessible part. In \cite{Alosaimi2024} a nonlinear minimization problem for the reconstruction of a time-dependent perfusion coefficient is investigated. The analysis in our work is closely related to these articles. Let us mention that \cite{NovruziProtas2018,Novruzi2023} considered unconstrained optimisation problems in which the authors provided and analysed modified gradient descent algorithms.  
The conjugate gradient method was applied in \cite{Hasanov2013b} to identify a time-dependent heat source from overdetermined boundary data, and in \cite{Hasanov2014} for a space-dependent source based on final-in-time and time-averaged measurements, both for a variable coefficient heat equation; see also \cite{Hasanov2011}. The identification of unknown space-dependent sources and unknown time-dependent sources for the inverse problem associated with the advection-diffusion equation with a space-dependent diffusion coefficient was studied in \cite{Sebu2017} using Dirichlet boundary data as an available measurement. The uniqueness analysis for source identification in a nonlinear parabolic equation and for the wave equation with nonlinear damping (both with time- and space-depending coefficients) from the final data measurement is presented in \cite{Slodicka2014}.

\subsection{Thermoelastic model} \label{subsec:TEmodel}
We consider an isotropic homogeneous thermoelastic body which occupies an open and bounded domain $\Omega \subset \mathbb{R}^d,$ where $d\in\mathbb{N}.$ The boundary $\Gamma =\partial\Omega$ is assumed to be Lipschitz continuous. Let $T>0$ be a given final time and set $Q_T = \Omega \times (0,T]$ and $\Sigma_T = \Gamma \times (0,T]$. The interaction between the elastic and thermal behaviour of the material is expressed by the functions
$$
\bm{u} \colon Q_T \to \mathbb{R}^d \quad \text{ and } \quad \theta \colon Q_T\to \mathbb{R},
$$
which denote the displacement vector and the temperature difference $\theta$ from the reference value $T_0>0$, respectively, both taken with reference to the unstressed and undeformed state of the material. 
The thermal conductivity, the mass density and the specific heat of the material are denoted by the positive constants $\kappa, \rho$ and $C_s$, respectively. The positive quantities $\lambda$ and $\mu$ denote the Lam\'{e} coefficients, which can be expressed in terms of the shear modulus $\overline{G}>0$ and the Poisson's ratio $\nu \in (0,1/2)$ of the system as $\lambda = 2(1-2\nu)^{-1}\nu \overline{G}$ and $\mu = \overline{G}.$ The strain and stress tensors are given by
\begin{equation*}
\label{eq:strain-stress}
\epsilon(\bm{u}) = \frac{1}{2}\left(\nabla \bm{u} + (\nabla \bm{u} )^\top\right), \quad \sigma(\bm{u}) =2 \mu \epsilon(\bm{u}) + \lambda \text{Tr}(\epsilon(\bm{u}))\Id,
\end{equation*} 
where $\Id$ is the identity tensor. Notice that $\text{Tr}(\epsilon(\bm{u})) = \nabla \cdot \bm{u}$ and 
\begin{equation*}
    \label{eq:divsigmau}
    \nabla \cdot \sigma(\bm{u}) = \mu \Delta \bm{u} + (\lambda + \mu) \nabla (\nabla\cdot \bm{u}).
\end{equation*}
The coupling, which relates the deformation in shape to the heating and cooling of the material, is expressed by the coupling parameter
$$
\gamma = \alpha_T (3\lambda + 2\mu) = 2\alpha_T \overline{G} \frac{1+\nu}{1-2\nu} >0,
$$
where $\alpha_T>0$ is the coefficient of linear thermal expansion of the material.

The evolution of the displacement $\bm{u},$ the temperature profile $\theta$ and their mutual interaction is governed by the isotropic thermoelastic system of type-{III} (see, e.g. \cite{MunozRivera2002})
\begin{equation} \label{eq:problem}
			\left\{
			\begin{array}{rll}
                \rho \partial_{tt} \bm{u} - \nabla \cdot \sigma(\bm{u}) + \gamma \nabla \theta &= \bm{p} & \quad \text{ in } Q_T \\ 
				\rho C_s \partial_t \theta - \kappa \Delta \theta - (k\ast \Delta \theta) + T_0 \gamma \nabla \cdot \partial_t \bm{u} &= h & \quad \text{ in } Q_T \\
				\bm{u}(\X, t) = \bm{0}, \quad \theta(\X,t) &= 0 &\quad \text{ in } \Sigma_T \\
				\bm{u}(\X,0) = \overline{\bm{u}}_0(\X), \quad \partial_t \bm{u}(\X,0)=  \overline{\bm{u}}_1(\X), \quad \theta(\X,0) &= \overline{\theta}_0(\X) & \quad \text{ in }  \Omega. 
			\end{array}
			\right.
	\end{equation}
The source functions $\bm{p}(\X,t)$ and $h(\X,t)$ on the right-hand side of \eqref{eq:problem}, where $(\X,t) \in Q_T,$ stand for the applied load and heat source respectively, applied to the material at the point $\X\in\Omega$ at time $t.$ The symbol $\ast$ denotes the Laplace convolution of the kernel function $k$ and a function $z$ on $Q_T$ via
$$
(k\ast z)(\X, t)  = \int_0^t k(t-s) z(\X,s) \di s, \quad (\X,t) \in Q_T.
$$
The convolution term $k \ast \Delta \theta$ in the model allows for memory effects, as it ensures a finite speed of heat propagation, as opposed to systems of type-I where $\kappa \neq 0$ and $k=0$. The system \eqref{eq:problem} is referred to as type-{II} if $\kappa =0$ and $k \neq 0,$ whereas $\kappa, k \neq0$ hold in systems of type-{III}.  The kernel function $k \in \Cont^2([0,T])$ is assumed to be strongly positive definite (see e.g.\ \cite{Nohel1976}) and decays to zero for $t\to\infty$, i.e.\
\begin{equation} 
\label{eq:kernelcondition}
    \lim\limits_{t\to\infty}k(t) = 0, \qquad 
    (-1)^j k^{(j)}(t) \geq 0, \quad k'(t) \not\equiv 0, \text{ for all } t \in [0,T],\,  j=0,1,2.
\end{equation} 
Typically, one considers  $k(t) = a\exp(-bt)$ with $ a,b>0,$ see e.g. \cite{FatoriLuedersMunozRivera2003,WangGuo2006}.

\subsection{Inverse source problems} \label{subsec:isps}
The following inverse source problems will be considered in this work. In all of them, the spatially dependent part of the source is to be recovered from a given measurement.
\begin{description}
\item[ISP1] Let 
\begin{equation}
    \label{eq:decomP}
    \bm{p}(\X,t) = g(t) \bm{f}(\X) + \bm{r}(\X,t), \quad (\X,t) \in Q_T,
\end{equation}
where the functions $g \colon [0,T]\to \mathbb{R}$ and $\bm{r} \colon Q_T \to \mathbb{R}^d$ are given. Find the function $\bm{f}\colon \Omega \to \mathbb{R}^d$ from either the final-in-time measurement
\begin{equation}
    \label{eq:measurementfinaltimeU}
		\bm{\xi}_T(\X)=\bm{u}(\X,T), \quad \X\in\Omega,
\end{equation} 
for \textbf{ISP1.1}, or the time-averaged measurement
\begin{equation}\label{eq:measurementtimeaverageU}
	\bm{\chi}_T(\X)=\int_0^T \bm{u}(\X,t) \dt , \quad \X \in \Omega,
\end{equation}
for \textbf{ISP1.2}.

\item[ISP2] Let 
\begin{equation}
    \label{eq:decompH}
    h(\X,t) = g(t) f(\X) + s(\X,t), \quad (\X,t) \in Q_T,
\end{equation}
where the functions $g \colon [0,T] \to \mathbb{R}$ and $s\colon Q_T \to \mathbb{R}$ are given. Find the function $f\colon \Omega \to \mathbb{R}$ from the time-averaged measurement
\begin{equation}
    \label{eq:measurementtimeaveragetheta}
	    \psi_T(\X)=\int_0^T \theta(\X,t) \dt, \quad \X \in \Omega. 
	\end{equation} 
\end{description}
This paper investigates the numerical reconstruction of space-dependent parts of $\bm{p}$ and $h$ under final time and/or time-average measurements taken from $\bm{u}(\X,t)$ and/or $\theta(\X,t).$ The identifiability of the problems considered here was addressed in \cite{fmkvb2022} and is briefly recalled in \Cref{sec:prelim}.
Compared to the work done in \cite{VanBockstal2017b,VanBockstal2014} for \textbf{ISP1.1}, no additional damping term $\bm{g}(\partial_t \bm{u})$ is considered in the hyperbolic equation. The limitation mentioned in \cite{VanBockstal2017b} concerning the reconstruction of non-symmetric sources will be resolved, see \Cref{rem:isp11nonsymetric}. Moreover, our results will also be based on the minimisation of the cost functional combined with the Sobolev gradient method. For conciseness, we have collected the results for \textbf{ISP1.1}, \textbf{ISP1.2} and \textbf{ISP2} into a single article since the numerical schemes to retrieve the missing sources are similar and so only the adaptations need to be explained. 

\subsection{Structure of the article}\label{subsec:structure}
The organisation of this article is as follows. We have formulated the inverse problems and the available measurements in \Cref{subsec:isps}. Some relevant results regarding the weak formulation, the well-posedness of the direct problem, the time discretisation, and the uniqueness of a solution to the ISPs are collected in \Cref{sec:prelim}, where notations will be fixed as well. The iterative scheme based on the Landweber method is described and discussed in \Cref{sec:landweber}, whilst in \Cref{sec:gradientmethods} several gradient methods are investigated. The proposed procedures to reconstruct the unknown source function are tested on a numerical example in \Cref{sec:numerical}.


\section{Preliminaries}\label{sec:prelim}
\subsection{Notations} \label{subsec:notations}
The classical $\Leb^2$-inner product for scalar functions is denoted by $\scal{\cdot}{\cdot}$ and the corresponding norm by $\nrm{\cdot}.$ Let $X$ be a Banach space with norm $\nrm{\cdot}_X$. For $p\geq 1,$ and $p \neq \infty,$ the space $\lpkIX{p}{X}$ is the space of all measurable functions $u \colon (0,T) \to X$ such that
$$
\nrm{u}^p_{\lpkIX{p}{X}} := \int_0^T \nrm{u(t)}_X^p\dt < \infty.
$$
The space $\cIX{X}$ consists of all continuous functions $u \colon [0,T] \to X$ such that 
$$
\nrm{u}_{\cIX{X}} := \max\limits_{t\in[0,T]}\nrm{u(t)}_X < \infty. 
$$
The space $\Hi^k\left((0,T),X\right)$ consists of all functions $u \colon (0,T) \to X$ such that the weak derivative of $u$ with respect to $t$ up to order $k$ exists and
$$
\nrm{u}^2_{\Hi^k\left((0,T),X\right)}:= \int_0^T \left( \sum_{i=0}^k  \nrm{u^{(i)}(t)}^2_X\right) \dt < \infty. 
$$
Similar notations with bold letters will be used for vector-valued functions. We set $\Lp{2} = (\lp{2})^d$ with  inner product given by $\scal{\bm{u}}{\bm{v}} = \sum_{j=1}^d \scal{{u}_j}{{v}_j}$ where $\bm{u} = ({u}_1, \dots, {u}_d) \in \Lp{2}$ and $\bm{v} = ({v}_1,\dots,{v}_d) \in \Lp{2}.$ The Sobolev space $\Hk{1}$ is the space of vector fields $\bm{u} \in \Lp{2}$ with $\nabla \bm{u} \in (\lp{2})^{d\times d}.$ The following norm is used
$$
\nrm{\bm{u}}_{\Hk{1}}^2 = \nrm{\bm{u}}^2_{\Lp{2}} + \nrm{\nabla\bm{u}}^2_{(\lp{2})^{d\times d}},
$$
with
$$
 \nrm{\bm{u}}^2_{\Lp{2}} = \sum_{i=1}^d \nrm{{u}_j}^2 \quad \text{ and }  \quad
 \nrm{\nabla \bm{u}}^2_{(\lp{2})^{d\times d}} 
 = \sum_{i,j=1}^d \nrm{\partial_{x_i}{u}_j}^2.
 $$
Furthermore, we will also denote the norm in $\Lp{2}$ and $(\lp{2})^{d\times d}$ by $\nrm{\cdot},$ where the meaning will be clear from the context.

\subsection{Weak formulation} \label{subsec:weakformulation} 
Since we consider Dirichlet boundary conditions for both $\bm{u}$ and $\theta$ in \eqref{eq:problem}, the variational formulation of the system of equations reads as follows: 
Find  $(\bm{u}(t), \theta(t)) \in \Hko{1} \times \hko{1}$ such that for all $\bm{\varphi} \in \Hko{1}$ and $\psi \in \hko{1}$ and a.e. in $t \in (0,T)$ it holds
\begin{multline}
    \label{eq:weaku}
    \rho \scal{\partial_{tt} \bm{u}(t)}{\bm{\varphi}} + \mu \scal{\nabla \bm{u}(t)}{\nabla \bm{\varphi}} + (\lambda + \mu) \scal{(\nabla \cdot \bm{u})(t)}{\nabla \cdot \bm{\varphi}}\\ \ + \gamma \scal{\nabla \theta(t)}{\bm{\varphi}} = \scal{\bm{p}(t)}{\bm{\varphi}},
\end{multline}
\begin{multline}
    \label{eq:weaktheta}
    \rho C_s \scal{\partial_t \theta(t)}{\psi} + \kappa \scal{\nabla \theta(t)}{\nabla \psi} + \scal{(k\ast \nabla \theta)(t)}{\nabla \psi} \\- T_0\gamma \scal{\partial_t \bm{u}(t)}{\nabla\psi}= \scal{h(t)}{\psi}.
\end{multline}
A pair of functions $(\bm{u}, \theta)$ satisfying the weak formulation \eqref{eq:weaku}-\eqref{eq:weaktheta} will be called a weak solution to the problem \eqref{eq:problem}.

The following technical well-posedness result of \eqref{eq:weaku}-\eqref{eq:weaktheta} was reported in \cite[Theorem~4.1]{VanBockstal2017b}.

\begin{lemma}[Well-posedness forward problem \cite{VanBockstal2017b}]~\label{lem:forward_problem}
Let $k \in \Cont^2([0,T])$ fulfil the conditions in \eqref{eq:kernelcondition}.
\begin{enumerate}[label=(\roman*),wide,labelindent=0pt]
\item Assume that  $\bm{p} \in \lpkIX{2}{\Lp{2}}$ and $h \in \lpkIX{2}{\lp{2}}$ and that the initial data satisfy $\overline{\bm{u}}_0 \in \Hko{1}, \, \overline{\bm{u}}_1 \in \Lp{2}$ and $\overline{\theta}_0\in\lp{2}.$ Then, the problem \eqref{eq:problem} has a unique weak solution $(\bm{u},\theta)$ obeying
    \begin{align*}
        \bm{u} &\in \Cont\left([0,T], \Lp{2}\right) \cap \lpkIX{2}{\Hko{1}}, \quad   \partial_t \bm{u} \in \lpkIX{2}{\Lp{2}}, \\
        \partial_{tt} \bm{u} &\in \lpkIX{2}{\Hko{1}^\ast}, \\
        \theta &\in \Cont\left([0,T], \lp{2}\right) \cap \lpkIX{2}{\hko{1}}, \quad \partial_t \theta \in \lpkIX{2}{\hko{1}^\ast}.
    \end{align*}
    The following a priori estimate holds (for a positive constant $C(T)>0$)
    \begin{multline}
\label{eq:estimate under (i)}
\max\limits_{t\in [0,T]} \left\{ \nrm{\bm{u}(t)}^2 + \nrm{\int_0^t \bm{u}(s)\di s}^2_{\Hko{1}} + \nrm{\int_0^t \theta(s)\ds}^2_{\hko{1}} \right\}  \\
\leq C \left(\nrm{\overline{\bm{u}}_0}^2 +  \nrm{\bm{p}}^2_{\lpkIX{2}{\Lp{2}}} + \nrm{h}^2_{\lpIlp}\right).
\end{multline} 

    \item Assume that  $\bm{p} \in \HHi^1\left((0,T), \Lp{2}\right)$ and $h \in \Hi^1\left((0,T), \lp{2}\right)$ and that the initial data satisfy $\overline{\bm{u}}_0 \in \HHi^2(\Omega) \cap \Hko{1},\, \overline{\bm{u}}_1\in\Hko{1} $ and $\overline{\theta}_0 \in \hko{1}.$ Let the PDEs in \eqref{eq:problem} be fulfilled for $t=0.$ Then, the problem \eqref{eq:problem} has a unique weak solution $(\bm{u}, \theta)$ obeying
    \begin{align*}
        \bm{u} & \in \Cont\left([0,T],\Hko{1}\right), \quad  \partial_t \bm{u} \in \Cont\left([0,T], \Lp{2}\right) \cap \lpkIX{2}{\Hko{1}}, \qquad \qquad\\
        \partial_{tt} \bm{u} & \in \lpkIX{2}{\lp{2}}, \\
        \theta & \in \Cont\left([0,T], \hko{1} \right), \quad \partial_t \theta \in \lpkIX{2}{\hko{1}}.
    \end{align*}
    The following a priori estimate holds (for a positive constant $C(T)>0$)
    \begin{multline}
    \label{eq:estimate under (ii)}
\max\limits_{t \in[0,T]} \left\{\nrm{\partial_t \bm{u}(t)}^2 + \nrm{\bm{u}(t)}^2_{\Hko{1}} + \nrm{\theta(t)}^2  \right\} + \int_0^T \nrm{\nabla \theta(t)}^2\dt \\+ \int_0^T \nrm{(k\ast \nabla \theta)(t)}^2\dt  \leq C\left(\nrm{\bm{p}}^2_{\lpkIX{2}{\Lp{2}}} + \nrm{h}^2_{\lpkIX{2}{\lp{2}}}\right. \\+ \left.  \nrm{\overline{\bm{u}}_0}^2_{\Hko{1}} + \nrm{\overline{\bm{u}}_1}^2 + \nrm{\overline{\theta}_0}^2   \right). 
\end{multline}
\end{enumerate}
\end{lemma}

\subsection{Time discretisation} \label{subsec:timediscretisation}
Let $n_t\in\mathbb{N}$ and divide the time interval $[0,T]$ into $n_t$ subintervals, $[t_{i-1}, t_i], $ with $i=1,\dots, n_t.$ We will consider an equidistant time partition by setting $\abs{t_i - t_{i-1}} = \tau$ for all $i=1,\ldots,n_t,$ so $\tau = \frac{T}{n_t}.$ For a function $z(\X,t),$ we denote its approximation at $t=t_i$ by $z_i(\X).$ Moreover, the backward Euler differences are used to discretise the time derivatives
\begin{equation}
    \label{eq:backwardeuler}
    \begin{array}{rl}
    \partial_t z(\X,t_i) &\approx \delta z_i = \frac{1}{\tau} \left(z_i - z_{i-1}\right), \\ \partial_{tt} z(\X,t_i) &\approx \delta^2 z_i = \frac{1}{\tau} \left( \frac{z_i - z_{i-1}}{\tau} - \delta z_{i-1}\right) .
    \end{array} 
\end{equation} 
Henceforth, we use $\bm{u}_{0} = \overline{\bm{u}}_0, \delta \bm{u}_0 = \overline{\bm{u}}_1,$ and $\theta_0 = \overline{\theta}_0$ according to the initial conditions in \eqref{eq:problem}.  
The discretised weak formulation at $t=t_i$ now reads as follows: For $i=1,\dots,n_t,$ find $(\bm{u}_i, \theta_i) \in \Hko{1} \times \hko{1}$ such that
\begin{multline}
     \rho \scal{\bm{u}_i}{\bm{\varphi}} + \tau^2 \scal{\nabla \bm{u}_i}{\nabla \bm{\varphi}} + \tau^2 (\lambda + \mu) \scal{\nabla \cdot \bm{u}_i}{\nabla \cdot \bm{\varphi}} + \tau^2\gamma \scal{\nabla \theta_i}{\bm{\varphi}} \\\qquad = \tau^2 \scal{\bm{p}_i}{\bm{\varphi}} + \rho \scal{\bm{u}_{i-1} + \tau \delta \bm{u}_{i-1}}{\bm{\varphi}}\label{eq:weakudisc}
\end{multline} 
and
\begin{multline} 
    \rho C_s \scal{\theta_i}{\psi} + \tau \kappa \scal{\nabla \theta}{\nabla \psi} + \tau \scal{\left(k\ast \nabla \theta\right)_i}{\nabla \psi} - T_0 \gamma \scal{\bm{u}_i}{\nabla \psi} \\ \qquad = \tau \scal{h_i}{\psi} + \rho C_s \scal{\theta_{i-1}}{\psi} - T_0 \gamma \scal{\bm{u}_{i-1}}{\nabla \psi}
    \label{eq:weakthetadisc}
\end{multline} 
for all $\bm{\varphi} \in \Hko{1}$ and $\psi \in \hko{1}.$ 

For $k\in \Cont([0,T]),$ the convolution term associated with the operator 
\[
K : \Leb^2(0,T)\rightarrow \Leb^2(0,T) : z \mapsto  k \ast z 
\]
is discretised in the following way. We set $(k\ast z)_0  = (k\ast z)(0)  = 0$ and 
$$
(k\ast z)_i = \sum_{j=1}^i k(t_i -t_j) z_j \tau \approx (k\ast z)(t_i) = \int_0^{t_i} k(t_i - s) z(\X,s)\ds
$$
for $i=1,\dots, n_t.$
The (adjoint) operator $K^*:\Leb^2 (0,T) \rightarrow \Leb^2 (0,T)$ defined by
\begin{equation} 
\label{eq:Kast}
K^\ast(z)(s) := \int_s^T k(t-s)z(t)\,\mathrm{d}t, \quad s \in [0,T],
\end{equation} 
is discretised analogously. We set $K^\ast(z)(T) = (k\circledast z)_{n_t} = 0$ and 
$$
(k\circledast z)_i = \sum_{j=i}^{n_t-1} k(t_j - t_i) z_j \tau \approx K^\ast(z)(t_i) = \int_{t_i}^T k(t-t_i) z(\X,t)\,\mathrm{d}t
$$
for $i = n_t- 1, \dots, 1, 0. $

\subsection{Uniqueness results for the ISPs}\label{subsec:uniqueness}
We will consider the spaces
\begin{equation} 
\label{eq:classesforg}
\begin{array}{rl}
    \mathcal{C}_{0} &:= \left\{ g \in \Cont([0,T]) \mid 0 \notin g([0,T])\right\}, \\
    \mathcal{C}_{1} &:= \left\{ g \in \Cont^1([0,T]) \mid \partial_t(g(t)^2)>0,  \, \forall t \in [0,T] \right\},\\
    \mathcal{C}_{1,\ast} &:= \left\{g \in \Cont^1([0,T]) \cap \mathcal{C}_{0} \mid \partial_t(g(t)^2) \geq  0, \, \forall t \in [0,T]\right\}.
\end{array}
\end{equation}
The uniqueness results corresponding to the inverse problems are recorded below. Their main distinction lies in which set of assumptions in \Cref{lem:forward_problem} are used, and to which class of \eqref{eq:classesforg} the known function $g$ belongs. The following theorems collect the results from \cite[Theorems~2.2,~2.7,~3.1]{fmkvb2022} dealing respectively with \textbf{ISP1.1}, \textbf{ISP1.2} and \textbf{ISP2}, see also \cite{MaesVanBockstalMPiPA,VanBockstalMWCAPDE2023} for related results.
\begin{theorem}[Uniqueness \textbf{ISP1.1}]\label{thm:isp11}
    Let the conditions of \Cref{lem:forward_problem}(ii) be fulfilled. Assume that $g \in \mathcal{C}_{1}$. Then, given  $\bm{\xi}_T \in \Lp{2},$ there exists at most one triple
    \[
    (\bm{u}, \theta, \bm{f}) \in \Cont\left([0,T], \Hko{1}\right) \times \Cont\left([0,T], \hko{1}\right) \times \Lp{2},
    \]
    that solves the problem \eqref{eq:problem} and for which the final-in-time measurement \eqref{eq:measurementfinaltimeU} holds.
\end{theorem} 
\begin{theorem}[Uniqueness \textbf{ISP1.2}]\label{thm:isp12}
    Let the conditions of \Cref{lem:forward_problem}(i) be fulfilled. Assume that $g \in \mathcal{C}_0$. Then, given $\bm{\chi}_T \in \Lp{2},$ there exists at most one triple
    \[
    (\bm{u}, \theta, \bm{f}) \in \Cont\left([0,T], \Hko{1}\right) \times \Cont\left([0,T], \hko{1}\right) \times \Lp{2},
    \]
    that solves the problem \eqref{eq:problem} and for which the time-averaged measurement \eqref{eq:measurementtimeaverageU} holds.
\end{theorem}
\begin{theorem}[Uniqueness \textbf{ISP2}] \label{thm:isp2}
    Let the conditions of \Cref{lem:forward_problem}(ii) be fulfilled. Assume that $g \in \mathcal{C}_{1,\ast}.$	Then, given $\psi_T \in \lp{2}$, there exists at most one triple
    \[
    (\bm{u}, \theta, f) \in \Cont\left([0,T], \Hko{1}\right) \times \Cont\left([0,T], \hko{1}\right) \times \lp{2},
    \]
    that solves the problem \eqref{eq:problem} and for which the time-averaged measurement \eqref{eq:measurementtimeaveragetheta} holds.
\end{theorem}

\section{Landweber scheme}\label{sec:landweber}
In this section, we describe the Landweber-Friedman iteration, first applied to the problem \textbf{ISP1.1} of finding $\bm{f}(\X)$ from the given final-in-time measurement $\bm{u}(\X,T) = \bm{\xi}_T(\X)\in\Lp{2}$ under the decomposition \eqref{eq:decomP} for $\bm{p}$ and knowledge of the heat source $h.$ According to \Cref{thm:isp11}, we will assume that $g \in \mathcal{C}_1$ and the assumptions of \Cref{lem:forward_problem}(ii) are met for the other given data, which guarantees the uniqueness of a solution to the problem. At the end of this section, we provide the adaptions for the \textbf{ISP1.2} and \textbf{ISP2} cases.

\subsection{Description of the method for \textbf{ISP1.1}}\label{subsec:landweber:discription}
By linearity, we construct two subproblems,
\begin{equation}
    \label{eq:ISP11problemP1}
    \tag{$P_1^{1}$}
    \left\{
    \begin{array}{rll}
    \rho \partial_{tt} \bm{u} - \nabla \cdot \sigma(\bm{u}) + \gamma \nabla \theta &= g(t)\bm{f}(\X)  & \quad \text{ in } Q_T \\
        \rho C_s \partial_t \theta - \kappa \Delta \theta - k\ast \Delta \theta + T_0 \gamma \nabla \cdot \partial_t \bm{u} &= 0&  \quad  \text{ in } Q_T \\
        \bm{u}(\X,t) = \bm{0}, \quad \theta(\X,t) &=0 & \quad  \text{ in } \Sigma_T \\
        \bm{u}(\X,0) = \bm{0}, \quad \partial_t \bm{u}(\X,0) =\bm{0}, \quad \theta(\X,0) &= 0 & \quad \text{ in } \Omega,
    \end{array} 
    \right.
\end{equation}
which captures the unknown function $\bm{f}(\X),$ and 
\begin{equation}
    \label{eq:ISP11problemP2}
    \tag{$P_2^{1}$}
    \left\{
    \begin{array}{rll}
    \rho \partial_{tt} \bm{u}   - \nabla \cdot \sigma(\bm{u}) + \gamma \nabla \theta &= \bm{r}(\X,t) & \quad  \text{ in } Q_T \\
        \rho C_s \partial_t \theta - \kappa \Delta \theta - k\ast \Delta \theta + T_0 \gamma \nabla \cdot \partial_t \bm{u} &= h(\X,t) & \quad  \text{ in } Q_T \\
        \bm{u}(\X,t) = \bm{0}, \quad \theta(\X,t) &=0 & \quad  \text{ in } \Sigma_T \\
        \bm{u}(\X,0) = \overline{\bm{u}}_0(\X), \,\partial_t \bm{u}(\X,0)=  \overline{\bm{u}}_1(\X), \, \theta(\X,0) &= \overline{\theta}_0(\X) & \quad \text { in } \Omega
    \end{array} 
    \right.
\end{equation}
which captures all the known data and can be solved directly. Note that problem \eqref{eq:ISP11problemP1} is merely a special case of \eqref{eq:ISP11problemP2} and that, by linearity, the solution $(\bm{u}, \theta, \bm{f})$ to the inverse problem is the sum of the solutions to the subproblems, i.e.\
$$
\bm{u} = \bm{u}_\star + \widetilde{\bm{u}}, \quad \theta = \theta_\star + \widetilde{\theta},
$$
where $(\bm{u}_\star, \theta_\star)$ is the solution to \eqref{eq:ISP11problemP2} and $(\widetilde{\bm{u}}, \widetilde{\theta}, \bm{f})$ is the  solution to \eqref{eq:ISP11problemP1}.  

The procedure for the reconstruction of the load source is as follows. First, let $(\bm{u}_\star, \theta_\star)$ solve \eqref{eq:ISP11problemP2} and compute the part of the measurement $\bm{\xi}_T$ that is captured by this part of the solution, i.e.\
$$
\bm{\xi}_{T,\star}(\X) = \bm{u}_\star(\X,T), \quad \X \in \Omega. 
$$
The remainder of the measurement
$$
\bm{\Xi}_T(\X) = \bm{\xi}_T(\X) - \bm{\xi}_{T,\star}(\X)
$$
then needs to be captured by the solution of \eqref{eq:ISP11problemP1}. In this regard, the problem \eqref{eq:ISP11problemP2} needs to be computed only once, while \eqref{eq:ISP11problemP1} will be used iteratively. 

Consider the measurement operator
\begin{equation}
    \label{eq:ISP11measurementmapM}
    M_T \colon \Lp{2} \to \Lp{2} \colon \quad \bm{f} \mapsto \widetilde{\bm{u}}(\cdot,T ; \bm{f}),
\end{equation}
where $(\widetilde{\bm{u}}(\X,t; \bm{f}), \widetilde{\theta}(\X,t; \bm{f}))$ forms a solution to \eqref{eq:ISP11problemP1} with $\bm{f}$ as the right-hand side. The extra argument in the notation is to emphasise the dependence of the obtained solution on the source function.  The following lemma shows that the inverse problem under consideration is ill-posed.

\begin{lemma}\label{lem:MTcompactlinearL2L2}
    The map $M_T \colon \Lp{2} \to \Lp{2}$ defined in \eqref{eq:ISP11measurementmapM} is a compact linear operator.
\end{lemma}

\begin{proof}
    The linearity of $M_T$ is checked by a standard argument. Let $\bm{f}_1, \bm{f}_2 \in \Lp{2}$ and denote with $(\bm{u}_i, \theta_i)$ the solution to \eqref{eq:ISP11problemP1} corresponding with $\bm{f}_i, i =1,2.$  Then, for arbitrary $\eta \in \mathbb{R}, (\bm{u}, \theta)=(\bm{u}_1 + \eta \bm{u}_2, \theta_1 + \eta \theta_2)$ is the solution to \eqref{eq:ISP11problemP1} corresponding to $\bm{f}_1 + \eta\bm{f}_2,$ i.e. 
    \begin{align*}
    M_T(\bm{f}_1 + \eta \bm{f}_2) &= \bm{u}(\cdot,T;\bm{f}_1 + \eta\bm{f}_2)\\ 
    &=\bm{u}_1(\cdot,T;\bm{f}_1) + \eta\bm{u}_2(\cdot,T;\bm{f}_2)=M_T(\bm{f}_1) + \eta M_T(\bm{f}_2),
    \end{align*}
    for all $\eta \in \mathbb{R}.$
    Moreover, the a priori estimate \eqref{eq:estimate under (ii)} shows that $M_T$ is bounded from $\Lp{2}$ with values in $\Hko{1}\subset\Lp{2}.$ The Rellich-Kondrachov theorem gives the compact embedding from $\Hko{1}$ into $\Lp{2},$ hence $M_T,$ seen as a map from $\Lp{2}$ to itself, is compact as the composition of a compact and bounded operator.\qedhere
\end{proof}

Notice that the remainder of the measurement for the exact source $\bm{f}$ satisfies
\begin{equation}
    \label{eq:isp11operatereq}
    \bm{\Xi}_T = M_T(\bm{f}).
\end{equation}
Therefore, the objective in \textbf{ISP1.1} is to invert the operator equation \eqref{eq:isp11operatereq}, i.e. given $\bm{\Xi}_T$, solve \eqref{eq:isp11operatereq} for $\bm{f}.$ This corresponds to solving the fixed point equation
\begin{equation}
    \label{eq:isp11fixedpointeq}
    \bm{f} = \bm{f} - \alpha M_T\left(M_T(\bm{f}) - \bm{\Xi}_T\right), \quad \alpha >0,
\end{equation}
where $\alpha$ is a relaxation parameter. Applying the method of successive approximations yields the following procedure
\begin{equation} 
\label{eq:isp11successiveapproxf}
\bm{f}_k = \bm{f}_{k-1} - \alpha M_T\left(M_T(\bm{f}_{k-1}) - \bm{\Xi}_T\right), \quad k = 1,2,\dots,
\end{equation} 
with initial guess  $\bm{f}_0 \colon \overline{\Omega} \to \mathbb{R}$ satisfying $\bm{f}_0\in\Lp{2}$. 

\begin{remark}
Note that since $\text{im}(M_T) \subseteq \Hko{1},$ 
the values of the iterates $\bm{f}_k$ in \eqref{eq:isp11successiveapproxf} at the boundary remain fixed and equal to $\left.\bm{f}_k \right|_{\partial\Omega}=\left.\bm{f}_0 \right|_{\partial\Omega}$ at each iteration step $k$. 
\end{remark}

\subsection{Algorithm} \label{subsec:landweber:algorithm}
The above discussion leads to the following algorithm for the reconstruction of $(\bm{u}, \theta, \bm{f})$ for \textbf{ISP1.1} based on \eqref{eq:isp11successiveapproxf} and given a choice of $\alpha>0.$ The approach is similar to the algorithm discussed in \cite{VanBockstal2014,VanBockstal2017b}. 
\begin{enumerate}
    \item Let $(\bm{u}_\star, \theta_\star)$ be the solution of \eqref{eq:ISP11problemP2} and determine $\bm{\Xi}_T =  \bm{\xi}_T - \bm{\xi}_{T,\star}.$

    \item Make a guess $\bm{f}_0,$ 
    and solve \eqref{eq:ISP11problemP1}, let $(\widetilde{\bm{u}}_0, \widetilde{\theta}_0)=(\bm{u}(\cdot,\cdot; \bm{f}_0), \theta(\cdot,\cdot; \bm{f}_0))$ be its solution.

    \item Assume $\bm{f}_k$ and $(\widetilde{\bm{u}}_k, \widetilde{\theta}_k)$ have been constructed. Let  $\bm{F}_k = \widetilde{\bm{u}}_k(\cdot,T) - \bm{\Xi}_T$ and solve \eqref{eq:ISP11problemP1} with the choice of $\bm{f} = \bm{F}_k$ in the load source, let $(\bm{w}_k, \eta_k):=(\bm{u}(\cdot,\cdot;\bm{F}_k),\theta(\cdot,\cdot; \bm{F}_k))$ be the solution.

    \item Update $\bm{f}_k$ via the rule
    $$
    \bm{f}_{k+1} = \bm{f}_k - \alpha \bm{w}_k(\cdot, T),
    $$

    and solve \eqref{eq:ISP11problemP1} using $\bm{f}_{k+1}$ to find the new solution pair $(\widetilde{\bm{u}}_{k+1}, \widetilde{\theta}_{k+1})$. 

    \item Repeat the above steps until a stopping criterion is met, say at step $K.$ The approximating solution to the inverse problem is then given by the triple $(\bm{u}_\star + \widetilde{\bm{u}}_K, \theta_\star + \widetilde{\theta}_K, \bm{f}_K).$
\end{enumerate}

\subsection{Convergence and stopping criterion}\label{subsec:landweber:convergencecriteria}
Due to the well-posedness result in \Cref{lem:forward_problem}, both problems \eqref{eq:ISP11problemP1} and \eqref{eq:ISP11problemP2} are well-posed. The convergence of the algorithm in \Cref{subsec:landweber:algorithm} is shown in the next theorem if the relaxation parameter $\alpha>0$ is suitably bounded. The crucial lemma is found in \cite[Theorem~3]{slodicka2010} and gives a general version of the proof of convergence for non self-adjoint operators in Landweber iterations, see also \cite[Theorem~6.1]{Engl1996}. More precisely, we show that the iterations \eqref{eq:isp11successiveapproxf} converge in $\Lp{2}$ when the (positive) relaxation parameter $\alpha$ is bounded by the inverse of the square of the operator norm of $M_T.$ The proof is as in \cite[Theorem 3.3]{VanBockstal2014} and \cite[Theorem 4.3]{VanBockstal2017b} and is added here in order to be self-contained and to see the analogy between the considered inverse problems in this work.

\begin{theorem}\label{thm:convergenceLandweberISP11}
Assume that the conditions of \Cref{lem:forward_problem}(ii) are fulfilled and $g \in \mathcal{C}_1$. Assume that  $0< \alpha < \nrm{M_T}_{{\mathcal L}(\Lp{2},\Lp{2})}^{-2}.$  Let $(\bm{u}, \theta, \bm{f}) = (\bm{u}_\star + \bm{v}, \theta_\star + \zeta, \bm{f})$ be the unique solution to the inverse problem \textbf{ISP1.1}, where $(\bm{u}_\star, \theta_\star)$ solves \eqref{eq:ISP11problemP2} and $(\bm{v}, \zeta, \bm{f})$ solves \eqref{eq:ISP11problemP1} under the condition $\bm{v}(\cdot, T) = \bm{\Xi}_T.$ Let  $(\bm{v}_k, \eta_k, \bm{f}_k)$ the $k$-th approximation according to the iterative procedure in \Cref{subsec:landweber:algorithm}. Then, it holds that
\begin{equation}
\label{eq:convergencevkzetak}
        \lim\limits_{k\to \infty}\left( \nrm{\bm{v} - \bm{v}_k}_{\cIX{\Hko{1}}} + \nrm{\zeta-\zeta_k}_{\cIX{\hko{1}}}  \right)  = 0
    \end{equation}
and 
\begin{equation}
\label{eq:convergencevtk}
    \lim\limits_{k\to \infty} \left(\nrm{\partial_t \bm{v} - \partial_t \bm{v}_k}_{\lpkIX{2}{\Lp{2}}}  \right) = 0
\end{equation}
for every choice of initial guess $f_0 \in \Lp{2}.$
\end{theorem}
\begin{proof}

    Using the linearity of $M_T$ and the fact that $M_T(\bm{f}) = \bm{\Xi}_T,$ we write
    \begin{align*}
        \bm{f}_k &= \bm{f}_{k-1} - \alpha M_T(M_T(\bm{f}_{k-1}) - \bm{\Xi}_T) \\
        &= \bm{f}_{k-1} - \alpha M_T(M_T(\bm{f}_{k-1} - \bm{f}))
    \end{align*}
    hence, denoting with $I$ the identity operator on $\lp{2}$, we get 
    $$
    \bm{f}_k - \bm{f} = (I-\alpha M_T^2)(\bm{f}_{k-1} -\bm{f}),
    $$
    Upon taking norms, we see that the condition $0 < \alpha < \nrm{M_T}_{{\mathcal L}(\Lp{2},\Lp{2})}^{-2}$ implies convergence of $\bm{f}_k$ to $\bm{f}$ in $\Lp{2}$ for arbitrary $\bm{f}_0 \in \Lp{2},$ by applying \cite[Theorem~3]{slodicka2010}.
    Inequality \eqref{eq:estimate under (ii)} implies that we have convergence of $\bm{v}_k \to \bm{v}$ in $\cIX{\Hko{1}}$ and $\zeta_k \to \zeta$ in $\cIX{\hko{1}}$ and, moreover, $\partial_t \bm{v}_k \to \partial_t \bm{v}$ in $\lpkIX{2}{\Lp{2}}.$ This observation completes the proof.
\end{proof}

To simulate real measurements, a randomly generated noise is added to the measurement in the numerical experiments (see  \Cref{sec:numerical}). That is, instead of working with $\bm{\xi}_T$, the available measurement is  $\bm{\xi}_T^e \in \Lp{2}$ and it is assumed that it satisfies
\begin{equation}
    \label{eq:eq:perturbedchiTnorm}
\nrm{\bm{\xi}_T - \bm{\xi}_T^e} \leq e,
\end{equation}
where $e\geq 0$ denotes the noise level. Therefore, the remainder of the measurement $\bm{\Xi}_T$ is perturbed as well to $\bm{\Xi}_T^e$ and it holds that
\begin{equation} 
\label{eq:perturbedChiTnorm}
\nrm{\bm{\Xi}_T - \bm{\Xi}_T^e} \leq e.
\end{equation}
The proposed algorithm can be recast in this noisy case. The constructed approximations $(\widetilde{\bm{u}}_k ^e, \widetilde{\theta}_k^e, \bm{f}_k^e)$ are only influenced by the noise via the measurement $\bm{\Xi}_T^e$ as there is no noise considered in the initial data. Let us denote with $E_{k, \bm{\Xi}_T}$ the absolute $\Lp{2}$-error between the noisy measurement $\bm{\Xi}_T^e$ and the $k$-th approximated solution evaluated at final time $\widetilde{\bm{u}}_k^e(\cdot, T),$ i.e.\ 
\begin{equation} 
\label{eq:stoppingisp11}
E_{k,\bm{\Xi}_T} = \nrm{ \bm{\Xi}_T^{e} - \widetilde{\bm{u}}_k^e(\cdot,T)}.
\end{equation} 
Then, given the noise level $e,$ Mozorov's discrepancy principle is used to device the stopping criterion in order to stop the iterations at the first index $K=k$ for  which \[
E_{k, \bm{\Xi}_T} \leq r e,\] 
where $r>1$ is fixed, see e.g.\ \cite[Eq~(6.7)]{Engl1996}.
If the noise level $e$ is unknown, one can use a heuristic stopping device as proposed in \cite{Real2024}. This approach ensures a fully data-driven approach, as no information is needed on the precision and accuracy of the tool to produce the measurement.

\subsection{Adaptations for \textbf{ISP1.2}}\label{subsec:adaptations12land}
The analysis above carries over to the problem of finding $\bm{f}$ in \textbf{ISP1.2}, given the measurement $\bm{\chi}_T \in\Lp{2}$. We now assume that the conditions of \Cref{lem:forward_problem}(i) are met and that the temporal part $g \in \mathcal{C}_0.$ Then, \Cref{thm:isp12}  ensures the uniqueness of a solution to the inverse problem. The same form of the subproblems \eqref{eq:ISP11problemP1} and \eqref{eq:ISP11problemP2} can be used in this case. However, letting $(\bm{u}_\ast, \theta_\ast)$ be solutions to \eqref{eq:ISP11problemP1}, we now compute the captured portion of the measurement and the remainder as
$$
\bm{\chi}_{T,\ast}(\X) = \int_0^T \bm{u}_\ast(\X,t)\ds, \quad \bm{X}_T(\X) = \bm{\chi}_T(\X) - \bm{\chi}_{T,\ast}(\X).
$$
This remainder $\bm{X}_T$ then needs to be captured by the solution of \eqref{eq:ISP11problemP1}. For it, we now consider the mapping.
\begin{equation} 
\label{eq:ISP12measurementmapN}
N_T \colon \Lp{2} \to  \Lp{2}  \colon \bm{f} \mapsto \int_0^T \widetilde{\bm{u}}(\cdot, t; \bm{f})\dt 
\end{equation}
where $(\widetilde{\bm{u}}(\cdot, \cdot,\bm{f}), \widetilde{\theta}(\cdot,\cdot;\bm{f}))$ is a solution to \eqref{eq:ISP11problemP1} with right-hand side $\bm{f}.$  The corresponding method of successive approximations reads in this case as 
\begin{equation}
    \label{eq:isp12successiveapproxf}
    \bm{f}_k = \bm{f}_{k-1} - \alpha_1 N_T\left(N_T(\bm{f}_{k-1}) - \bm{X}_T\right), \quad \bm{f}_{0} \in \Lp{2},
\end{equation}
with relaxation parameter $\alpha_1>0$ and the corresponding algorithm becomes:
\begin{enumerate}
    \item Let $(\bm{u}_\ast, \theta_\ast)$ be the  solution of \eqref{eq:ISP11problemP2} and determine $\bm{X}_T = \bm{\chi}_T - \bm{\chi}_{T,\ast}.$

    \item Make a guess $\bm{f}_0 \in\Lp{2},$ and solve \eqref{eq:ISP11problemP1}, denote its soltution with  $(\widetilde{\bm{u}}_0, \widetilde{\theta}_0) = (\bm{u}(\cdot,\cdot;\bm{f}_0), \theta(\cdot,\cdot; \bm{f}_0)).$

    \item Assume that $\bm{f}_k$ and $(\widetilde{\bm{u}}_k, \widetilde{\theta}_k)$ have been constructed. Let
    \begin{equation*} \bm{F}_k = \int_0^T \widetilde{\bm{u}}(\cdot,t)\dt  - \bm{X}_T\end{equation*} and solve \eqref{eq:ISP11problemP1} with the choice $\bm{f} = \bm{F}_k$ in the load source, let $(\bm{w}_k, \eta_k) = (\bm{u}(\cdot, \cdot; \bm{F}_k), \theta(\cdot, \cdot; \bm{F}_k)$ be the solution.

    \item Update $\bm{f}_k$ via the rule
    $$
    \bm{f}_{k+1} = \bm{f}_k -  \alpha_1 \int_0^T \bm{w}_k(\cdot, t)\dt,
    $$
    and solve \eqref{eq:ISP11problemP1} using $\bm{f}_{k+1}$ to find the new solution pair $(\widetilde{\bm{u}}_{k+1}, \widetilde{\theta}_{k+1}).$

    \item Repeat the above steps until a stopping criterion is met at step $K.$ The triple $(\bm{u}_\ast + \widetilde{\bm{u}}_K, \theta_\ast + \widetilde{\theta}_K, \bm{f}_K)$ is the approximation to the solution of the inverse problem \textbf{ISP1.2}. 
    
\end{enumerate}
We record the following convergence result.
\begin{theorem}\label{thm:convergenceLandweberISP12}
    Assume that the conditions of \Cref{lem:forward_problem}(i) are fulfilled and $g\in \mathcal{C}_0.$ Assume that  $0 < \alpha_1 < \left\|N_T\right\|^{-2}_{{\mathcal L}(\Lp{2},\Lp{2})}.$ Let $(\bm{u},\theta, \bm{f}) = (\bm{u}_\ast + \bm{v}, \theta_\ast + \zeta, \bm{f})$ be the unique solution to the inverse problem \textbf{ISP1.2}, where $(\bm{u}_\ast,\theta_\ast)$ solves \eqref{eq:ISP11problemP2} and $(\bm{v}, \zeta, \bm{f})$ solves \eqref{eq:ISP11problemP1} under the condition $\int_0^T \bm{v}(\cdot, t) \dt = \bm{X}_T.$ Let $(\bm{v}_k, \theta_k, \bm{f}_k)$ be the $k$-th approximation according to the iterative procedure described above. Then it holds that
    \begin{equation}\label{ip1.2:stopping}
    \lim\limits_{k\to\infty} \left( \nrm{\bm{v}-\bm{v}_k}_{\cIX{\Lp{2}}} + \nrm{\int_0^{t} \left(\zeta(\cdot,s) - \zeta_k(\cdot,s)\right)\ds}_{\cIX{\hko{1}}} \right) = 0 
    \end{equation}
    for every choice of initial guess $\bm{f}_0 \in \Lp{2}.$ If the conditions of \Cref{lem:forward_problem}(ii) are met, then it holds that the limit transitions \eqref{eq:convergencevkzetak}-\eqref{eq:convergencevtk} are satisfied. 
\end{theorem}

The proof is analogous to that of \Cref{thm:convergenceLandweberISP11}, now using \eqref{eq:estimate under (i)}.
As a stopping criterion, we stop the iterations at the first value of $k$ where
\begin{equation}
\label{eq:stoppingisp12}
E_{k,\bm{X}_T} := \nrm{ \bm{X}_T^e - \int_0^T \widetilde{\bm{u}}_k^e(\cdot,t)\dt} \leq re, \quad r>0,
\end{equation} 
in case of noisy data with noise level $e\geq 0.$ Here $\bm{X}_T^e$ is the remainder of the noisy measurement $\bm{\chi}_T^e$ with $\nrm{\bm{\chi}_T - \bm{\chi}_T^e} \leq e.$ 

\subsection{Adaptations for \textbf{ISP2}}

In a similar fashion, we can adapt to the situation in \textbf{ISP2} as follows. The subdivision in the two subproblems 
\begin{equation}
    \label{eq:ISP2problemP1}
    \tag{${P}_1^2$}
    \left\{
    \begin{array}{rll}
            \rho \partial_{tt} \bm{u} - \nabla \cdot \sigma(\bm{u}) + \gamma \nabla \theta &= 0  & \quad \text{ in } Q_T \\
        \rho C_s \partial_t \theta - \kappa \Delta \theta - k\ast \Delta \theta + T_0 \gamma \nabla \cdot \partial_t \bm{u} &= g(t)f(\X)&  \quad  \text{ in } Q_T \\
        \bm{u}(\X,t) = \bm{0}, \quad \theta(\X,t) &=0 & \quad  \text{ in } \Sigma_T \\
        \bm{u}(\X,0) = \bm{0}, \quad \partial_t \bm{u}(\X,0) =\bm{0}, \quad \theta(\X,0) &= 0 & \quad \text{ in } \Omega,
    \end{array} 
    \right.
\end{equation}
and 
\begin{equation}
    \label{eq:ISP2problemP2}
    \tag{${P}_2^2$}
    \left\{
    \begin{array}{rll}
     \rho \partial_{tt} \bm{u} -  \nabla \cdot \sigma(\bm{u}) + \gamma \nabla \theta &= \bm{p}  & \quad \text{ in } Q_T \\ 
        \rho C_s \partial_t \theta - \kappa \Delta \theta - k\ast \Delta \theta + T_0 \gamma \nabla \cdot \partial_t \bm{u} &=  s  & \quad  \text{ in } Q_T \\
        \bm{u}(\X,t) = \bm{0}, \, \theta(\X,t) &=0 & \quad  \text{ in } \Sigma_T \\
        \bm{u}(\X,0) = \overline{\bm{u}}_0(\X), \, \partial_t \bm{u}(\X,0)=  \overline{\bm{u}}_1(\X), \quad \theta(\X,0) &= \overline{\theta}_0(\X) & \quad \text { in } \Omega
    \end{array} 
    \right.
\end{equation}
is now needed. Here \eqref{eq:ISP2problemP2} can be solved based on the available data as a direct problem, and \eqref{eq:ISP2problemP1} is used in the iterative scheme. The associated measurement operator is now given by
\begin{equation} 
\label{eq:ISP2measurementmapP}
P_T \colon \lp{2} \to  \lp{2} \colon f \mapsto \int_0^T \widetilde{\theta}(\cdot,t;f)\dt,
\end{equation} 
where $(\widetilde{\bm{u}}(\cdot,\cdot;f), \widetilde{\theta}(\cdot,\cdot;f))$ is the solution to \eqref{eq:ISP2problemP1} with right-hand side $f.$ The approximations 
$$
f_k = f_{k-1} - \alpha_2 P_T\left(P_T(f_{k-1}) - \Psi_T\right), \quad f_0 \in \lp{2},
$$
with $\alpha_2>0$ as relaxation parameter are now used. Here we denote with $\Psi_T$ the remainder of the measurement $\psi_T$, taking into account the part of which is captured by  the solution $(\bm{u}_\ast, \theta_\ast)$ of \eqref{eq:ISP2problemP2}, i.e.
$$
\psi_{T,\ast}(\X) = \int_0^T \theta_\ast(\X,t)\dt, \quad \Psi_T(\X) = \psi_T(\X) - \psi_{T,\ast}(\X).  
$$
The appropriate stopping criterion is now given by considering
$$
E_{k, \Psi_T} = \nrm{ \Psi_T^e - \int_0^T \widetilde{\theta}_k^e(\cdot, t)\dt} \leq re, \quad r>0,
$$
for the noisy measurement $\psi_T^e  \in \lp{2},$ with $\nrm{\psi_T - \psi_{T}^e}\leq e$ and its remainder $\Psi_T^e = \psi_T^e - \psi_{T,\ast}.$ 
The associated algorithm now proceeds as:

\begin{enumerate}
    \item Let $(\bm{u}_\ast, \theta_\ast)$ be the  solution of \eqref{eq:ISP2problemP2} and determine $\Psi_T = \psi_T - \psi_{T,\ast}.$

    \item Make a guess $f_0 \in\lp{2},$ and solve \eqref{eq:ISP2problemP1}, denote with $(\widetilde{\bm{u}}_0, \widetilde{\theta}_0) = (\bm{u}(\cdot,\cdot;f_0), \theta(\cdot,\cdot; f_0))$ its solution.

    \item Assume that $\bm{f}_k$ and $(\widetilde{\bm{u}}_k, \widetilde{\theta}_k)$ have been constructed. Let \begin{equation*} F_k = \int_0^T \widetilde{\theta}(\cdot,t)\dt  - \Psi_T\end{equation*} and solve \eqref{eq:ISP2problemP1} with the choice $f = F_k$ in the heat source, let $(\bm{w}_k, \eta_k) := (\bm{u}(\cdot, \cdot; F_k), \theta(\cdot, \cdot; F_k)$ be the solution.

    \item Update $f_k$ via the rule
    $$
    f_{k+1} = f_k -  \alpha_2\int_0^T \eta_k(\cdot, t)\dt,
    $$
    and solve \eqref{eq:ISP2problemP1} using $f_{k+1}$ to find the new solution pair $(\widetilde{\bm{u}}_{k+1}, \widetilde{\theta}_{k+1}).$

    \item Repeat the above steps until a stopping criterion is met in step $K.$. The triple $(\bm{u}_\ast + \widetilde{\bm{u}}_K, \theta_\ast + \widetilde{\theta}_K, f_K)$ is the approximate solution to the inverse problem \textbf{ISP2}. 
\end{enumerate}
For completeness, we list the following convergence theorem, its proof is an adaptation of that of \Cref{thm:convergenceLandweberISP11}.
\begin{theorem}\label{thm:convergenceLandweberISP2}
Assume that the conditions of \Cref{lem:forward_problem}(ii) are satisfied and $g \in \mathcal{C}_{1,\ast}.$ Assume that $0 < \alpha_2  < \nrm{P_T}^{-2}_{\mathcal{L} (\lp{2},\lp{2})}.$ Let $(\bm{u},\theta, f) = (\bm{u} + \bm{v}, \theta_\ast + \zeta, f)$ be the unique solution to the inverse problem \textbf{ISP2}, where $(\bm{u}_\ast, \theta_\ast)$ solves \eqref{eq:ISP2problemP2} and $(\bm{v},\zeta, f)$ solves \eqref{eq:ISP2problemP1} under the condition $\int_0^T \zeta(\cdot, t)\dt = \Psi_T.$ Let $(\bm{v}_k, \zeta_k, f_k)$ be the $k$-th approximation according to the iterative procedure described above. Then it holds that
\begin{equation*}
        \lim\limits_{k\to \infty}\left( \nrm{\bm{v} - \bm{v}_k}_{\cIX{\Hko{1}}} + \nrm{\zeta-\zeta_k}_{\cIX{\hko{1}}}  \right)  = 0
    \end{equation*}
and 
\begin{equation*}
    \lim\limits_{k\to \infty} \left(\nrm{\partial_t \bm{v} - \partial_t \bm{v}_k}_{\lpkIX{2}{\Lp{2}}}  \right) = 0
\end{equation*}
for every choice of initial guess $f_0 \in \lp{2}.$ 
\end{theorem}
\section{Minimisation methods} \label{sec:gradientmethods}
In this section, we rewrite the inverse problems as minimisation problems of certain objection functionals, for which  gradient methods can be applied.  The setup and procedure are first described for \textbf{ISP1.1} and afterwards adapted for \textbf{ISP1.2} and \textbf{ISP2}. By linearity, we consider only \eqref{eq:ISP11problemP1} and \eqref{eq:ISP2problemP1} as problems \eqref{eq:ISP11problemP2} and \eqref{eq:ISP2problemP2} are independent of the sought sources and are part of the direct method.

\subsection{Setup} \label{subsec:gradientsetup}
The objective functional for the inverse problem \textbf{ISP1.1} with final-in-time time measurement $\bm{\xi}_T \in \Lp{2}$ with regard to $(\bm{u}(\cdot,\cdot;\bm{f}),\theta(\cdot,\cdot;\bm{f}))$ as the solution to \eqref{eq:ISP11problemP1} in the noise free setting, is given by
\begin{equation}
    \label{eq:JT11}
    \mathcal{J}_\beta(\bm{f}) = \frac{1}{2} \nrm{\bm{u}(\cdot,T;\bm{f}) - \bm{\Xi}_T(\cdot)}^2 + \frac{\beta}{2} \nrm{\bm{f}}^2.
\end{equation}
Here $\beta\geq 0$ is the regularisation parameter in the above Tikhonov functional, which compromises between data fidelity in the first term (which indicates how the solution fits the measurement) and the size of the solution in the second term (which gives a penalty for solutions with large norm). Using the measurement operator \eqref{eq:ISP11measurementmapM}, we can write 
\begin{equation}
    \label{eq:JT112}
\mathcal{J}_\beta(\bm{f}) = \frac{1}{2} \nrm{M_T(\bm{f}) - \bm{\Xi}_T}^2 + \frac{\beta}{2} \nrm{\bm{f}}^2.
\end{equation}
\begin{remark}  \label{rem:datafitandpenalty}
In case $\beta=0$, the minimisation of \eqref{eq:JT11} corresponds to the minimisation of the residual error. If $\beta>0,$ then the optimal value of $\beta$ can be determined by, e.g., the $L$-curve criterion \cite{Engl1996}. A more general version of a cost functional  is given by $\frac{1}{2}\nrm{M_T(\bm{f})-\bm{\Xi}_T}^2 + \frac{\beta}{2}\mathcal{R}(\bm{f})$, where $\mathcal{R}(\bm{f})$ is some penalty term, e.g. $\mathcal{R}(\bm{f})=\nrm{\bm{f}-\bm{g}}^2$ for some $\bm{g}\in\Lp{2}.$ Larger values of $\beta$ favor solutions $\bm{f}$ with small deviation from a given $\bm{g}$ at the cost of possible worse datafit, i.e.\ large $\nrm{M_T(\bm{f}) - \bm{\Xi}_T},$ whereas small values of $\beta$ allow for a better data fit, possibly with a larger penalty.
\end{remark} 
The main goal is to find a minimiser of \eqref{eq:JT11}. The existence and uniqueness of a solution in the absence and presence of noise are discussed in the following two theorems, together with their relation as the noise level $e$ tends to zero.

\begin{theorem}\label{thm:JT11uniqemin}
The minimisation problem of finding
$$
\argmin\limits_{\bm{f}\in\Lp{2}}\mathcal{J}_\beta(\bm{f}), \quad \beta >0
$$
has a unique solution in $\Lp{2}.$ In case $\beta=0,$  a unique solution exists on each nonempty 
bounded closed convex subset $\mathcal{A}\subset \Lp{2}.$
\end{theorem}
\begin{proof}
We start by proving that the functional $\mathcal{J}_\beta$ is strictly convex,  continuous and weakly coercive if $\beta>0.$
    \Cref{lem:MTcompactlinearL2L2} shows the continuity of $M_T$ and hence that of $\mathcal{J}_\beta$ by composition of continuous maps, for any $\beta \geq 0.$ The convexity follows by the linearity of $M_T$ and  the strict convexity of $\nrm{\cdot}^2$ as for distinct $\bm{f}_1, \bm{f}_2 \in \Lp{2}$ and $a \in(0,1),$ we find 
    \begin{align*}
    &\mathcal{J}_\beta(a\bm{f}_1 + (1-a) \bm{f}_2) \\
    &= \frac{1}{2} \nrm{M_T(a\bm{f}_1 + (1-a)\bm{f}_2)- \bm{\Xi}_T}^2  + \frac{\beta}{2} \nrm{ a\bm{f}_1+(1-a) \bm{f}_2}^2 \\
& \leq  \frac{1}{2}\nrm{a(M_T(\bm{f}_1)-\bm{\Xi}_T) + (1-a)(M_T(\bm{f}_2)-\bm{\Xi}_T)}^2 + \frac{\beta}{2} \nrm{a\bm{f}_1 + (1-a)\bm{f}_2}^2 \\
 &  < a\mathcal{J}_\beta(\bm{f}_1) + (1-a) \mathcal{J}_\beta(\bm{f}_2). 
    \end{align*} 
From \eqref{eq:JT11} it follows that $\mathcal{J}_\beta(\bm{f}) \to \infty$ as $\nrm{\bm{f}}\to \infty$ if $\beta >0,$ hence $\mathcal{J}_\beta$ is weakly coercive. The main theorem on convex minimum problems \cite[Theorem~25.E]{Zeidler-IIB} gives the existence and uniqueness of the minimal point of \eqref{eq:JT11}. 
If $\beta=0,$ then the functional $\mathcal{J}_0$ remains strictly convex and continuous on bounded closed convex sets $\mathcal{A}$ and is therefore weakly sequentially lower semicontinuous on $\mathcal{A}.$ The existence and uniqueness of a minimal point then follows from \cite[Theorem~25.C - Corollary~25.15]{Zeidler-IIB}.
\end{proof}

In case of noise with noise level $e$, the measurement $\bm{\Xi}_T$ is perturbed to $\bm{\Xi}_T^e$  with the property that $\bm{\Xi}_T^e \to \bm{\Xi}_T$ in $\Lp{2}$ as $e \to 0$ by \eqref{eq:perturbedChiTnorm}. Therefore, instead of \eqref{eq:JT11}, the functional $\mathcal{J}_\beta^e\colon\Lp{2} \to \mathbb{R}$ for $ \beta\geq 0$ defined by
\begin{equation}
    \label{eq:JTe11}
    \mathcal{J}_\beta^e(\bm{f}) = \frac{1}{2} \nrm{M_T(\bm{f}) - \bm{\Xi}_T^e}^2 + \frac{\beta}{2} \nrm{\bm{f}}^2
\end{equation}
and its corresponding minimisation problem should be considered for noisy data. 
\begin{theorem}\label{thm:JTe11uniquemin}
    The minimisation problem of finding
    $$
    \argmin\limits_{\bm{f}\in\Lp{2}} \mathcal{J}_\beta^e(\bm{f}), \quad \beta>0,\; e\geq 0,
    $$
    has a unique solution  in $\Lp{2}$. In case $\beta=0,$ a unique solution exists on each nonempty bounded closed convex $\mathcal{A}\subset \Lp{2}.$ 
     Moreover, for any $\beta \geq 0,$ let $\{e_n\}_n\subseteq{\mathbb{R}_{\geq0}}$ be a sequence with $e_n \searrow 0$ as $n\to \infty,$ and let $\{\bm{\Xi}_T^{e_n}\}_n$ be the corresponding sequence of noisy measurements in $\Lp{2}$ such that $\bm{\Xi}_T^{e_n}\to \bm{\Xi}_T$ in $\Lp{2}$ as $n\to \infty.$ If the associated sequence $\{\bm{f}_{e_n}\}_n$ in $\Lp{2}$ of minimisers of $\mathcal{J}_\beta^{e_n}$ is uniformly bounded, then $\{\bm{f}_{e_n}\}_n$ converges weakly to a minimiser of $\mathcal{J}_\beta.$ 
\end{theorem}
\begin{proof}
Arguing similarly as in \Cref{thm:JT11uniqemin} yields the existence and uniqueness of the minimiser $\bm{f}_{e_n}$ of $\mathcal{J}_\beta^{e_n}$ for all $n\in\NN.$ Since this sequence $\{\bm{f}_{e_n}\}_n$ is assumed to be uniformly bounded, we have that 
    $\bm{f}_{e_n} \weakto \bm{f}^\ast$ in $\Lp{2}$  as $n\to\infty.$ 
    Now, we show that  $\mathcal{J}_\beta(\bm{f}^\ast) \le \mathcal{J}_\beta(\bm{f})$ for all $\bm{f}\in \Lp{2},$ i.e.\ $\bm{f}^\ast$ is a minimiser of $\mathcal{J}_\beta.$  
    By \Cref{lem:MTcompactlinearL2L2},  we have that  $M_T(\bm{f}_{e_n}) =\bm{u}(\cdot, T, \bm{f}_{e_n}) \to M_T(\bm{f}^\ast)  = \bm{u}(\cdot,T,\bm{f}^\ast)$ strongly in $\Lp{2}$ as $n\to \infty.$ Hence, by
    \[
    \nrm{M_T(\bm{f}^\ast) - \bm{\Xi}_T} \le  \nrm{M_T(\bm{f}^\ast) - M_T(\bm{f}_{e_n})} + \nrm{M_T(\bm{f}_{e_n}) - \bm{\Xi}_T^{e_n}} + \nrm{\bm{\Xi}_T^{e_n}- \bm{\Xi}_T}
    \]
    and the lower semicontinuity of the norm, we have that 
    \begin{align*}
    \mathcal{J}_\beta(\bm{f}^\ast) &= \frac{1}{2} \nrm{M_T(\bm{f}^\ast) - \bm{\Xi}_T}^2+ \frac{\beta}{2}\nrm{\bm{f}^\ast}^2 \\
        & \leq \frac{1}{2}\lim\limits_{n\to\infty} \nrm{M_T(\bm{f}_{e_n}) - \bm{\Xi}_T^{e_n}}^2 + \frac{\beta}{2} \liminf\limits_{n\to\infty} \nrm{\bm{f}_{e_n}}^2 \\
        & =  \liminf\limits_{n\to\infty} \mathcal{J}^{e_n}_\beta(\bm{f}_{e_n}).
    \end{align*}
    Moreover, as $\bm{f}_{e_n}$ is the minimiser of $\mathcal{J}^{e_n}_\beta$ for all $n\in\NN,$ we have for all $\bm{f}\in\Lp{2}$ that 
    \[
     \mathcal{J}_\beta(\bm{f}^\ast) \le   \liminf\limits_{n\to\infty} \mathcal{J}^{e_n}_\beta(\bm{f}) = \mathcal{J}_\beta(\bm{f}). \qedhere
    \] 
\end{proof}

\subsection{Sensitivity and adjoint problems}
Alluding to the sequel, we will need to determine the gradient of the functional in \eqref{eq:JT11} for which we will compute the G\^{a}teaux differential. These computations rely on suitably formed sensitivity and adjoint problems, which are introduced here. As before, we denote with $(\bm{u}(\cdot; \bm{f}), \theta(\cdot;\bm{f}))$ the solution to \eqref{eq:ISP11problemP1} with given right-hand side source $\bm{f}.$  For $\delta_{\bm{f}} \in \Lp{2},$ we denote the variations of the solution to \eqref{eq:ISP11problemP1} by $ \delta\bm{u}(\cdot;\bm{f},\delta_{\bm{f}})$ and $\delta \theta(\cdot; \bm{f}, \delta_{\bm{f}}).$ These variations are defined as 
\begin{align}
    \label{eq:variationudelta}
    \delta\bm{u}(\cdot;\bm{f},\delta_{\bm{f}}) &= \lim\limits_{\varepsilon\to 0} \frac{1}{\varepsilon}\left(\bm{u}(\cdot;\bm{f}+\varepsilon\delta_{\bm{f}}) - \bm{u}(\cdot;\bm{f})\right)  \in\cIX{\Hko{1}},\\
    \label{eq:variationthetadelta}
    \delta\theta(\cdot;\bm{f}, \delta_{\bm{f}}) &= \lim\limits_{\varepsilon\to 0} \frac{1}{\varepsilon}\left(\theta(\cdot; \bm{f} + \varepsilon\delta_{\bm{f}}) - \theta(\cdot;\bm{f}))\right) \in \cIX{\hko{1}}.
\end{align}
Note that the occurring limits make sense in view of \eqref{eq:estimate under (ii)}. 

The sensitivity problem is now stated by the following set of equations,
\begin{equation}
    \label{eq:ISP11problemPS}
    \tag{$P_S^{11}$}
    \left\{
    \begin{array}{rll}
        \rho \partial_{tt}(\delta \bm{u}) -\nabla \cdot \sigma(\delta \bm{u}) + \gamma \nabla (\delta\theta) &= g\delta_{\bm{f}}  &\text{ in } Q_T \\
        \rho C_s \partial_t (\delta \theta) - \kappa \Delta (\delta \theta) - k\ast \Delta (\delta\theta) + T_0 \gamma \nabla \cdot \partial_t (\delta \bm{u}) &= 0&   \text{ in } Q_T \\
        \delta\bm{u}(\X,t) = \bm{0}, \, \delta\theta(\X,t) &=0 &  \text{ in } \Sigma_T \\
        \delta\bm{u}(\X,0) = \bm{0}, \, \partial_t \delta\bm{u}(\X,0) =\bm{0}, \, \delta\theta(\X,0) &= 0 &  \text{ in } \Omega,
    \end{array} 
    \right.
\end{equation}
and is obtained by expressing that the couples $(\bm{u}(\cdot;\bm{f}+\varepsilon\delta_{\bm{f}}), \theta(\cdot;\bm{f}+\varepsilon\delta_{\bm{f}}))$ and $(\bm{u}(\cdot;\bm{f}), \theta(\cdot;\bm{f}))$ form solutions to \eqref{eq:ISP11problemP1} with (perturbed) source $\bm{f}+\varepsilon\delta_{\bm{f}}$ and $\bm{f}$ respectively, for $\varepsilon>0$ and using \eqref{eq:variationudelta}-\eqref{eq:variationthetadelta}.  

\begin{remark}\label{rem:sensitivityproblemresebmlesdirect}
    Note that the problem \eqref{eq:ISP11problemPS} is equivalent to \eqref{eq:ISP11problemP1}, taking the perturbation $\delta_{\bm{f}}$ in the right-hand side instead of the function $\bm{f}.$
\end{remark}

The adjoint problem is stated in the unknown functions $\bm{u}^\ast$ and $\theta^\ast.$ It involves the remainder $\bm{u}(\cdot;\bm{f}) - \bm{\Xi}_T$ as a terminal velocity value and the operator $K^\ast$ from \eqref{eq:Kast}, i.e.
\begin{equation}
    \label{eq:ISP11problemPA}
    \tag{$P_A^{11}$}
    \left\{
    \begin{array}{rll}
        \rho \partial_{tt}\bm{u}^\ast -\nabla \cdot \sigma(\bm{u}^\ast) + \gamma T_0 \nabla \partial_t \theta^\ast  &=\bm{0}  & \quad \text{ in } Q_T \\
        -\rho C_s \partial_t \theta^\ast - \kappa \Delta  \theta^\ast - K^\ast(\Delta \theta^\ast ) - \gamma \nabla \cdot \bm{u}^\ast &= 0&  \quad  \text{ in } Q_T \\
        \bm{u}^\ast(\X,t) = \bm{0}, \quad \theta^\ast(\X,t) &=0 & \quad  \text{ in } \Sigma_T \\
        \bm{u}^\ast(\X,T) = \bm{0},\, \partial_t \bm{u}^\ast(\X,T) =\bm{u}(\X,T;\bm{f})-\bm{\Xi}_T(\X), \,\theta^\ast(\X,T) &= 0 & \quad \text{ in } \Omega.
    \end{array} 
    \right.
\end{equation}
This is a terminal-value problem, where the values of the unknown adjoint functions $(\bm{u}^\ast, \theta^\ast)$ are given by the final time conditions at $t=T.$ This system will be solved backwards in time. The problem \eqref{eq:ISP11problemPA} remains well posed due to the negative sign in $-\rho C_s\partial_t \theta^\ast.$ Notice that the solution depends on the function $\bm{f}$, we will write $(\bm{u}^\ast(\cdot;\bm{f}), \theta^\ast(\cdot;\bm{f}))$ for the solution to \eqref{eq:ISP11problemPA} to emphasise this dependence. This specific form of the adjoint problem will naturally follow in the computations carried out in the next subsection.

\subsection{G\^{a}teaux differential and gradient formulas}
We now aim to compute the G\^{a}teaux differential $\mathcal{J}^\prime_\beta(\bm{f}; \delta_{\bm{f}})$ of the cost functional $\mathcal{J}_\beta$ at $\bm{f}$ in the direction $\delta_{\bm{f}}.$ Denoting the space of directions $\delta_{\bm{f}}$ by $\mathcal{H}   \subseteq \Lp{2},$ which is assumed to be a real Hilbert space, we will be able to realise $\mathcal{J}^\prime_\beta(\bm{f};\cdot)$ as a bounded linear functional $\mathcal{H}\to \mathbb{R},$ see \Cref{thm:gateauxJT11}, for which we will use the sensitivity problem \eqref{eq:ISP11problemPS}. Therefore, invoking the Riesz representation theorem, there is a Riesz representer, denoted by $ \nabla\mathcal{J}_\beta[\bm{f}] = \nabla\mathcal{J}_\beta \in \mathcal{H},$ which is the unique element in $\mathcal{H}$ such that
\begin{equation}
    \label{eq:riesz}
\mathcal{J}^\prime_\beta(\bm{f}; \delta_{\bm{f}}) = \langle \nabla \mathcal{J}_\beta[\bm{f}], \delta_{\bm{f}}\rangle_{\mathcal{H}}, \quad \forall \delta_{\bm{f}} \in \mathcal{H}. 
\end{equation} 
We will suppress the dependence of $\bm{f}$ in $\nabla\mathcal{J}_\beta[\bm{f}]$ if it is clear from the context. Specifying different search directions $\mathcal{H}$ will yield different `gradients'. As will be seen in \eqref{eq:directionalderivativeJ}, the dependence of the perturbation $\delta_{\bm{f}}$ will not be explicit in the first part of the formula for $\mathcal{J}^\prime_\beta(\bm{f};\delta_{\bm{f}}).$ By suitable using the adjoint problem, we will make this dependence explicit for the case $\mathcal{H}=\Lp{2},$ see \Cref{thm:gradientsL2}, obtaining an expression for $\nabla \mathcal{J}_\beta \in\Lp{2}$ in terms of the solution $(\bm{u}^\ast, \theta^\ast)$ of the adjoint problem \eqref{eq:ISP11problemPA}. Afterwards, we discuss how the obtained $\Lp{2}$-gradient can be modified to obtain the Sobolev gradient $\nabla_S \mathcal{J}_\beta$ which corresponds to \eqref{eq:riesz} for the case $\mathcal{H} =\Hk{1},$ see \Cref{thm:gradientsS}.

\begin{theorem}\label{thm:gateauxJT11}
    The G\^{a}teaux directional derivative of $\mathcal{J}_\beta$ given by \eqref{eq:JT11} in the direction of the perturbation $\delta_{\bm{f}} \in \mathcal{H} \subseteq \Lp{2}$ at the function  $\bm{f}\in \Lp{2}$  is given by 
    \begin{equation}
        \label{eq:directionalderivativeJ}
        \mathcal{J}^\prime_\beta(\bm{f};\delta_{\bm{f}}) = \int_\Omega \delta\bm{u}(\X,T;\bm{f}, \delta_{\bm{f}}) \cdot \left(\bm{u}(\X,T;\bm{f}) - \bm{\Xi}_T(\X) \right) \,\mathrm{d}\X  + \beta \int_\Omega \bm{f} \cdot \delta_{\bm{f}} \,\mathrm{d}\X.
    \end{equation}
    The mapping $\mathcal{J}^\prime_\beta(\bm{f};\cdot) \colon \mathcal{H} \to \mathbb{R}$ defines a bounded linear functional.
\end{theorem}
\begin{proof}
    By definition, the G\^{a}teaux derivative at $\bm{f}$ in the direction $\delta_{\bm{f}}$ is given by
    $$
    \mathcal{J}^\prime_\beta(\bm{f}; \delta_{\bm{f}}) = \lim\limits_{\varepsilon\to 0} \frac{\mathcal{J}_{\beta}(\bm{f} + \varepsilon\delta_{\bm{f}}) - \mathcal{J}_{\beta} (\bm{f})}{\varepsilon} = \frac{\mathrm{d}}{\mathrm{d}\varepsilon} \left.\mathcal{J}_\beta\left(\bm{f} + \varepsilon\delta_{\bm{f}}\right) \right|_{\varepsilon=0}.
    $$
    A small calculation shows that 
   
    \begin{multline}\label{thm:gateauxJT11:eq1}
\mathcal{J}_\beta(\bm{f} + \varepsilon\delta_{\bm{f}}) - \mathcal{J}_\beta(\bm{f}) \\
= \varepsilon\scal{ M_T(\bm{f}) - \bm{\Xi}_T}{M_T(\delta_{\bm{f}})} + \frac{\epsilon^2}{2}\nrm{M_T(\delta_{\bm{f}})}^2 
 + \varepsilon\beta\scal{\bm{f}}{\delta_{\bm{f}}} + \frac{\varepsilon^2\beta}{2} \nrm{\delta_{\bm{f}}}^2.
\end{multline}
After dividing \eqref{thm:gateauxJT11:eq1} by $\varepsilon\neq 0$ and passing to the limit $\varepsilon\to 0$ in the result, we recover \eqref{eq:directionalderivativeJ}.
It is clear that $\mathcal{J}^\prime_\beta(\bm{f};\cdot)$ is linear. By \Cref{lem:forward_problem}, we have that $\nrm{M_T(\delta_{\bm{f}})}  \leq C\nrm{\delta_{\bm{f}}}$ and hence
\begin{align*}
\abs{\mathcal{J}^\prime_\beta(\bm{f};\delta_{\bm{f}})}
 & \leq (C \nrm{M_T(\bm{f}) - \bm{\Xi}_T} + \beta\nrm{\bm{f}})\nrm{\delta_{\bm{f}}}. \qedhere
\end{align*}
\end{proof}
\begin{theorem}\label{thm:gradientsL2}
Let $\mathcal{H}=\Lp{2},$ then we have
\begin{equation}
    \label{eq:gradientL2}
    \nabla \mathcal{J}_\beta[\bm{f}] = -\frac{1}{\rho} \int_0^T g(t) \bm{u}^\ast(\cdot,t;\bm{f}) \,\mathrm{d}t + \beta \bm{f} \in \Lp{2},
\end{equation}
where $(\bm{u}^\ast, \theta^\ast)$ is the solution to the adjoint problem \eqref{eq:ISP11problemPA}.
\end{theorem} 
\begin{proof} 
Fix $\bm{f}\in \Lp{2}$ and $\delta_{\bm{f}}$ in $\mathcal{H}=\Lp{2}.$ Let  the couple $(\delta\bm{u},\delta\theta)=(\delta\bm{u}(\cdot;\bm{f},\delta_{\bm{f}}), \delta\theta(\cdot;\bm{f},\delta_{\bm{f}}))$ be the solution to the sensitivity problem \eqref{eq:ISP11problemPS} and $(\bm{u}^\ast, \theta^\ast)=(\bm{u}^\ast(\cdot;\bm{f}), \theta^\ast(\cdot;\bm{f}))$ be the solution to the adjoint problem \eqref{eq:ISP11problemPA}. Multiply the equations for $(\delta\bm{u}, \delta\theta)$ with $(\bm{u}^\ast, \theta^\ast)$ respectively, and integrate both over the domain $\Omega$ and over the time interval $(0,T).$ This procedure yields the equations
\begin{multline} 
\label{eq:hyperbolicsensadj}
        \rho \int_0^T \scal{ \partial_{tt}(\delta \bm{u})}{\bm{u}^\star}\dt - \int_0^T \scal{\nabla \cdot \sigma(\delta \bm{u})}{\bm{u}^\star}\dt + \gamma \int_0^T\scal{\nabla \delta \theta}{\bm{u}^\star}\dt  \\= \int_0^T g(t) \scal{\delta_{\bm{f}}}{\bm{u}^\star} \dt  = \scal{\delta_{\bm{f}}}{\int_0^T g(t) \bm{u}^\star\dt }
\end{multline}
and 
\begin{multline}
\label{eq:parabolicsensadj}
        \rho C_s \int_0^T \scal{\partial_t(\delta\theta)}{\theta^\star} \dt  - \kappa \int_0^T \scal{\Delta(\delta \theta)}{\theta^\star} \dt  - \int_0^T \scal{(k \ast \Delta(\delta\theta))}{\theta^\star} \dt \\ + \gamma T_0 \int_0^T  \scal{\nabla \cdot \partial_t(\delta \bm{u})}{\theta^\star}\dt = 0.
\end{multline}
Applying partial integration in time and Green's theorem, taking into account the homogeneous Dirichlet boundary conditions, the initial conditions for $(\delta \bm{u}, \delta\theta)$ and the final time conditions for $(\bm{u}^\ast, \theta^\ast),$ leads to the following set of relations
\begin{align*}
\int_0^T \scal{\partial_{tt}(\delta \bm{u})}{\bm{u}^\star}\dt
&=  \left[ \scal{\partial_t(\delta\bm{u})}{\bm{u}^\star}\right]_{t=0}^T - \left[ \scal{\delta\bm{u}}{\partial_t\bm{u}^\star}\right]_{t=0}^T + \int_0^T \scal{\delta\bm{u}}{\partial_{tt}\bm{u}^\star} \dt  \\ 
&= - \int_\Omega\delta \bm{u}(\X,T;\bm{f},\delta_{\bm{f}}) \cdot \left(\bm{u}(\X,T; \bm{f})  - \bm{\Xi}_T(\X)\right)\di \X \\
 & \qquad \qquad+ \int_0^T \scal{\delta\bm{u}}{\partial_{tt} \bm{u}^\ast} \dt,
\end{align*} 
together with
\begin{align*} 
\int_0^T \scal{\nabla \cdot \sigma(\delta \bm{u})}{\bm{u}^\star} \dt &=  \int_0^T \scal{ \delta \bm{u}}{\nabla \cdot \sigma(\bm{u}^\star)}\dt, \\  \int_0^T \scal{\nabla \delta \theta}{\bm{u}^\star}\dt &= - \int_0^T \scal{\delta \theta}{\nabla\cdot \bm{u}^\star} \dt, \\
\int_0^T \scal{\partial_t(\delta\theta)}{\theta^\star}\dt &= - \int_0^T \scal{\delta\theta}{\partial_t \theta^\star}\dt, \\
\int_0^T \scal{\Delta (\delta \theta)}{\theta^\star}\dt  &=\int_0^T \scal{\delta\theta}{\Delta \theta^\star}\dt, \\
\int_0^T \scal{(k\ast \Delta (\delta \theta)}{\theta^\star}\dt  &= \int_0^T \scal{\delta \theta}{K^\ast(\Delta \theta^\star}\dt, \\
\int_0^T \scal{\nabla \cdot \partial_t(\delta \bm{u})}{\theta^\star}\dt &= \int_0^T \scal{\delta \bm{u}}{\nabla \partial_t \theta^\star}\dt. 
\end{align*}
In doing so, the equations \eqref{eq:hyperbolicsensadj}-\eqref{eq:parabolicsensadj} transform into 
\begin{multline}
\label{eq:hyperbolicsensadjtransf}
    \rho\int_0^T \scal{\delta \bm{u}}{\partial_{tt} \bm{u}^\ast} \dt  - \rho \int_\Omega \delta\bm{u}(\X,T;\bm{f},\delta_{\bm{f}})\cdot \left( \bm{u}(\X,T;\bm{f}) - \bm{\Xi}_T(\X)\right)\dX \\ - \int_0^T  \scal{\delta \bm{u}}{\nabla \cdot\sigma(\bm{u}^\ast)} \dt  - \gamma \int_0^T \scal{\delta \theta}{\nabla \cdot \bm{u}^\ast}\dt = \scal{\delta_{\bm{f}}}{ \int_0^T g(t)\bm{u}^\ast(\cdot,t;\bm{f})\dt }
\end{multline}
and 
\begin{multline}
\label{eq:parabolicsensadjtransf}
    \rho C_s \int_0^T \scal{\delta \theta}{-\partial_t \theta^\ast}\dt - \kappa \int_0^T \scal{\delta \theta}{\Delta \theta^\ast}\dt - \int_0^T  \scal{\delta \theta}{ K^\ast(\Delta \theta^\ast)}\dt  \\ + \gamma T_0 \int_0^T \scal{\delta \bm{u}}{\nabla \partial_t \theta^\ast}\dt = 0.
\end{multline}
The final step is to add \eqref{eq:hyperbolicsensadjtransf} and \eqref{eq:parabolicsensadjtransf} together, collecting the terms with $\delta\bm{u}$ and $\delta\theta$ and relying on the fact that $(\bm{u}^\ast, \theta^\ast)$ is a solution to \eqref{eq:ISP11problemPA}. This yields after simplification
$$
\int_\Omega \delta\bm{u}(\X,T; \bm{f},\delta_{\bm{f}}) \cdot\left( \bm{u}(\X,T;\bm{f}) - \bm{\Xi}_T(\X)\right)\dX  = \scal{\delta_{\bm{f}}}{-\frac{1}{\rho} \int_0^T g(t)\bm{u}^\ast(\cdot,t;\bm{f}) \dt}.
$$
Recalling the expressions \eqref{eq:directionalderivativeJ} from  \Cref{thm:gateauxJT11} and \eqref{eq:riesz}, we conclude
$$
\mathcal{J}^\prime_\beta(\bm{f};\delta_{\bm{f}}) = \scal{\delta_{\bm{f}}}{-\frac{1}{\rho} \int_0^T g(t)\bm{u}^\ast(\cdot,t;\bm{f})  + \beta \bm{f}} = \langle \nabla \mathcal{J}_\beta[\bm{f}], \delta_{\bm{f}}\rangle_{\Lp{2}}
$$
and therefore the proof of this theorem is finished.
\end{proof} 

By realising that, as a solution to the adjoint problem \eqref{eq:ISP11problemPA}, the function $\bm{u}^\ast(\cdot;\bm{f})$ vanishes at the boundary, so does the $\Lp{2}$-gradient $\nabla \mathcal{J}_{0}.$ 
Therefore, in the gradient methods based on the $\Lp{2}$-gradient $\nabla\mathcal{J}_0$ considered in the following subsection, the values of the iterates $\bm{f}_n$ will remain fixed once the starting point $\bm{f}_0\colon\overline{\Omega}\to\mathbb{R}$ is chosen. This behaviour is also observed for the Landweber procedure discussed in \Cref{sec:landweber}. 
Hence, near the boundary, the iterative method might struggle to converge to the sought solution if there is a mismatch in the boundary data. To overcome this difficulty, we will use the Sobolev gradient method \cite{JinZou2010,Neuberger2009}. Furthermore, to accelerate the convergence of the gradient algorithm, we parameterise the standard $\mathcal{H}=\Hk{1}$ inner product. This approach can be seen as gradient preconditioning; see \cite{NovruziProtas2018} and the references therein. Henceforth, we consider two continuous strictly positive weight functions
\begin{equation} 
\label{eq:r0r1weights}
r_0, r_1 \colon \Omega \to \mathbb{R},  \quad  0<\underline{r}_0\leq r_0 \leq \overline{r}_0, \quad 0<\underline{r}_1\leq r_1 \leq \overline{r}_1,
\end{equation}  which will tune the regularisation of the gradient, and are used in the weighted inner product
\begin{equation}
    \label{eq:weightedinnerproductHk}
    \langle \bm{u}, \bm{v}\rangle_{\HHi^1(\Omega,r_0,r_1)} := \scal{r_0 \bm{u}}{\bm{v}} + \scal{r_1\nabla \bm{u}}{\nabla{\bm{v}}}, \quad \forall \bm{u}, \bm{v} \in \Hk{1},
\end{equation}
where we denote with $\HHi^1(\Omega, r_0, r_1)$ the space $\Hk{1}$ equipped with the above inner product \eqref{eq:weightedinnerproductHk}.

\begin{theorem}\label{thm:gradientsS}
Let $\mathcal{H} = \Hk{1}$ be equipped with the inner product structure \eqref{eq:weightedinnerproductHk}. Let $\bm{f} \in \Lp{2}$ be fixed and let $\nabla\mathcal{J}_\beta:=\nabla\mathcal{J}_\beta[\bm{f}]$ be the $\Lp{2}$-gradient \eqref{eq:gradientL2}. Then we have that the Sobolev gradient $\nabla_S \mathcal{J}_\beta:=\nabla_S \mathcal{J}_\beta[\bm{f}] \in\Hk{1}$ satisfying the relation \eqref{eq:riesz} is the weak solution $\mathcal{K}$ to the elliptic boundary value problem
\begin{equation}
    \label{eq:gradientellitpicbvpJT}
    \left\{
    \begin{array}{rll}
    - \nabla \cdot \left( r_1 \nabla \mathcal{K} \right) + r_0\mathcal{K} &= \nabla \mathcal{J}_\beta & \text{ in } \Omega \\
    \nabla\mathcal{K}\,\bm{n} &= \bm{0} & \text{ on } \Gamma,
    \end{array} 
    \right.
\end{equation}
 where $\bm{n}$ is the outward unit vector on the boundary $\Gamma$ of the domain. 
\end{theorem} 
\begin{proof}
Noting that $\mathcal{H}=\Hk{1} \subset \Lp{2}$ an invoking \Cref{thm:gateauxJT11}, we have on the one hand that
$$
\mathcal{J}^\prime_\beta(\bm{f};\delta_{\bm{f}}) = \langle \nabla \mathcal{J}_\beta[\bm{f}], \delta_{\bm{f}} \rangle_{\Lp{2}},
$$
and on the other hand that $\mathcal{J}^\prime(\bm{f};\cdot)$ is a bounded linear functional 
\begin{equation*}
\mathcal{J}^\prime(\bm{f};\cdot) \colon \HHi^1(\Omega, r_0,r_1) \to\mathbb{R}
\end{equation*} whose Riesz representer will be denoted as $\nabla_S \mathcal{J}_\beta[\bm{f}].$ Therefore, the relation
$$
\langle \nabla_S \mathcal{J}_\beta[\bm{f}], \delta_{\bm{f}}\rangle_{\HHi^1(\Omega,r_0,r_1)} = \langle \nabla \mathcal{J}_\beta[\bm{f}], \delta_{\bm{f}}\rangle_{\Lp{2}},
$$
should hold for any $\delta_{\bm{f}} \in \Hk{1} \subset \Lp{2}.$ This precisely means that $\nabla_S\mathcal{J}_\beta[\bm{f}]$ is a weak solution to \eqref{eq:gradientellitpicbvpJT}. \qedhere 
\end{proof} 
The Sobolev gradient $\nabla_S\mathcal{J}_\beta \in \Hk{1}$ is smoother than the $\Lp{2}$-gradient $\nabla \mathcal{J}_\beta.$ Usually $r_0(\X)=1$ is taken, and $r_1$ is chosen by trial and error dependent on the amount of regularisation needed. However, some results for choosing these parameters are available in \cite{NovruziProtas2018}. Taking $r_1$ constant, we see that the Sobolev gradient is the $\Lp{2}$-gradient preconditioned by the operator $(I - r_1 \Delta)^{-1}.$ On the Fourier side, $r_1$ defines a cutoff region for the most oscillating modes \cite{JinZou2010,Lesnic2019}.

\subsection{Algorithms}\label{subsec:gradientalgorithms}
In this part, we describe two gradient methods. The standard steepest descent algorithm starts from the initial guess $\bm{f}_0\colon\overline{\Omega} \to \mathbb{R}$ with $\bm{f}_0 \in\Lp{2}$ and produces the updates
\begin{equation}
    \label{eq:steepest descent}
\left\{ 
\begin{array}{l} 
\bm{f}_n = \bm{f}_{n-1} - \tau_n \nabla_{\mathcal{H}} \mathcal{J}_\beta\left[\bm{f}_{n-1}\right],\\
\tau_n = \argmin\limits_{ \tau>0}  \mathcal{J}_\beta\left(\bm{f}_{n-1} -\tau \nabla_{\mathcal{H}} \mathcal{J}_\beta[\bm{f}_{n-1}]\right). 
\end{array} 
\right. 
\end{equation}
Here,  $\tau_n>0$ is the stepsize at the $n$-th iteration step whose optimal value can be found using e.g.\ line search in the corresponding (1D) minimisation problem,
 and $\nabla_{\mathcal{H}}\mathcal{J}_\beta$ is some gradient. For the problem at hand, we are able to obtain an explicit formula for $\tau_n$ based on the following lemma.
 
\begin{lemma}\label{lem:optimaltau}
Let $(\mathcal{H}_1,\scal{\cdot}{\cdot}_{\mathcal{H}_1})$ and $(\mathcal{H}_2,\scal{\cdot}{\cdot}_{\mathcal{H}_2})$ be real Hilbert spaces and fix $u,v\in \mathcal{H}_1$ and $w,z\in\mathcal{H}_2,$ with $v\neq0$ and $z\neq 0$. The function $j \colon \mathbb{R} \to \mathbb{R}$ given by
$$
j(\tau) = \frac{1}{2} \nrm{u + \tau v}^2_{\mathcal{H}_1} + \frac{\beta}{2}\nrm{w+\tau z}^2_{\mathcal{H}_2}, \quad \beta\geq 0,
$$
attains it minimum at
\begin{equation}\label{eq:steptauopt} 
\tau^\ast = - \frac{\scal{u}{v}_{\mathcal{H}_1} + \beta \scal{w}{z}_{\mathcal{H}_2} }{\nrm{v}^2_{\mathcal{H}_1}+ \beta\nrm{z}^2_{\mathcal{H}_2}}.
\end{equation} 
\end{lemma}
\begin{proof}
    Note that
    $$
    2j(\tau) = \nrm{u}^2_{\mathcal{H}_1} + \beta \nrm{w}^2_{\mathcal{H}_2} + 2\tau\left(\scal{u}{v}_{\mathcal{H}_1} + \beta \scal{w}{z}_{\mathcal{H}_2}\right)  + \tau^2 \left(\nrm{v}^2_{\mathcal{H}_1}+\beta \nrm{z}^2_{\mathcal{H}_2} \right) 
    $$
    is an upwards parabola attaining its minimal value at the value $\tau^\ast$ satisfying $\partial_\tau j(\tau^\ast)=0,$ which is exactly \eqref{eq:steptauopt}.
\end{proof}

Applying \Cref{lem:optimaltau}  and using \eqref{eq:JT112}, the (optimal) value of $\tau_n$ for each step in \eqref{eq:steepest descent} is given by the expression
\begin{equation}
    \label{eq:steepestdesenttauoptimal}
    \tau_n^\ast =\frac{\scal{M_T(\bm{f}_{n-1}) - \bm{\Xi{_T}}}{M_T(\nabla_{\mathcal{H}}\mathcal{J}_\beta[\bm{f}_{n-1}])}+\beta \scal{\bm{f}_{n-1}}{\nabla_{\mathcal{H}} \mathcal{J}_\beta[\bm{f}_{n-1}]}}{\nrm{M_T(\nabla_{\mathcal{H}}\mathcal{J}_\beta[\bm{f}_{n-1}])}^2+ \beta \nrm{\nabla_\mathcal{H}\mathcal{J}_\beta[\bm{f}_{n-1}]}^2}.
\end{equation}

Based on the discussions above, we will consider the $\Lp{2}$-gradient and the Sobolev gradient, i.e.\ $\nabla_{\mathcal{H}}\mathcal{J}_\beta = \nabla\mathcal{J}_\beta$ and $\nabla_{\mathcal{H}}\mathcal{J}_\beta = \nabla_S\mathcal{J}_\beta$ given by \eqref{eq:gradientL2} and \eqref{eq:gradientellitpicbvpJT}, respectively. Note that the computation of $\nabla \mathcal{J}_\beta[\bm{f}]$ involves solving the direct problem \eqref{eq:ISP11problemP1} for $\bm{f}$ from which the solution is used in the adjoint problem \eqref{eq:ISP11problemPA} and then the expression \eqref{eq:gradientL2} is calculated. In the Sobolev case $\nabla_S\mathcal{J}_\beta[\bm{f}]$,  an additional elliptic boundary value problem \eqref{eq:gradientellitpicbvpJT} must be solved afterwards. 

The conjugate gradient method is based on the following procedure, starting from the initial guess  $\bm{f}_0,$ we calculate the first gradient $\nabla_{\mathcal{H}} \mathcal{J}_\beta[\bm{f}_0]$ and perform a line search to determine the first stepsize using \eqref{eq:steepestdesenttauoptimal}
$$
\tau_1^\ast =  \argmin\limits_{\tau>0} \left(\bm{f}_0 - \tau \nabla_{\mathcal{H}}\mathcal{J}_\beta[\bm{f}_0]\right),
$$
which yield the first iterate
$$
\bm{f}_1 = \bm{f}_0 - \tau_1^\ast \nabla_{\mathcal{H}}\mathcal{J}_\beta[\bm{f}_0].
$$
We initialize the conjugate gradient by setting $\nabla \mathcal{F}_0 = - \nabla_{\mathcal{H}}\mathcal{J}_\beta[\bm{f}_0] = \Lambda_0$ and start the iterative procedure for the $n$-th step ($n\geq1$):
\begin{enumerate}
    \item Calculate the steepest direction $ \Lambda_n = -\nabla_\mathcal{H} \mathcal{J}_\beta[\bm{f}_n];$
    \item Compute the Fletcher-Reeves coefficient $    \zeta_n = \nrm{\Lambda_n}^2 / \nrm{\Lambda_{n-1}}^2;    $
    \item Calculate the conjugate gradient $\nabla \mathcal{F}_n = \Lambda_n + \zeta_n \nabla \mathcal{F}_{n-1};$
    \item Compute the stepsize $\tau_n^\ast$ given by 
    $$
    \tau_n^\ast = - \frac{\scal{M_T(\bm{f}_n)-\bm{\Xi}_T}{M_T(\nabla \mathcal{F}_n)} + \beta \scal{\bm{f}_n}{\nabla \mathcal{F}_n}}{ \nrm{M_T(\nabla \mathcal{F}_n)}^2  + \beta \nrm{\nabla \mathcal{F}_n}^2};
    $$
    \item Update the approximation
    $$
    \bm{f}_{n+1} = \bm{f}_{n} + \tau_n^\ast \nabla \mathcal{F}_{n}. 
    $$
\end{enumerate}
Note that here $M_T(\nabla \mathcal{F}_n)$ corresponds to the solution of the 
direct problem with given right-hand side $\nabla \mathcal{F}_n$, evaluated at the final time, i.e.\ $\bm{u}(\cdot,T; \nabla\mathcal{F}_n).$ 

\begin{remark}\label{rem:stoppinggradient}
The iterations for the gradient algorithms are stopped whenever the maximal number of iterations is surpassed or if the resulting $\bm{f}_{n+1}$ produces a larger value of the cost functional, i.e.\ we stop when $\mathcal{J}_\beta(\bm{f}_{n+1})>\mathcal{J}_\beta(\bm{f}_n)$ and retain $\bm{f}_n$ as the final approximation. 
\end{remark}

Summarising, the approximation of the unknown source in the inverse source problem can be attained by either the Landweber method or a gradient method (steepest descent or conjugate gradient). For the gradient methods, one can consider the standard $\Lp{2}$-gradient formula for which (as in the Landweber scheme) the values of $\bm{f}$ at the boundary should be available. For the Sobolev gradients, both the steepest descent and the conjugate gradient method need no assumptions on $\left.\bm{f}\right|_{\partial \Omega}$ and are in this view more flexible.

\subsection{Adaptations for \textbf{ISP1.2}}
The objective functional associated to the inverse problem \textbf{ISP1.2}, where the given measurement is $\bm{X}_T\in\Lp{2}$ and the corresponding operator is given by $N_T,$ see \Cref{subsec:adaptations12land} and \eqref{eq:ISP12measurementmapN}, is given by
\begin{equation}
    \label{eq:IT12}
    \mathcal{I}_\beta(\bm{f}) = \frac{1}{2} \nrm{ N_T(\bm{f})  - \bm{X}_T}^2 + \frac{\beta}{2}\nrm{\bm{f}}^2,  \quad \beta \geq 0.
\end{equation}
The noisy version is
\begin{equation}
    \label{eq:ITe12}
    \mathcal{I}^e_\beta(\bm{f}) = \frac{1}{2}\nrm{N_T(\bm{f}) - \bm{X}_T^e}^2 + \frac{\beta}{2}\nrm{\bm{f}}^2, \quad \beta\geq0,  \quad e \geq 0.
\end{equation}
As in \Cref{lem:MTcompactlinearL2L2}, the linear measurement map $N_T \colon \Lp{2}\to\Lp{2}$ is compact (i.e.\ continuous). Consequently, the arguments of \Cref{thm:JT11uniqemin} and \Cref{thm:JTe11uniquemin} can be reused to show the analogous results for \textbf{ISP1.2}.
\begin{theorem}\label{thm:ITE12uniquemins}
The minimisation problems
$$
\argmin_{\bm{f}\in\Lp{2}} \mathcal{I}_\beta(\bm{f}) \quad \text{ and } \quad  \argmin_{\bm{f} \in \Lp{2}} \mathcal{I}^e_\beta(\bm{f})
$$
for $ \beta >0$ and $ e \geq 0$ have unique solutions $\bm{f}$ and $\bm{f}_e$ in $\Lp{2},$ respectively. In case $\beta=0,$ unique solutions exist on each nonempty bounded closed convex subset.
Moreover, for any $\beta\geq 0$, let $\{e_n\}_n\subseteq{\mathbb{R}_{\geq0}}$ be a sequence with $e_n \searrow 0$ as $n\to \infty,$ and let $\{\bm{X}_T^{e_n}\}_n$ be the corresponding sequence of noisy measurements in $\Lp{2}$ such that $\bm{X}_T^{e_n}\to \bm{X}_T$ in $\Lp{2}$ as $n\to \infty.$ If the associated sequence $\{\bm{f}_{e_n}\}_n$ in $\Lp{2}$ of minimisers of $\mathcal{I}_\beta^{e_n}$ is uniformly bounded, then $\{\bm{f}_{e_n}\}_n$ converges weakly to a minimiser of $\mathcal{I}_\beta.$ 
\end{theorem}
\begin{proof}
    The functionals $\mathcal{I}_\beta$ and $\mathcal{I}_\beta^e$ are  strictly convex, continuous and weakly coercive if $\beta>0$, hence \cite[Theorem~25.E]{Zeidler-IIB} is applicable. If $\beta=0,$ we can instead apply \cite[Theorem~25.C-Corollary~25.15]{Zeidler-IIB}. If $\nrm{\bm{f}_e} <c,$ then $\bm{f}_e \rightharpoonup \bm{f}^\ast$ for some $\bm{f}^\ast\in\Lp{2}.$ Since $N_T$ is compact, the sequence of measurements $N_T(\bm{f}_e) \to N_T(\bm{f^\ast})$ strongly in $\Lp{2}.$ Scrutinising the proof of \Cref{thm:JTe11uniquemin} yields once more that for any $\bm{f}\in\Lp{2}$ with $\nrm{\bm{f}} <c$ it holds that $\mathcal{I}_\beta(\bm{f}^\ast)\leq \mathcal{I}_\beta(\bm{f}).$
\end{proof}

The sensitivity problem remains the same as for \textbf{ISP1.1}, i.e.\
\begin{equation}
    \label{eq:ISP12problemPS}
    \tag{$P_S^{12}$}
    \left\{
    \begin{array}{rll}
        \rho \partial_{tt}(\delta \bm{u}) -\nabla \cdot \sigma(\delta \bm{u}) + \gamma \nabla (\delta\theta) &= g(t)\delta_{\bm{f}}(\X)  & \quad \text{ in } Q_T \\
        \rho C_s \partial_t (\delta \theta) - \kappa \Delta (\delta \theta) - k\ast \Delta (\delta\theta) + T_0 \gamma \nabla \cdot \partial_t (\delta \bm{u}) &= 0&  \quad  \text{ in } Q_T \\
        \delta\bm{u}(\X,t) = \bm{0}, \, \delta\theta(\X,t) &=0 & \quad  \text{ in } \Sigma_T \\
        \delta\bm{u}(\X,0) = \bm{0}, \, \partial_t \delta\bm{u}(\X,0) =\bm{0}, \, \delta\theta(\X,0) &= 0 & \quad \text{ in } \Omega.
    \end{array} 
    \right.
\end{equation}
The adjoint problem remains a terminal-value problem, but now contains the deficit in the measurement as   a forcing term, i.e.\
\begin{equation}
    \label{eq:ISP12problemPA}
    \tag{$P_A^{12} $}
    \left\{
    \begin{array}{rll}
        \rho \partial_{tt}\bm{u}^\ast -\nabla \cdot \sigma(\bm{u}^\ast) + \gamma T_0 \nabla \partial_t \theta^\ast  &= 
        N_T(\bm{f})- \bm{X}_T(\X)  &  \text{ in } Q_T \\
        -\rho C_s \partial_t \theta^\ast - \kappa \Delta  \theta^\ast - K^\ast(\Delta \theta^\ast ) - \gamma \nabla \cdot \bm{u}^\ast &= 0&    \text{ in } Q_T \\
        \bm{u}^\ast(\X,t) = \bm{0}, \, \theta^\ast(\X,t) &=0 &   \text{ in } \Sigma_T \\
        \bm{u}^\ast(\X,T) = \bm{0}, \, \partial_t \bm{u}^\ast(\X,T) & =\bm{0}, \ \theta^\ast(\X,T) = 0 &  \text{ in } \Omega.
    \end{array} 
    \right.
\end{equation}

The relevant properties of $\mathcal{I}_\beta$ are collected in the next theorem.
\begin{theorem}\label{thm:gateauxIT11}
    The G\^{a}teaux directional derivative of $\mathcal{I}_\beta$ in the direction of the perturbation $\delta_{\bm{f}} \in\mathcal{H} \subseteq \Lp{2}$ at the function $\bm{f}\in\Lp{2}$ is given by
    \begin{multline}
        \label{eq:directionalderivativeI}
        \mathcal{I}^\prime_\beta(\bm{f};\delta_{\bm{f}}) = \int_\Omega \left(\int_0^T\delta\bm{u}(\X,t;\bm{f},\delta_{\bm{f}})\dt\right)\cdot\left(\int_0^T \bm{u}(\X,t;\bm{f})\dt - \bm{X}_T(\X)\right) \dX\\ + \beta \int_\Omega \bm{f}\cdot \delta_{\bm{f}}\dX.
    \end{multline}
     The mapping $\mathcal{I}_\beta^\prime(\bm{f}; \cdot) \colon \mathcal{H} \to \mathbb{R}$ is a bounded linear functional. Denoting with $\nabla_\mathcal{H}  \mathcal{I}_\beta[\bm{f}]$ the unique element in $\mathcal{H}$ such that
    $$
    \mathcal{I}^\prime_\beta(\bm{f};\delta_{\bm{f}}) = \langle \nabla_\mathcal{H} \mathcal{I}_\beta[\bm{f}], \delta_{\bm{f}}\rangle_{\mathcal{H}}, \quad \forall \delta_{\bm{f}} \in \mathcal{H},
    $$
    it holds that  
    \begin{enumerate}
        \item in case $\mathcal{H} = \Lp{2}$, 
    \begin{equation} 
        \label{eq:gradientIL2}
        \nabla\mathcal{I}_\beta[\bm{f}] = \int_0^T g(t) \bm{u}^\ast(\cdot,t;\bm{f})\dt + \beta \bm{f} \in \Lp{2},
    \end{equation}
    where $(\bm{u}^\ast, \theta^\ast)$ is the solution to the adjoint problem \eqref{eq:ISP12problemPA};
    \item in case $\mathcal{H} = \Hk{1}$ equipped with \eqref{eq:weightedinnerproductHk}, the Sobolev gradient $\nabla_S \mathcal{I}_\beta \in \Hk{1}$ is the weak solution $\mathcal{K}$ to the problem
    \begin{equation}
    \label{eq:gradientellitpicbvpIT}
    \left\{
    \begin{array}{rll}
    - \nabla \cdot \left( r_1 \nabla  \mathcal{K}\right) + r_0 \mathcal{K} &= \nabla \mathcal{I}_\beta & \text{ in } \Omega \\
    \nabla\mathcal{K}\,\bm{n} &= \bm{0} & \text{ on } \Gamma.
    \end{array} 
    \right.
\end{equation}
\end{enumerate}
\end{theorem}
\begin{proof} 
 Fixing  $\bm{f}\in \Lp{2}$ and $\delta_{\bm{f}} \in \mathcal{H}$, we find that
\begin{align*}
\mathcal{I}_\beta(\bm{f} + \varepsilon\delta_{\bm{f}}) - \mathcal{I}_\beta(\bm{f}) &= \epsilon\left[ \scal{N_T({\bm{f}}) - \bm{X}_T}{N_T(\delta_{\bm{f}})} + \beta \scal{\bm{f}}{\delta_{\bm{f}}}\right] \\ & \qquad + \frac{\varepsilon^2}{2} \left(\nrm{N_T(\delta_{\bm{f}})}^2 + \beta\nrm{\delta_{\bm{f}}}^2\right).
\end{align*}
Dividing by $\varepsilon \neq 0$ and passing to the limit $\varepsilon\to 0$ we obtain the expression \eqref{eq:directionalderivativeI}. The map $\mathcal{I}^\prime_\beta(\bm{f};\cdot) \colon\mathcal{H} \to \mathbb{R}$ is linear and bounded as
\begin{align*}
    \abs{\mathcal{I}_\beta^\prime(\bm{f};\delta_{\bm{f}})}&\leq \nrm{N_T(\bm{f})-\bm{X}_T} \nrm{N_T(\delta_{\bm{f}})} + \beta \nrm{\bm{f}} \nrm{\delta_{\bm{f}}} \\
    & \leq \left(C \nrm{N_T(\bm{f})-\bm{X}_T} + \beta\nrm{\bm{f}}\right) \nrm{\delta_{\bm{f}}},
\end{align*}
by appealing to the fact that $N_T$ is bounded. 

Specialising to the case $\mathcal{H} = \Lp{2},$ the starting point is as in \Cref{thm:gradientsL2}. Fix $\bm{f}, \delta_{\bm{f}} \in \Lp{2}$ and let $(\delta\bm{u},\delta\theta)$ be the solution to the sensitivity problem \eqref{eq:ISP12problemPS} and $(\bm{u}^\ast, \theta^\ast)$ be the solution to the adjoint problem \eqref{eq:ISP12problemPA}. Reconsider \eqref{eq:hyperbolicsensadj}-\eqref{eq:parabolicsensadj} and apply Green's theorem and partial integration in time, now considering zero initial conditions for $(\delta\bm{u},\delta\theta)$ and zero final time conditions for $(\bm{u}^\ast,\theta^\ast).$ This yields
\begin{multline*}
    \rho \int_0^T \scal{\delta\bm{u}}{\partial_{tt} \bm{u}^\ast} \dt - \int_0^T \scal{\delta\bm{u}}{\nabla\cdot \sigma(\bm{u}^\ast)} \dt - \gamma \int_0^T \scal{\delta \theta}{\nabla \cdot \bm{u}^\ast}\dt \\ = \scal{ \delta_{\bm{f}}}{\int_0^T g(t)\bm{u}^\ast \dt} 
\end{multline*}
and
\begin{multline*}
    -\rho C_s \int_0^T \scal{\delta\theta}{\partial_t\theta^\ast}\dt - \kappa \int_0^T \scal{\delta\theta}{\Delta \theta^\ast}\dt - \int_0^T \scal{\delta\theta}{K^\ast(\Delta \theta^\ast)} \dt \\+ \gamma T_0 \int_0^T \scal{\delta\bm{u}}{\nabla \partial_t \theta^\ast}\dt = 0.
\end{multline*}
Adding these expressions and collecting terms together gives 
\begin{multline*}
    \scal{\delta_{\bm{f}}}{ \int_0^T g(t)\bm{u}^\ast(\cdot,t;\bm{f}) \dt} = 
    \int_0^T \scal{\delta\bm{u}}{ \rho \partial_{tt}\bm{u}^\ast - \nabla \cdot \sigma(\bm{u}^\ast) + \gamma T_0 \nabla\partial_t \theta^\ast}\dt \\\qquad\qquad+ \int_0^T\scal{\delta\theta}{-\rho C_s \partial_t \theta^\ast - \kappa \Delta \theta^\ast - K^\ast(\Delta\theta^\ast) - \gamma \nabla \cdot\bm{u}^\ast}\dt. 
\end{multline*}
Hence, by the definition of \eqref{eq:ISP12problemPA}, we have that 
\begin{align*}
    &\scal{\delta_{\bm{f}}}{ \int_0^T g(t)\bm{u}^\ast(\cdot,t;\bm{f}) \dt} \\&\quad =\int_0^T \scal{\delta\bm{u}(\cdot,t;\bm{f},\delta_{\bm{f}})}{\int_0^T \bm{u}(\X,s;\bm{f})\ds - \bm{X}_T}\dt \\ 
    &\quad = \int_\Omega \left(\int_0^T\delta\bm{u}(\X,t;\bm{f},\delta_{\bm{f}})\dt\right)\cdot\left(\int_0^T \bm{u}(\X,t;\bm{f}) \dt - \bm{X}_T(\X)\right) \dX.
\end{align*}
Therefore, in the case $\mathcal{H}=\Lp{2}$ we obtain the expression \eqref{eq:gradientIL2} by using \eqref{eq:directionalderivativeI}. The transition to the Sobolev version follows the same lines  as before by  realising
$$
\mathcal{I}^\prime_\beta(\bm{f};\delta_{\bm{f}}) = \langle \nabla \mathcal{I}_\beta [\bm{f}], \delta_{\bm{f}} \rangle_{\Lp{2}} = \langle \nabla_{S}\mathcal{I}_\beta[\bm{f}], \delta_{\bm{f}}\rangle_{\HHi^{1}(\Omega,r_0,r_1)},
$$
using the weighted inner product \eqref{eq:weightedinnerproductHk} and the weak formulation of \eqref{eq:gradientellitpicbvpIT}.
\end{proof} 

Using these formulas for the gradients $\nabla_{\mathcal{H}} \mathcal{I}_\beta,$ the algorithms for the minimisation problem are mutatis mutandis in the same fashion as those in \Cref{subsec:gradientalgorithms}. Let us note that the stepsize $\tau_n^\ast$ in the  steepest decent gradient scheme now takes the form
$$
\tau_n^\ast =   \frac{\scal{N_T(\bm{f}_{n-1}) - \bm{X}_T}{N_T(\nabla_{\mathcal{H}} \mathcal{I}_\beta[\bm{f}_{n-1}])} + \beta\scal{\bm{f}_{n-1}}{\nabla_\mathcal{H} \mathcal{I}_\beta[\bm{f}_{n-1}]}}{\nrm{N_T(\nabla_\mathcal{H}\mathcal{I}_\beta[\bm{f}_{n-1}])}^2+ \beta\nrm{\nabla_\mathcal{H} \mathcal{I}_\beta[\bm{f}_{n-1}]}^2}.
$$

\subsection{Adaptations for \textbf{ISP2}}
Concerning the inverse problem \textbf{ISP2} to reconstruct the space dependent heat source $f\in\lp{2}$ in \eqref{eq:ISP2problemP1}, we work with the functional
\begin{equation}
    \label{eq:GT2}
    \mathcal{G}_\beta(f) = \frac{1}{2}\nrm{\int_0^T \theta(\cdot,t;f)\dt - \Psi_T }^2 + \frac{\beta}{2}\nrm{f}^2, \quad \beta \geq 0.
\end{equation}
Where $\Psi_T \in\lp{2}$ as the remainder of the integral measurement $\psi_T$ is the available data and $(\bm{u}(\cdot;f),\theta(\cdot;f))$ is the solution to the direct problem \eqref{eq:ISP2problemP1} with given right-hand side $f.$ Similarly as before, under the conditions of \Cref{thm:isp2}, the linear operator $ P_T\colon \lp{2} \to \lp{2}$ defined in \eqref{eq:ISP2measurementmapP} is  compact. This enables us to conclude the existence and uniqueness of a solution to the (noisy) minimisation problem associated to the cost functionals $\mathcal{G}_\beta$ and $\mathcal{G}_\beta^e,$ with
\begin{equation}
    \label{eq:GTE2}
    \mathcal{G}^e_\beta(f) = \frac{1}{2} \nrm{P_T(f) - \Psi_T^e}^2 + \frac{\beta}{2} \nrm{f}^2,\quad \beta \geq 0,\quad e \geq 0.
\end{equation}
\begin{theorem}\label{thm:GTE12uniquemins}
The minimisation problems
$$
\argmin\limits_{f \in \lp{2}} \mathcal{G}_\beta(f) \quad \text{ and } \quad \argmin\limits_{f \in \lp{2}} \mathcal{G}^e_\beta(f)
$$
for $\beta>0$ and $e \geq 0$ have unique solutions in $\lp{2}$. In case $\beta = 0,$ unique solutions exist in nonempty bounded closed convex subsets. 
 Moreover, for any $\beta\geq0,$ let $\{e_n\}_n\subseteq{\mathbb{R}_{\geq0}}$ be a sequence with $e_n \searrow 0$ as $n\to \infty,$ and let $\{\Psi_T^{e_n}\}_n$ be the corresponding sequence of noisy measurements in $\lp{2}$ such that $\Psi_T^{e_n}\to \Psi_T$ in $\lp{2}$ as $n\to \infty.$ If the associated sequence $\{f_{e_n}\}_n$ in $\lp{2}$ of minimisers of $\mathcal{G}_\beta^{e_n}$ is uniformly bounded, then $\{f_{e_n}\}_n$ converges weakly to a minimiser of $\mathcal{G}_\beta.$
\end{theorem}
\begin{proof}
Repeat the arguments from \Cref{thm:JTe11uniquemin} (or \Cref{thm:ITE12uniquemins}) now noting that $P_T(f_e) \to P_T(f)$ strongly in $\lp{2}$ whenever $f_e \rightharpoonup f^\ast$ weakly for  $f^\ast \in \lp{2}.$
\end{proof} 

The sensitivity problem for \textbf{ISP2} is now formulated as
\begin{equation}
    \label{eq:ISP2problemPS}
    \tag{${P}_S^{2}$}
    \left\{
    \begin{array}{rll}
    \rho \partial_{tt}(\delta\bm{u})  - \nabla \cdot\sigma(\delta\bm{u}) + \gamma \nabla \theta &= 0  & \quad \text{ in } Q_T \\
        \rho C_s \partial_t \theta - \kappa \Delta \theta - k\ast \Delta \theta + T_0 \gamma \nabla \cdot \partial_t (\delta\bm{u}) &= g(t)\delta_f(\X)&  \quad  \text{ in } Q_T \\
        \bm{u}(\X,t) = \bm{0}, \quad \theta(\X,t) &=0 & \quad  \text{ in } \Sigma_T \\
        \bm{u}(\X,0) = \bm{0}, \quad \partial_t \bm{u}(\X,0) =\bm{0}, \quad \theta(\X,0) &= 0 & \quad \text{ in } \Omega.
    \end{array} 
    \right.
\end{equation}
The remainder term $\int_0^T \theta(\X,t;f)\dt - \Psi_T(\X)$ now acts as a forcing term in the role of the heat source in the (terminal-value) adjoint problem for $(\bm{u}^\ast,\theta^\ast),$ i.e.
\begin{equation}
    \label{eq:ISP2problemPA}
    \tag{${P}_A^{2}$}
    \left\{ 
    \begin{array}{rll}
    \rho \partial_{tt} \bm{u}^\ast - \nabla \cdot \sigma(\bm{u}^\ast) + \gamma T_0 \nabla \partial_t \theta^\ast &= \bm{0} &  \text{ in } Q_T \\ 
    -\rho C_s \partial_t \theta^\ast -\kappa \Delta \theta^\ast - K^\ast(\Delta \theta^\ast) - \gamma \nabla \cdot \bm{u}^\ast
    &=P_T(f) - \Psi_T(\X) & \text{ in } Q_T \\
    \bm{u}^\ast(\X,t) = \bm{0}, \quad \theta^\ast(\X,t) &= 0 &  \text{ in } \Sigma_T \\
    \bm{u}^\ast(\X,T) = \bm{0},  \quad \partial_t \bm{u}(\X,T) = \bm{0}, \quad \theta^\ast(\X,T) &= 0 & \text{ in } \Omega.
    \end{array} 
    \right. 
\end{equation}
Finally, the G\^{a}teaux derivative and the gradients of $\mathcal{G}_\beta$ are listed in the following theorem.
\begin{theorem}\label{thm:gateauxGT2}
    The G\^{a}teaux directional derivative of $\mathcal{G}_\beta$ in the direction of the perturbation $\delta_f \in \mathcal{H}\subseteq \lp{2}$ at the function $f \in\lp{2}$ is given by
    \begin{align}
        \label{eq:directionalderivativeG}
        \mathcal{G}_\beta^\prime(f;\delta_{f})& = \int_\Omega \left(\int_0^T  \delta\theta(\X,t;f,\delta_f) \dt \right)\left(\int_0^T \theta(\X,t;f)\dt - \Psi_T(\X)\right) \dX\nonumber\\ & \qquad  + \beta \int_\Omega f\delta_f \dX. 
    \end{align}
    The mapping  $\mathcal{G}_\beta^\prime(f;\cdot) \colon \mathcal{H}\to\mathbb{R}$ is a bounded linear functional. Denoting by $\nabla_\mathcal{H} \mathcal{G}_\beta[f]$ the unique element in $\mathcal{H}$ such that 
    $$
    \mathcal{G}_\beta^\prime(f;\delta_f) = \langle \nabla_{\mathcal{H}} \mathcal{G}_\beta[f], \delta_f\rangle_{\mathcal{H}}, \quad \forall \delta_{f} \in\mathcal{H},
    $$
    it holds that
    \begin{enumerate}
        \item in case $\mathcal{H} = \lp{2},$
        \begin{equation}
            \label{eq:gradientGL2}
            \nabla\mathcal{G}_\beta[f] = \int_0^T g(t) \theta^\ast(\cdot,t;f)\dt + \beta f \in \lp{2},
        \end{equation}
        where $(\bm{u}^\ast, \theta^\ast)$ is the solution to the adjoint problem \eqref{eq:ISP2problemPA};

        \item in case $\mathcal{H} = \hk{1},$ equipped with the regularised inner product
        \begin{equation}\label{eq:weightedinnerproducthk}
        \langle w,z \rangle_{\Hi^1(\Omega, r_0, r_1)} := \scal{r_0 w}{z} + \scal{r_1 \nabla w}{\nabla z}, \quad \forall w,z\in\hk{1},
        \end{equation} 
        where $r_0$ and $r_1$ are as in \eqref{eq:r0r1weights}, the Sobolev gradient $\nabla_S\mathcal{G}_\beta \in \hk{1}$  is the weak solution $\mathcal{K}$ to the problem
    \begin{equation} 
    \label{eq:gradientellitpicbvpGT}
    \left\{
    \begin{array}{rll}
    - \nabla \cdot \left( r_1 \nabla  \mathcal{K}\right) + r_0 \mathcal{K} &= \nabla \mathcal{G}_\beta & \text{ in } \Omega \\
    \nabla \mathcal{K}\cdot \bm{n} &= 0 & \text{ on } \Gamma.
    \end{array} 
    \right.
\end{equation}
    \end{enumerate}
\end{theorem}
\begin{proof}
    Fix  $f\in\lp{2}$  and $\delta_f \in \mathcal{H}.$ We have that
    \begin{align*}
    \mathcal{G}_\beta(f + \varepsilon\delta_f) - \mathcal{G}_\beta(f) &= \varepsilon\left[\scal{P_T(f) - \Psi_T}{P_T(\delta_f)} + \beta\scal{f}{\delta_f} \right]\\ & \qquad + \frac{\varepsilon^2}{2} \left(\nrm{P_T(\delta_f)}^2 +\beta \nrm{\delta_f}^2\right), 
    \end{align*}
    yielding \eqref{eq:directionalderivativeG} after division by $\varepsilon\neq 0$ and limit considerations. Therefore, the map $\mathcal{G}_\beta^\prime(f;\cdot) \colon\mathcal{H} \to\mathbb{R}$ is linear and bounded as 
    \begin{align*}
   \abs{ \mathcal{G}_\beta^\prime(f;\delta_f) } & \leq \left( C \nrm{P_T(f) - \Psi_T} + \beta \nrm{f}\right) \nrm{\delta_f},
    \end{align*}
    since $P_T$ is bounded. 

    We argue as before for the gradients using $(\delta\bm{u},\delta\theta)$ and $(\bm{u}^\ast, \theta^\ast)$ as solutions to \eqref{eq:ISP2problemPS} and \eqref{eq:ISP2problemPA} respectively, we now recover the relation
    \begin{align*}
    &\int_\Omega \left(\int_0^T\delta\theta(\X,t;f)\dt\right)\left( \int_0^T \theta(\X,t;f)\dt - \Psi_T(\X)\right) \dX \\
    &\quad = \int_0^T \scal{\delta \theta(\cdot,t;f,\delta_f)}{\int_0^T \theta(\cdot,s;f)\ds - \Psi_T}\dt  \\
    &\quad = \int_0^T \scal{\delta\bm{u}}{ \rho \partial_{tt}\bm{u}^\ast - \nabla \cdot \sigma(\bm{u}^\ast) + \gamma T_0 \nabla \partial_t \theta^\ast }\dt\\ & \qquad \qquad + \int_0^T \scal{\delta \theta}{- \rho C_s \theta^\ast - \kappa \nabla \theta^\ast - K^\ast(\Delta \theta^\ast) - \gamma \nabla \cdot \bm{u}^\ast}\dt \\ 
    &\quad = \scal{\delta_f}{ \int_0^T g(t) \theta^\ast(\cdot,t;f)\dt},
    \end{align*}
    which reveals the formula for $\nabla\mathcal{G}_\beta$ in \eqref{eq:gradientGL2} by using \eqref{eq:directionalderivativeG} for the case $\mathcal{H}=\lp{2}.$ The formulation \eqref{eq:gradientellitpicbvpGT} follows by applying Green's theorem to find the variational formulation of \eqref{eq:gradientellitpicbvpGT} and concluding in a similar fashion as before.
\end{proof}
The stepsize in the gradient method steepest decent now takes the form
$$
\tau_n^\ast =   \frac{\scal{P_T(f_{n-1}) - \Psi_T}{P_T(\nabla_{\mathcal{H}} \mathcal{G}_\beta[f_{n-1}])} + \beta\scal{f_{n-1}}{\nabla_\mathcal{H} \mathcal{G}_\beta[f_{n-1}]}}{\nrm{P_T(\nabla_\mathcal{H}\mathcal{G}_\beta[f_{n-1}])}^2+ \beta\nrm{\nabla_\mathcal{H} \mathcal{G}_\beta[f_{n-1}]}^2},
$$
and the steps to execute the schemes are analogously as before.

\section{Numerical results and discussions} \label{sec:numerical}
In this section, we focus our attention on the numerical implementation of the proposed algorithms and their application to a numerical example. We start by discretising the adjoint problems and providing a scheme to solve the elliptic boundary value problems needed for the computation of the Sobolev gradients. Afterwards, we describe the setup of the experiment, collecting relevant pieces of data. In the last part, we discuss our findings and the performance of the aforementioned approaches to solve the inverse problems. 

\subsection{Discretisation of the adjoint problems}\label{subsec:discretisationadjointproblem}
The time discretisation for the direct problem \eqref{eq:problem} is given by \eqref{eq:weakudisc}-\eqref{eq:weakthetadisc}, as described in \Cref{subsec:timediscretisation}. It gives an iterative way to compute the solution to the direct problem at discrete time steps.  This approach  is readily adapted to produce  a discretisation scheme for the considered adjoint problems,   whose general form is given by
\begin{equation}
    \label{eq:generaladjointproblem}
    \left\{ 
    \begin{array}{rll}
    \rho \partial_{tt} \bm{u}^\ast - \nabla \cdot \sigma(\bm{u}^\ast) + \gamma T_0 \nabla \partial_t \theta^\ast &= \bm{p}^\ast(\X,t) & \quad \text{ in } Q_T \\
    - \rho C_s \partial_t \theta^\ast - \kappa \Delta \theta^\ast - K^\ast(\Delta \theta^\ast) - \gamma \nabla \cdot \bm{u}^\ast &= h^\ast(\X,t)  & \quad \text{ in } Q_T \\
    \bm{u}^\ast(\X,t) = \bm{0}, \quad \theta^\ast(\X,t) &= 0 &\quad \text{ in } \Sigma_T  \\
    \bm{u}^\ast(\X,T) = \overline{\bm{u}}^\ast_0(\X), \quad \partial_t \bm{u}^\ast(\X,T) = \overline{\bm{u}}^\ast_1(\X), \quad \theta^\ast(\X,T)&= \overline{\theta}^\ast_0(\X) & \quad \text{ in } \Omega. 
    \end{array} 
    \right. 
\end{equation} 
Those adjoint problems have the nature of a terminal-value problem and are solved backwards in time.  Hence, the occuring time derivatives at the points $t_i$ in the weak formulation of the adjoint problems are discretised using a 2-point forward Euler formula with stepsize $\tau>0$, i.e. 
\begin{align*} 
\partial_t w(\X,t_i)& \approx \delta_A w_i = \frac{w_{i+1}- w_i}{\tau}, \\ \partial_{tt}w(\X,t_i)
&\approx \delta_A^2 w_i = \frac{-1}{\tau} \left(\frac{w_{i+1}-w_i}{\tau} - \delta_A w_{i+1}\right). 
\end{align*} 
The solutions $(\bm{u}_i^\ast, \theta^\ast_i)$ to the discretised weak formulation  satisfy the following system of equations for $i=n_t-1,\dots 1,0$:
\begin{multline*}
\rho \scal{\bm{u}^\ast_i}{\bm{\varphi}} + \tau^2  \scal{ \sigma(\bm{u}^\ast_i)}{\varepsilon(\bm{\varphi})} - \tau \gamma T_0 \scal{\nabla \theta^\ast_i}{\bm{\varphi}} \\= \tau^2\scal{\bm{p}^\ast}{\bm{\varphi}}+\rho \scal{\bm{u}^\ast_{i+1}}{\bm{\varphi}} - \rho \tau \scal{\delta_A\bm{u}^\ast_{i+1}}{\bm{\varphi}} - \tau \gamma T_0 \scal{\nabla\theta^\ast_{i+1}}{\bm{\varphi}} 
\end{multline*}
and
\begin{multline*}
\rho C_s \scal{\theta^\ast_i}{\psi} + \tau \kappa \scal{\nabla \theta^\ast_i}{\nabla \psi}+\tau \scal{(k\circledast \theta^\ast)_i}{\nabla\psi} - \tau \gamma \scal{\nabla \cdot \bm{u}^\ast_i}{\psi} \\
= \tau \scal{h^\ast}{\psi} + \rho C_s \scal{\theta^\ast_{i+1}}{\psi}
\end{multline*}
for all $\bm{\varphi}\in\Hko{1}$ and $\psi \in \hko{1},$ where we set $\bm{u}_{n_t}^\ast = \overline{\bm{u}}^\ast_0, \delta_A\bm{u}^\ast_{n_t} = \overline{\bm{u}}^\ast_1$ and $\theta_{n_t}^\ast = \overline{\theta}^\ast_0.$ 

\subsection{Finite differences}
The problems \eqref{eq:gradientellitpicbvpJT}-\eqref{eq:gradientellitpicbvpIT}-\eqref{eq:gradientellitpicbvpGT} for obtaining the Sobolev gradients can be solved by using e.g.\ the finite difference method. In the experiments below, we will focus on the one-dimensional problem, i.e. $\Omega=(0,1)\subset\mathbb{R}$. These problems at hand reduce to the following form
\begin{equation}
    \label{eq:ellipticbvp1D}
    \left\{ 
    \begin{array}{rll}
    -(r_1(x) j^\prime(x))^\prime + r_0(x) j(x) &= J(x) & \quad \text{ in } \Omega \\
    j^\prime(0) = 0, \quad j^\prime(1)&=0 & \quad \text{ on } \Gamma,
    \end{array} 
    \right. 
\end{equation}
for the unknown function $j\colon\Omega \to \mathbb{R}$ and a given $J.$  The spatial discretisation of $\Omega=(0,1)$ is done in $n=n_x$ equidistant subintervals of size $h=1/n_x$ with gridpoints $0=x_0 < x_1 <\dots <x_{n-1}<x_n=1.$ The approximation of $j$ at the node $x_i$ is denoted as $j_i.$ The boundary conditions are discretised using ghost points $x_{-1}=-h$ and $x_{n+1}=1+h$ and the second order formulas
$$
\frac{j_1 - j_{-1}}{2h} = 0 \quad \text{ and }  \quad \frac{j_{n+1} - j_{n-1}}{2h}=0.
$$
That is, we set $j_{-1} = j_1$ and $j_{n+1} = j_{n-1}.$ Furthermore, let $r_\ell(x_i):=r_\ell^i$  for $\ell=0,1$ and consider the half-grid points $x_{i\pm \frac{1}{2}} = x_i \pm \frac{h}{2}.$ The differential equation for $j(x)$ is  now  discretised using the second-order central difference scheme.  For $i=1,2,\dots,n-1,$ we have 
\begin{equation}
    \label{eq:discji}
    \frac{1}{h^2} \left(j_{i-1}\left[-r_1^{i+1/2}\right]+j_i\left[r_1^{i+1/2} + r_1^{i-1/2} + h^2r_0^i\right] + j_{i+1}\left[-r_1^{i-1/2} \right] \right) = J(x_i).
\end{equation}
Rewriting the equations for $i=0$ and $i=n$ yields 
\begin{equation}
    \label{eq:discj0}
    \frac{1}{h^2} \left(j_0 \left[r_1^{1/2} + r_1^{-1/2} + h^2r_0^0\right] + j_1\left[-r_1^{1/2}-r_1^{-1/2} \right] \right) = J(x_0)
\end{equation}
  and 
\begin{equation}
    \label{eq:discjn}
    \frac{1}{h^2} \left(j_{n-1}\left[-r_1^{n+1/2} - r_1^{n-1/2}\right] + j_n\left[r_1^{n+1/2} + r_1^{n-1/2} + h^2r_0^n \right] \right) = J(x_n). 
\end{equation}
Combining the equations \eqref{eq:discji}-\eqref{eq:discj0}-\eqref{eq:discjn} produces the matrix system
\begin{equation}
\label{eq:AjJmatrix}
A \bm{j} = \bm{J}, \quad A \in \mathbb{R}^{(n+1)\times(n+1)}, \quad \bm{j},\bm{J} \in\mathbb{R}^{n+1},
\end{equation} 
where $\bm{j} = (j_0,j_1,\dots,j_{n-1},j_n)^\top \in\mathbb{R}^{n+1},$ $$\bm{J}=(J(x_0),J(x_1),\dots,J(x_{n-1}),J(x_n))^\top \in \mathbb{R}^{n+1}$$ and the matrix  $A$ is given by
$$
 A  = \frac{1}{h^2}
{\tiny{\begin{pmatrix}
    r_1^{1/2} + r_1^{-1/2} +h^2r_0^0 & -r_1^{1/2} -r_1^{-1/2} &   &  \\
    -r_1^{3/2} & r_1^{3/2} + r_1^{1/2}+ h^2r_0^1 & -r_1^{1/2} &    \\    
    &  \ddots & \ddots&  \\
    &  -r_1^{n-1/2} & r_1^{n-1/2} + r_1^{n-3/2}+ h^2 r_0^{n-1} & -r_1^{n-3/2} \\
     &  & -r_1^{n+1/2} - r_1^{n-1/2} & r_1^{n+1/2} + r_1^{n-1/2} + h^2 r_0^n 
\end{pmatrix}}},
$$
i.e.\  the elements of $A = (a_{ij})_{(n+1)\times(n+1)}$ are given by
\begin{align*}
&a_{0,0}= r_1^{1/2} + r_1^{-1/2} +h^2r_0^0, \quad  a_{0,1} = -r_1^{1/2} - r_1^{-1/2}, \\
&a_{i,i-1}=-r_1^{i+1/2},\quad  a_{i,i}= r_1^{i+1/2} + r_1^{i-1/2} + h^2r_0^i,\quad     a_{i,i+1}= -r_1^{i-1/2},\\ & \qquad  \quad i =1,\dots,n-1, \\
&a_{n,n-1} = -r_1^{n+1/2} - r_1^{n-1/2},\quad  a_{nn} =r_1^{n+1/2}+r_1^{n-1/2}+h^2r_0^n,\\
&a_{ij} =0, \quad \text{ if } |i-j|>1, \quad i,j = 0,1,\dots,n.
\end{align*}
Applying the Gershgorin circle theorem, recalling \eqref{eq:r0r1weights}, the strictly diagonally dominant matrix $A$ has strictly positive eigenvalues (as the Gershgorin discs with center $h^2r_0^{i} + r_1^{i+1/2} + r_1^{i-1/2}$ and radii $r_1^{i+1/2}+r_1^{i-1/2}$  lie at a distance $h^2r_0^i>0$ to the right of the origin)  and is invertible. Therefore, a unique solution to \eqref{eq:AjJmatrix} exists.

\subsection{Setup of the experiment}\label{sec:setup}
We consider the inverse problem \textbf{ISP1.2} as an illustrative example of the above methods. The results obtained for the other ISPs are found to be analogous. The domain $\Omega$ is taken to be $(0,1)\subset \mathbb{R}$ and the final time $T=1$ is fixed. The implementation of the finite element method for solving the direct and adjoint problems was carried out on the FEniCSx platform \cite{FEniCSx3,FEniCSx2,FEniCSx1} with the version \texttt{0.8.0} of the DOLFINx module, using P1-FEM Lagrange polynomials. An equidistant subdivision of $\Omega$ into $n_x=50$ subintervals (corresponding to the nodes $x_i=ih, i=0,1,\dots, n_x$ and the size $h=1/n_x$) has been considered. For the time discretisation of $[0,T],$ we have taken $n_t=50$ subintervals with stepsize $\tau = 1/n_t,$ and $t_j = j\tau, j =0,1,\dots,n_t$ as time points.  We have approximated the time integrals for the measurement maps and the gradients by means of the composite Simpson $1/3$ rule.

The 1D version of problem \eqref{eq:problem} reduces to the system of equations for the displacement and temperature $u,\theta\colon Q_T \to\mathbb{R}$ given by
\begin{equation}
    \label{eq:1Dproblemut}
    \left\{ 
    \begin{array}{rll}
    \rho u_{tt} - (\lambda + 2\mu) u_{xx} + \gamma \theta_x &=p(x,t) & \quad \text{ in } Q_T \\
    \rho C_s \theta_t - \kappa \theta_{xx} - k\ast \theta_{xx} + \gamma T_0 u_{tx} &= h(x,t) & \quad \text{ in } Q_T \\ 
    u(0,t)=u(1,t)=0, \, \theta(0,t)=\theta(1,t)&=0& \quad \text{ in } [0,T] \\
    u(x,0)=\overline{u}_0(x), \, u_t(x,0)=\overline{u}_1(x), \, \theta(x,0)&= \overline{\theta}_0(x) & \quad \text{ in } \Omega. 
    \end{array} 
    \right. 
\end{equation}
After nondimensionalisation of the equations, we recover the simplified system
\begin{equation}
    \label{eq:1Dproblemsimple}
    \left\{ 
    \begin{array}{rll}
    u_{tt} - u_{xx} + \theta_x &= \widetilde{p}(x,t) & \quad \text{ in } Q_T \\
    \theta_t - \theta_{xx} - (k\ast \theta_{xx}) + \epsilon u_{xt} &= \widetilde{h}(x,t) & \quad \text{ in } Q_T \\
    u(0,t)=u(1,t)=0, \, \theta(0,t)=\theta(1,t)&=0 &\quad \text{ in } [0,T] \\
    u(x,0)=\widetilde{u}_0(x), \, u_t(x,0)=\widetilde{u}_1(x), \, \theta(x,0)&= \widetilde{\theta}_0(x) & \quad \text{ in } \Omega,
    \end{array} 
    \right. 
\end{equation}
for a small positive parameter $\epsilon = (\gamma^2 T_0)/(\rho^2 C_s C_1^2)$ where $C_1 = \sqrt{(\lambda+2\mu)/\rho} $ which is set to $0.0189$ in the experiments.
This investigated setting corresponds to a copper-alloy material \cite{VanBockstal2017b} for which $G = 4.8\times 10^{10} \text{ N/m$^2$}, \nu = 0.34, \alpha_T = 16.5 \times 10^{-6} \text{ 1/K}, \kappa = 401 \text{ W/mK}, \rho = 8960 \text{ kg/m$^3$}, C_s = 385 \text{ J/kgK}$ and $T_0=293 \text{ K}, $ which gives that $\mu=G, \lambda  =2\nu G/(1-2\nu), \gamma =2G\alpha_T(1+\nu)/(1-2\nu)= 6633\times10^3$ and $\epsilon = 0.0189.$
Notice that \eqref{eq:1Dproblemsimple} corresponds to \eqref{eq:1Dproblemut} with the choices
$$
\rho =1,\quad \mu = 0,\quad \lambda= 1, \quad \gamma =1,\quad C_s = 1,\quad \kappa = 1, \quad T_0 = \epsilon. 
$$
The method of manufactured solutions is applied to compare the quality of the reconstruction with the exact source. To this end, we use the synthetic data prescribed by
\begin{equation}
    \label{eq:exactutheta}
    u(x,t) = \frac{1}{10}\left(t^3+t+1\right)\left(1-\cos(2\pi x)\right) \quad \text{ and } \quad \theta(x,t) = 2 (t^2+1)x(1-x)^2.
\end{equation}
This set of functions satisfies the Dirichlet conditions in \eqref{eq:1Dproblemut} and have initial conditions
$$
\overline{u}_0(x) = \frac{1}{10}(1-\cos(2\pi x)= \overline{u}_1(x), \quad \text{ and } \quad \overline{\theta}_0(x) = 2x(1-x)^2. 
$$
Furthermore, by substitution of \eqref{eq:exactutheta} in \eqref{eq:1Dproblemut} the exact sources $p$ and $h$ can be recovered. For  the type-{III} system of thermoelasticity, with $k=\frac{1}{100}\exp(-2t),$ they are given by
\begin{align*}
    p(x,t) &= \frac{3t}{5}\left(1-\cos(2\pi x)\right)  - \frac{2\pi^2}{5}\left(t^3+t+1\right)\cos(2\pi x) + 4x(t^2+1)(x-1) \\ & \qquad + 2(t^2+1)(x-1)^2\\
    h(x,t) &= - 12.06 t^{2} x + 0.01134 \pi t^{2} \sin{\left(2 \pi x \right)} + 8.04 t^{2} + 4 t x \left(x - 1\right)^{2} \\ & \qquad + 0.06 t x - 0.04 t - 12.09 x + 0.09 x \exp({- 2 t}) + 0.00378 \pi \sin{\left(2 \pi x \right)} \\& \qquad + 8.06 - 0.06 \exp({- 2 t})
\end{align*}
The decomposition of $p(x,t)=g(t)f(x) + r(x,t)$ consists of 
\begin{equation*}
g(t) = \frac{-2\pi^2}{5}(t^2+t+1) \quad \text{ and } \quad r(x,t) =p(x,t) -g(t)f(x),
\end{equation*} 
where we will consider two different choices of the exact function $f$ we want to recover, i.e.
\begin{equation}
    \label{eq:fchoices}
    f^0(x) = x\sin(2\pi x) \quad \text{ and } \quad f^1(x) = x\sin(2\pi x) + 0.2. 
\end{equation}
The $\lp{2}$-norms of $f^0$ and $f^1$ are
\begin{equation} 
\label{eq:fnormexperimentisp12}
\nrm{f^0} 
\approx 0.40041739822 \quad \text{ and } \quad \nrm{f^1} 
\approx 0.3696919197.
\end{equation} 
Note that both functions in \eqref{eq:fchoices} are asymmetric as compared to the ones considered in \cite{VanBockstal2014, VanBockstal2017b}.

In the experiments, we will check if the algorithms produce an approximation with $\Leb^2$-norm close to $0.400$ for $f^0$ and close to $0.370$ for $f^1.$ The exact measurement $\chi_T(x)$ is readily computed as
\begin{equation}
    \label{eq:exactchiT}
    \chi_T(x) = \int_0^T u(x,t)\dt = \frac{ 7}{40}(1-\cos(2\pi x)).
\end{equation}

Remark that by this construction,  the function $g \in \mathcal{C}_{1} \subset \mathcal{C}_0,$ the conditions of \Cref{lem:forward_problem} are met, and that \Cref{thm:isp12} guarantees the uniqueness of a solution to the inverse problem. For an appropriate choice of $\alpha_1$, the Landweber scheme is convergent by \Cref{thm:convergenceLandweberISP12}. Moreover, \Cref{thm:ITE12uniquemins,thm:gateauxIT11} yield information of the corresponding minimisation problems associated with the cost functional $\mathcal{I}_\beta.$ At the start of the experiment, only the initial and boundary data, the measurement $\chi_T$ and the functions $g(t), r(x,t)$ and $h(x,t)$ are considered as available data. 

In case of  a noisy measurement, with noise level $\widetilde{e} \in \{1\%,3\%, 5\%\}$, the measurement $\chi_T$ is perturbed to $\chi_T^e$ using additive white Gaussian noise, i.e.\ 
$$
\chi_T^e = \chi_T + \chi_{\text{noise},e}, \quad \chi_{\text{noise},e} \sim \mathcal{N}\left(0, \sigma_e^2 \right), 
$$
where $\chi_{\text{noise},e}$ is noise generated from a normal distribution with zero mean and standard deviation $\sigma_e = \widetilde{e}\max_{x \in \Omega}\abs{\chi_T(x)}.$ In order to avoid inverse crimes, the noise is added to the exact measurement which is at first considered on a finer grid $(n_x=1000)$ and afterwards projected back to the working grid $(n_x=50).$ The noisy measurements for the different noise levels $\widetilde{e}$ are shown in \Cref{fig:noisymeasurements},  with corresponding numerical values
$$
\nrm{\chi_T - \chi_T^e} = e(\widetilde{e}) \approx \begin{cases}
    0.00254  & \text{ if }  \widetilde{e} = 1\%, \\ 
    0.00763  & \text{ if } \widetilde{e} = 3\%,\\
    0.01272  & \text{ if } \widetilde{e} = 5\%. 
\end{cases}
$$
\begin{figure}[t]
    \centering
    \includegraphics[width=0.9\linewidth]{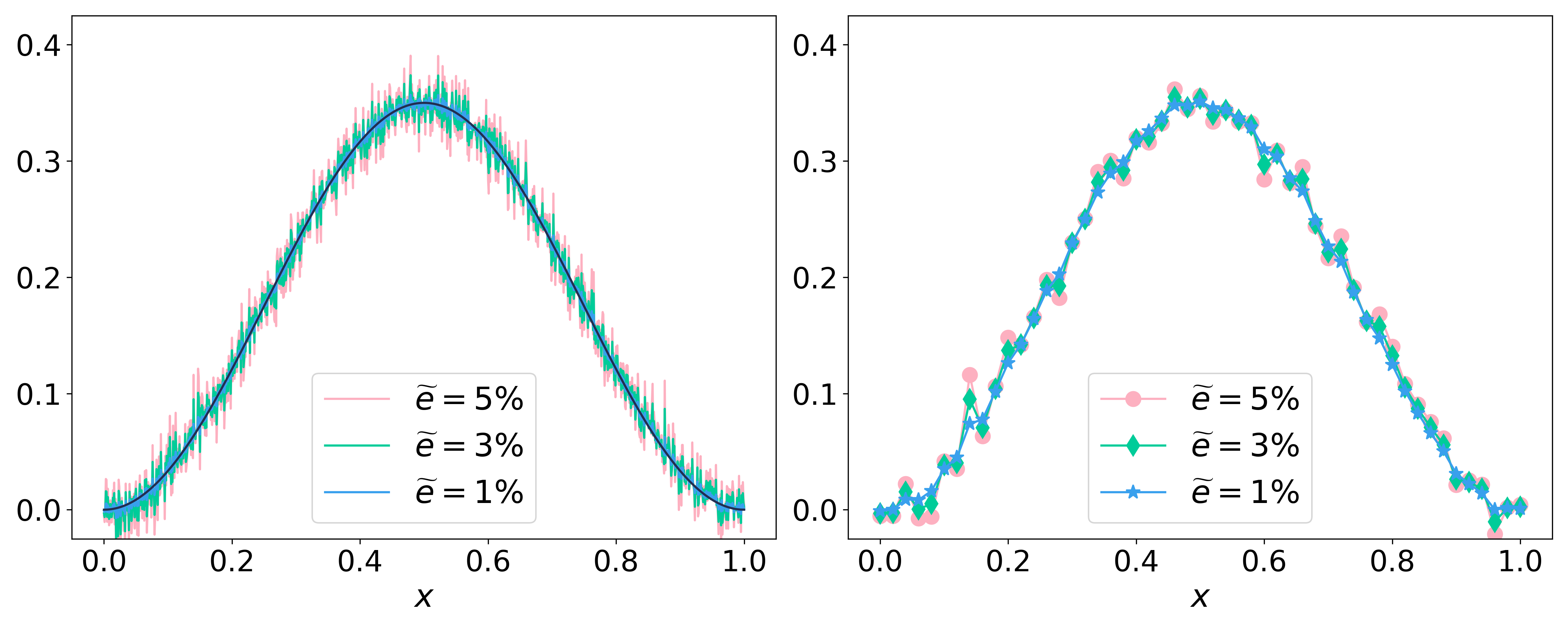}
    \caption{Exact $\chi_T(x)$ and noisy measurements $\chi_T^e(x)$ on the finer grid (left) and their projection on the working grid (right) for noise levels in $\{1\%,3\%,5\%\}.$}
    \label{fig:noisymeasurements}
\end{figure}

The scalar $r$ in the stopping criterion \eqref{ip1.2:stopping} is taken to be $1.001.$
The default regularisation parameters used in the determination of the Sobolev gradient are set to $r_0=1$ and $r_1=1/100$ by trial and error.  A more procedural approach for setting these values should be investigated in the future.

The numerical results for the reconstruction problem addressed by the proposed methods for $\textbf{ISP1.2}$ are discussed in the following sections. The initial guess is taken to be zero for all methods unless otherwise stated.  The quality of the reconstructed $f_K$ will be assessed by means of the relative error $e_r = \nrm{f-f_K}/\nrm{f},$ the data fidelity (DF) $\nrm{N_T(f_K) - X_T^e},$ the norm/penalty (P) $\nrm{f_K}$ and the value of the cost functional $\mathcal{J}_\beta[f].$ The obtained values for the relative errors for the different methods are compared in \Cref{tab:summary}.

\subsection{Landweber method} 

The results for the reconstructed sources $f_K$ obtained by applying the Landweber procedure to the test case described in \Cref{sec:setup} are collected in \Cref{fig:L-f0-relerror} and \Cref{fig:L-f0-K} for the  source $f^0$ and in \Cref{fig:L-f1-relerror} and \Cref{fig:L-f1-K} for $f^1.$ The relation between the choice of parameter value $0 < \alpha:=\alpha_1 < \nrm{N_T}^{-2}_{\mathcal{L}\left(\lp{2}, \lp{2}\right)}$ and the relative error $e_r$ is depicted in the right panels of \Cref{fig:L-f0-relerror} and \Cref{fig:L-f1-relerror}. The relative error decreases when increasing the relaxation parameter, until a critical value 
around $\alpha\approx6.1.$ Indeed, according to \Cref{thm:convergenceLandweberISP12}, the parameter $\alpha$ can not be chosen arbitrarily large. The different curves in the left panels correspond to different values of the noise. In the noise free case, the iterations are halted at the maximum number of $200$ iterations, whilst in the presence of noise the Mozorov discrepancy rule \eqref{eq:stoppingisp12} is used. Increasing the value of $\alpha$ in the range of convergence helps to reduce the relative error, yielding better approximations. However, because of the near horizontal graph, only a small improvement is visible. A drop in relative error is observed just before the critical point. Taking this value for $\alpha = \argmin\{e_r(a) \mid a \in (0,\nrm{N_T}^{-2}_{\mathcal{L}(\lp{2},\lp{2})})\}$ yields reconstructions with smallest relative error, which are plotted in the left panels of \Cref{fig:L-f0-relerror} and \Cref{fig:L-f1-relerror}.

\begin{figure}[t]
    \centering 
    \includegraphics[width=\linewidth]{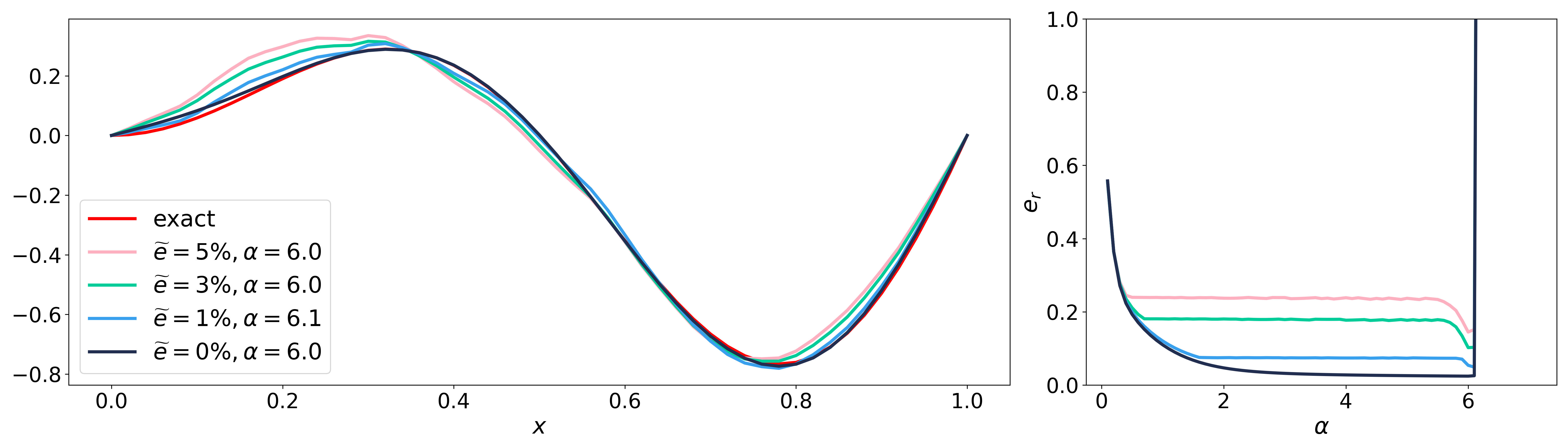}
    \caption{Reconstruction using the Landweber scheme for $f^0$ (left) with corresponding relative errors $e_r$ in function of the relaxation parameter $\alpha$ (right) for \textbf{ISP1.2}.}
    \label{fig:L-f0-relerror}
\end{figure} 
\begin{figure}[t]
    \centering 
    \includegraphics[width=\linewidth]{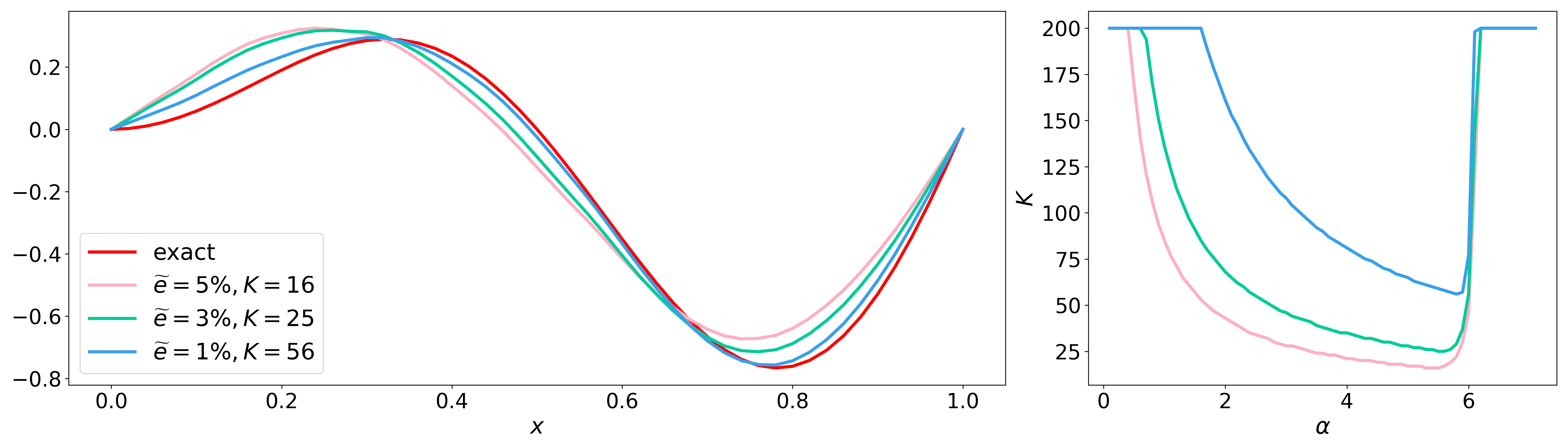}
    \caption{Reconstruction using the Landweber scheme for $f^0$ (left) with corresponding iterations $K$ in function of the relaxation parameter $\alpha$ (right) for \textbf{ISP1.2}.}
    \label{fig:L-f0-K}
\end{figure}
\begin{figure}[t]
    \centering 
    \includegraphics[width=\linewidth]{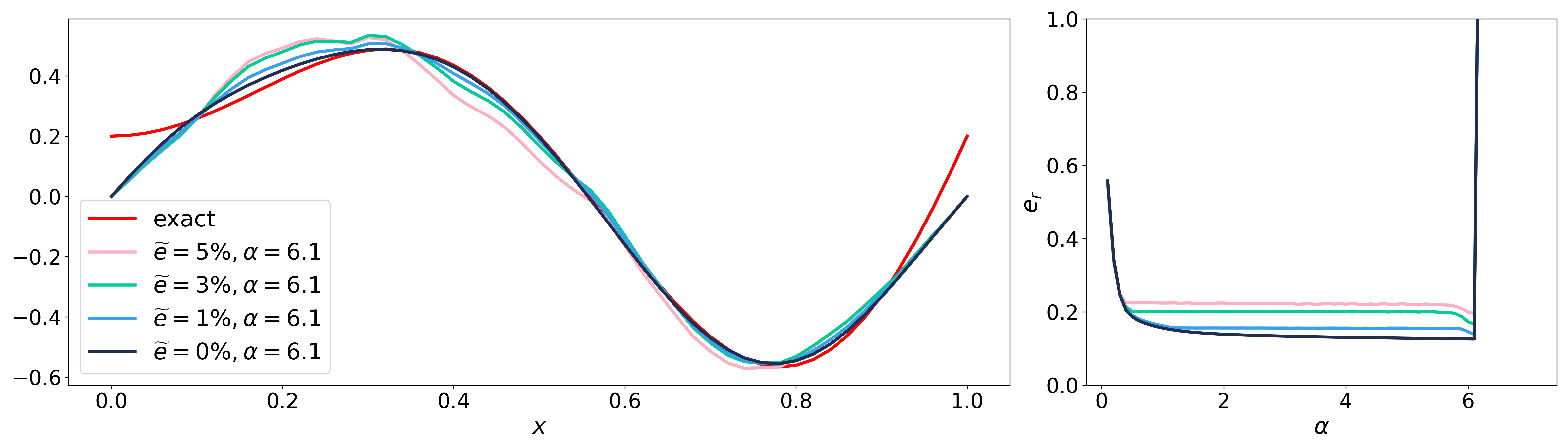}
    \caption{Reconstruction using the Landweber scheme for $f^1$ (left) with corresponding relative errors $e_r$ in function of the relaxation parameter $\alpha$ (right) for \textbf{ISP1.2}.}
    \label{fig:L-f1-relerror}
\end{figure} 
\begin{figure}[t]
    \centering 
    \includegraphics[width=\linewidth]{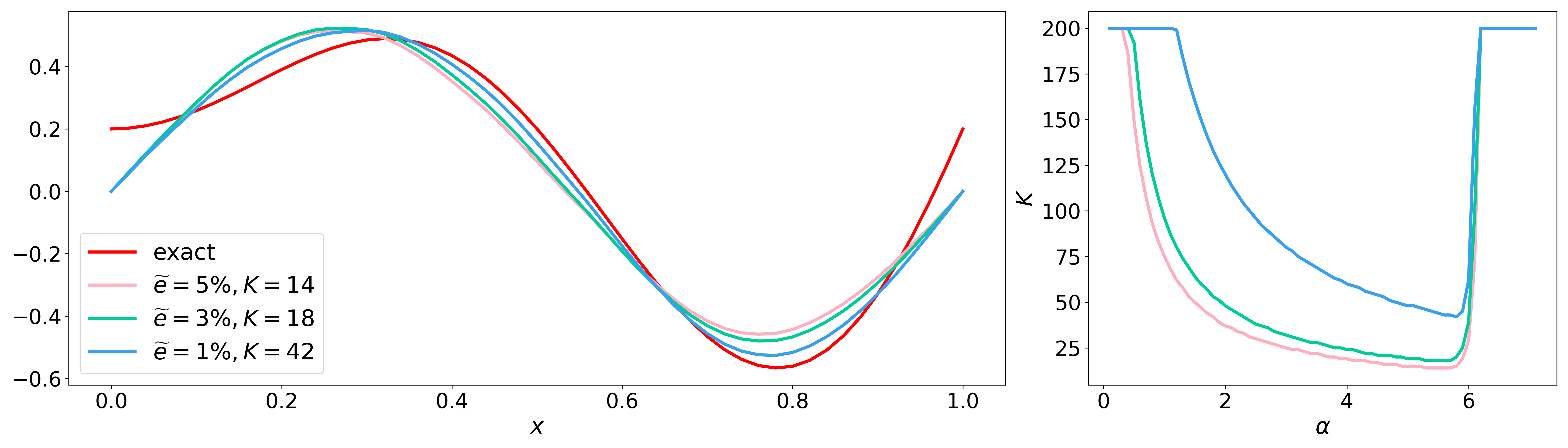}
    \caption{Reconstruction using the Landweber scheme for $f^1$ (left) with corresponding iterations $K$ in function of the relaxation parameter $\alpha$ (right) for \textbf{ISP1.2}.}
    \label{fig:L-f1-K}
\end{figure}

\begin{table}[t]
    \centering
    \caption{Result for the Landweber iterations for $f^0$ and $f^1$ for \textbf{ISP1.2}.}
    {\footnotesize{ 
    \begin{tabular*}{\linewidth}{@{\extracolsep{\fill}}  c|c c c  c|c c c c }\hline 
    &$\alpha$ & $e_r$ & $K$ & $DF$ & $\alpha$& $e_r$ & $K$ & $DF$ \\ \hline 
    $\widetilde{e}$& \multicolumn{4}{c|}{$0\%$} & \multicolumn{4}{c}{$1\%$}  \\
    \multirow{2}{*}{ $f^0$} & \multirow{2}{*}{6.0} & \multirow{2}{*}{0.0241} & \multirow{2}{*}{200} & \multirow{2}{*}{1.1740e-04} & 6.1& 0.0487 & 198 & 2.6909e-03\\
                             & & & &  & 5.8& 0.0730 & 56 & 2.6761e-03 \\
    \multirow{2}{*}{$f^1$} & \multirow{2}{*}{6.1}& \multirow{2}{*}{0.1255} & \multirow{2}{*}{200} & \multirow{2}{*}{7.1335e-04} &6.1&0.1381 & 155 & 2.6992e-03 \\ 
                           & & &  &&5.8&0.1551 & 42 & 2.6866e-03\\\hline
    $\widetilde{e}$& \multicolumn{4}{c|}{$3\%$} & \multicolumn{4}{c}{$5\%$}  \\
    \multirow{2}{*}{ $f^0$}& 6.0& 0.1020 & 56 & 7.9720e-03 & 6.0&0.1446 & 46 & 1.3361e-02 \\
                           &5.5  & 0.1784 & 25 & 8.0583e-03 &5.3& 0.2369 & 16 & 1.3229e-02 \\
    \multirow{2}{*}{$f^1$}&6.1 & 0.1661 & 99 & 8.0705e-03  &6.1&0.1956 & 72 & 1.3493e-02\\ 
                          &5.3&0.2007 & 18 & 7.8825e-03 &5.3& 0.2213 & 14 & 1.3045e-02\\\hline
    \end{tabular*}
    }}
    \label{tab:L}
\end{table}

Based on the relative errors, it seems beneficial to pick the value of $\alpha$ as large as possible, keeping it below the critical value. However, taking into account the computational time, as expressed by the number $2K$ of calls to the direct problem, reveals that picking $\alpha$ slightly smaller entails a substantial reduction in computational cost as fewer iterations are needed, albeit a moderate increase in relative error. This property is observed for the different noisy experiments. The reconstruction based on the fewest iterations and the relation between the values of $K$ and $\alpha$ are presented in \Cref{fig:L-f0-K} and \Cref{fig:L-f1-K}. The corresponding values of the stopping index $K$, the relative error $e_r$ and the values of $\alpha$ are tabulated in \Cref{tab:L}. We conclude that the reconstruction is adequately achieved (for both cases), yielding better results as the noise fades away, and if the boundary values are known a priori. The obtained reconstructions provide a good fitting of the measurement data. The behaviour at the boundary is enforced by the method as the values of the initial guess remain unchanged during the iterations.

\begin{remark}\label{rem:isp11nonsymetric}
\Cref{fig:isp11landweberf0} shows the result for  the Landweber algorithm for \textbf{ISP1.1} with the same setup as described above for the reconstruction of the non-symmetric source $f^0$ and hence shows the improvement compared to limitation of the work done in \cite{VanBockstal2017b}. The shape and form of the curve of the relative errors versus the relaxation parameter is similar to those considered for \textbf{ISP1.2}. 
The application of the other reconstruction methods to \textbf{ISP1.1} yielded analogous conclusions as those for \textbf{ISP1.2} as well.
\begin{figure}[t]
\centering 
 \includegraphics[width=\linewidth]{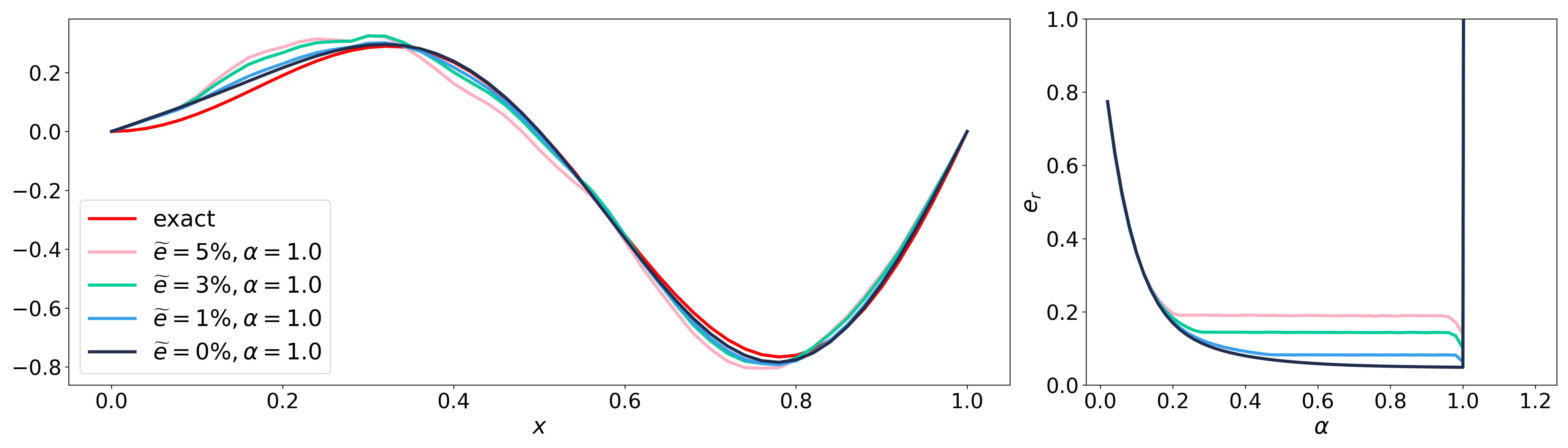}
\caption{Reconstruction for \textbf{ISP1.1} using the Landweber scheme for $f^0$ (left) with corresponding relative erros $e_r$ in function of the relaxation parameter (right).}
\label{fig:isp11landweberf0}
\end{figure} 
\end{remark}

\subsection{Steepest descent gradient method} 
The results for the standard steepest descent gradient algorithm developed in \Cref{sec:gradientmethods} are shown in \Cref{fig:L2GM-f0} and \Cref{fig:L2GM-f1} for the $\Leb^2$-gradient $\nabla \mathcal{I}_\beta$ for the exact source $f^0$ (which is zero at the boundary) and $f^1$ (which is nonzero at the boundary), for various levels of noise. The middle panels show the curve of the datafit versus the norm of $f_K,$ parameterised by the value of $\beta \in[0,0.1],$ which is crossed from the top left corner to the bottom right corner. For the same range of $\beta$, the right panels indicate the relative error of the reconstruction. From these figures, we observe that the best approximation in terms of datafit, norm and relative error are found when $\beta=0.$ The rationale being that the extra term $\beta\nrm{f}^2$ in the cost functional (if present) favours small sizes of $f,$ which impacts the fitting of the measurement.  The left panels then show the reconstructed source for the different levels of noise for this value of $\beta=0.$ The relative errors, fidelity and norm of the reconstructions are reported in \Cref{tab:L2GM}. When more noise is considered, a less accurate recovery is obtained. Notice that in both figures, the values of the reconstruction at the boundary match those of the initial guess. From \Cref{fig:L2GM-f1} one observes that the approximation near the boundary remains at the fixed value of zero, which is an artefact of the method. The iterations were halted when the cost functional increased. Increasing the value of $\beta$ leads to an increase in relative error, a larger deficit in the data fitting and a decrease in norm. This behaviour is observed for all the gradient methods, and is reasonable in our setting, as enforcing a penalty whenever $\nrm{f}$ is large in the cost functional tends to favour small, in $\Leb^2$-norm, approximations. More sophisticated penalty terms can be considered in different settings.

\begin{figure}[t]
    \centering 
    \includegraphics[width=\linewidth]{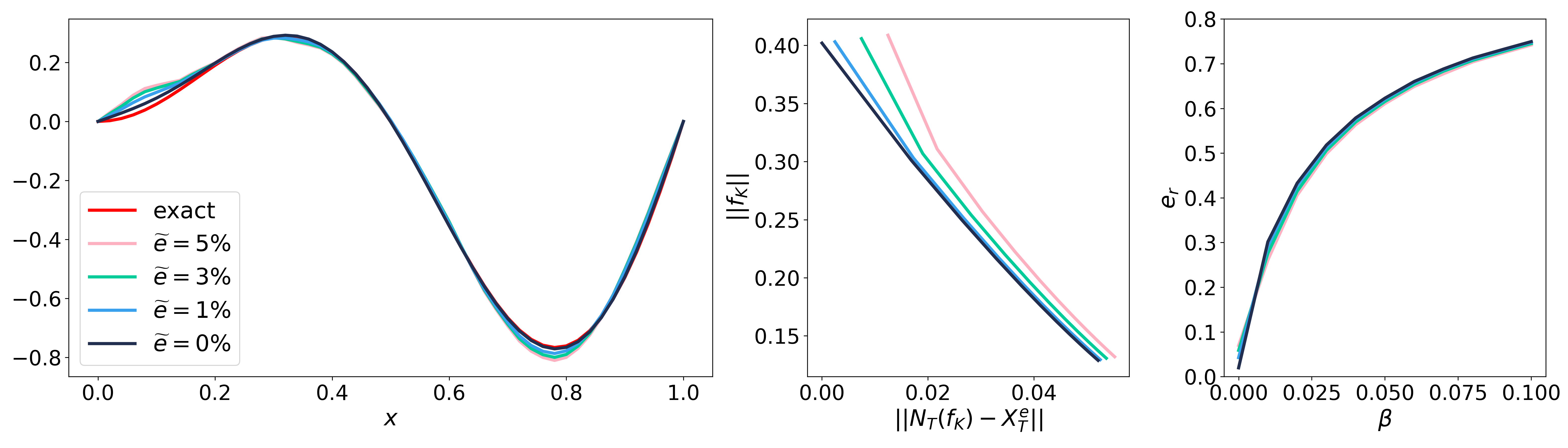}
    \caption{Reconstruction using the $\Leb^2$-gradient with the steepest descent algorithm for $f^0$ for $\beta=0$ (left), penalty versus fitting of the data (middle) and relative errors in function of $\beta$ (right) for \textbf{ISP1.2}.} 
    \label{fig:L2GM-f0}
\end{figure} 
\begin{figure}[t]
    \centering 
    \includegraphics[width=\linewidth]{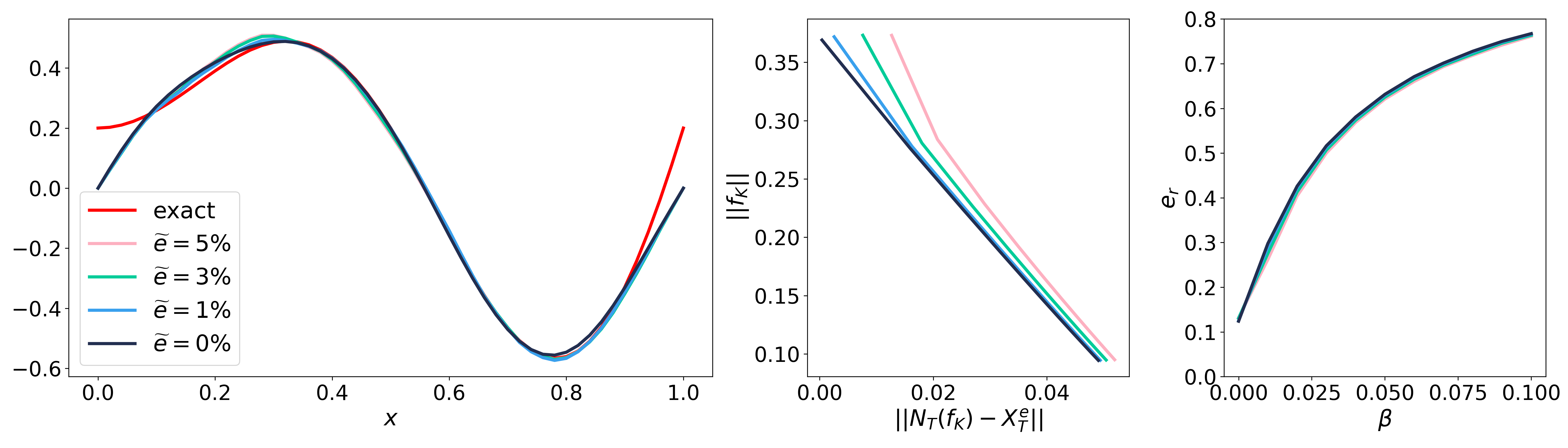}
    \caption{Reconstruction using the $\Leb^2$-gradient with the steepest descent algorithm for $f^1$ for $\beta=0$ (left), penalty versus fitting of the data (middle) and relative errors in function of $\beta$ (right) for \textbf{ISP1.2}.} 
    \label{fig:L2GM-f1}
\end{figure} 

\begin{table}[t]
    \centering 
    \caption{Results for minimisation of the functional \eqref{eq:IT12} using the $\Leb^2$-gradient \eqref{eq:gradientIL2} and the steepest descent algorithm    
    for $f^0$ and $f^1$ for \textbf{ISP1.2}, where $e_r$ is the relative error, $P=\nrm{f_K}$ and  $DF = \nrm{N_T(f_K)-X_T^e}$ where $f_K$ is the obtained approximation ($\beta=0.0$).}
    \footnotesize{ 
    \begin{tabular*}{\linewidth}{@{\extracolsep{\fill}}  c|c c c  c|c c c c }\hline 
    & $K$ & $e_r$ & $P$ & $DF$ &$K$ &$e_r$ & $P$ & $DF$ \\ \hline 
    $\widetilde{e}$& \multicolumn{4}{c|}{$0\%$} & \multicolumn{4}{c}{$1\%$}  \\
    $f^0$ &200&  0.0197 & 0.4019 & 3.3728e-05 &29&  0.0425 & 0.4031 & 2.4804e-03 \\
    $f^1$&200 & 0.1240 & 0.3691 & 3.9797e-04 &84&  0.1253 & 0.3719 & 2.5426e-03 \\ \hline
    $\widetilde{e}$&\multicolumn{4}{c|}{$1\%$} & \multicolumn{4}{c}{$5\%$} \\ 
    $f^0$&19 &0.0587 & 0.4058 & 7.4750e-03 &15& 0.0698 & 0.4087 & 1.2470e-02\\
    $f^1$&33 &  0.1308 & 0.3732 & 7.5871e-03 &23& 0.1320 & 0.3731 & 1.2667e-02 \\
    \hline 
    \end{tabular*}
   }     
   
    \label{tab:L2GM}
\end{table} 
The results for the steepest descent algorithm using the Sobolev gradient $\nabla_S \mathcal{I}_\beta$ are presented in \Cref{fig:SGM-f0} and \Cref{fig:SGM-f1} for the reconstruction of $f^0$ and $f^1,$ respectively. The best compromise between the datafit and the norm of the reconstruction was found for $\beta=0$ in both cases. The numerical values are reported in \Cref{tab:SGM}.  Compared to the results of the $\Leb^{2}$-gradient, the boundary values are free to be updated using the Sobolev gradient and give a good approximation of the left boundary value, whilst a reasonable one at the right. Although the relative errors for the $\Leb^2$-gradient seem to be better, we remark that the exact source has a steep increase near the right boundary and the fixed value at the boundary is able to help in the iteration. This extra assumption is not needed for the Sobolev variant.

\begin{figure}[t]
    \centering 
    \includegraphics[width=\linewidth]{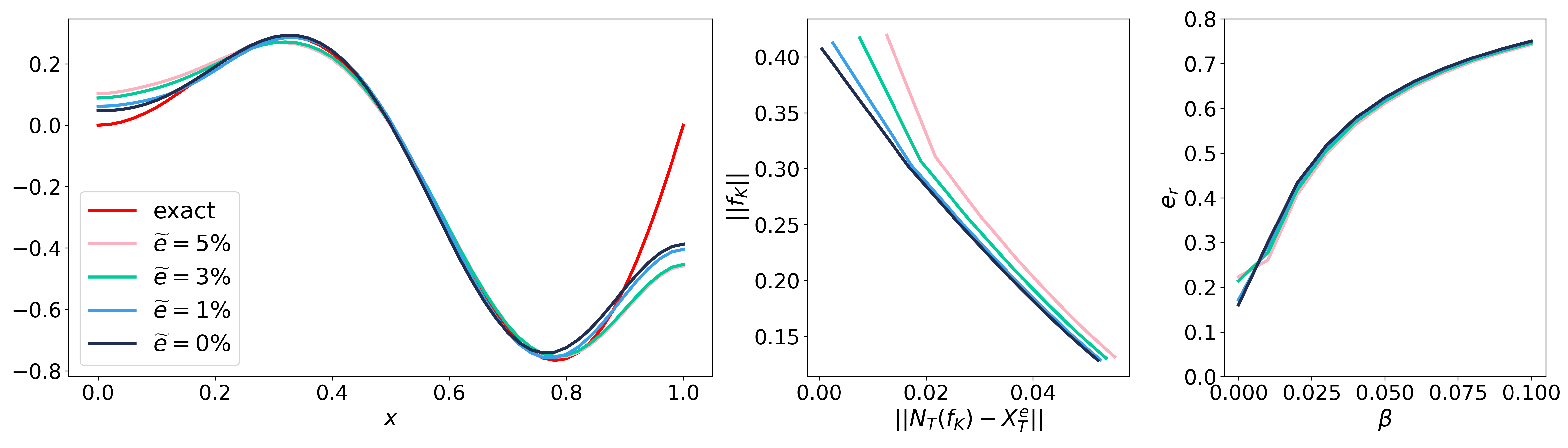}
    \caption{Reconstruction using the Sobolev gradient with the steepest descent algorithm for $f^0$ for $\beta=0$ (left), penalty versus data fitting (middle) and relative errors in function of $\beta$ (right) for \textbf{ISP1.2}.} 
    \label{fig:SGM-f0}
\end{figure} 
\begin{figure}[t]
    \centering 
    \includegraphics[width=\linewidth]{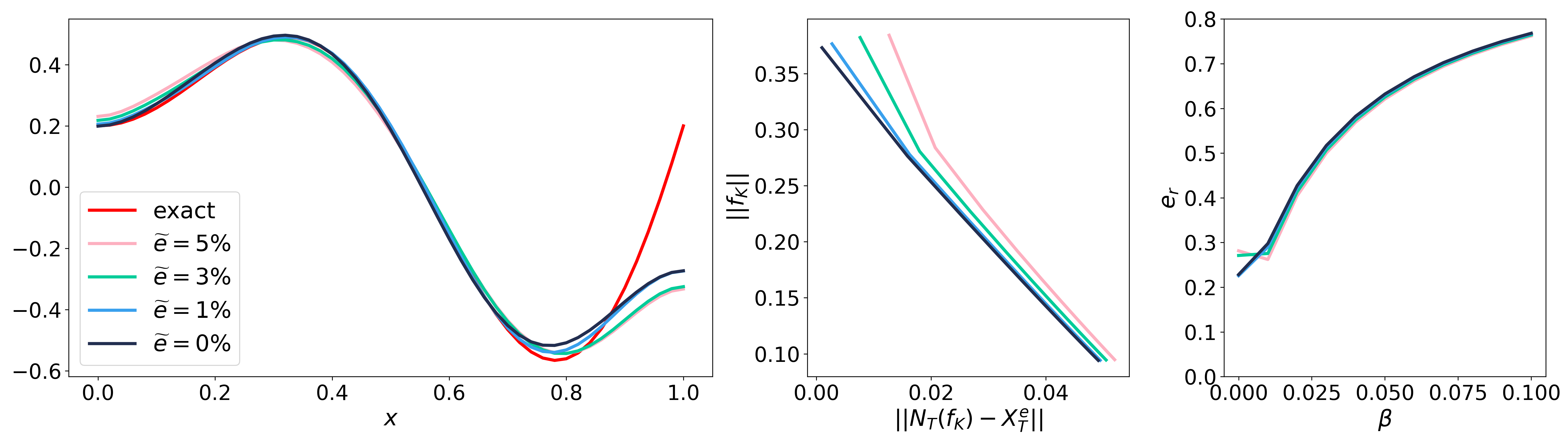}
    \caption{Reconstruction using the Sobolev gradient with the steepest descent algorithm for $f^1$ for $\beta=0$ (left), penalty versus data fitting (middle) and relative errors in function of $\beta$ (right) for \textbf{ISP1.2}.} 
    \label{fig:SGM-f1}
\end{figure} 
\begin{table}[t]
    \centering 
    \caption{Results for minimisation of the functional \eqref{eq:IT12} using the Sobolev gradient from \eqref{eq:gradientellitpicbvpIT}  and the steepest descent algorithm 
    for $f^0$ and $f^1$ for \textbf{ISP1.2}, where $e_r$ is the relative error, $P=\nrm{f_K}$ and  $DF = \nrm{N_T(f_K)-X_T^e}$ where $f_K$ is the obtained approximation ($\beta=0.0$).}
   \footnotesize{ 
    \begin{tabular*}{\linewidth}{@{\extracolsep{\fill}}  c|c c c  c|c c c c }\hline 
    & $K$ & $e_r$ & $P$ & $DF$ &$K$ &$e_r$ & $P$ & $DF$ \\ \hline 
    $\widetilde{e}$& \multicolumn{4}{c|}{$0\%$} & \multicolumn{4}{c}{$1\%$}  \\
    $f^0$ &200&0.1605 &0.4073 & 5.1246e-04 &113& 0.1718 & 0.4126 &2.5445e-03\\
    $f^1$&200 & 0.2283& 0.3733 &9.8183e-04 &200 & 0.2258 & 0.3767 & 2.7287e-03 \\ 
    \hline
    $\widetilde{e}$ &\multicolumn{4}{c|}{$3\%$} & \multicolumn{4}{c}{$5\%$} \\ 
    $f^0$ &28 & 0.2145 & 0.4174 & 7.5863e-03 &20&  0.2236 & 0.4195 & 1.2635e-02\\
    $f^1$ &62 &0.2710 & 0.3824 & 7.6283e-03 &38& 0.2815 & 0.3844 & 1.2652e-02 \\
    \hline 
    \end{tabular*}
 }     
    \label{tab:SGM}
\end{table} 

\subsection{Conjugate gradient method}
The results when applying the method described in \Cref{sec:gradientmethods} using the conjugate gradients, are shown in \Cref{fig:L2CGM-f0} and \Cref{fig:L2CGM-f1} for the $\Leb^2$-gradient approach. The numerical values are reported in \Cref{tab:L2CGM}. The quality of the reconstruction is similar to that of the steepest descent approach; however, we notice an improvement for higher noise levels. In this experiment also $\beta=0$ was the best value to consider, which is reasonable as no penalty term is inherent to the problem. The iterations stopped after $K=3$ steps, which is, compared to the steepest descent method, a remarkable improvement. This method also suffers from the fixation of the iterates at the boundary, as is visible from the left panel of \Cref{fig:L2CGM-f1}. 

\begin{figure}[t]
    \centering 
    \includegraphics[width=\linewidth]{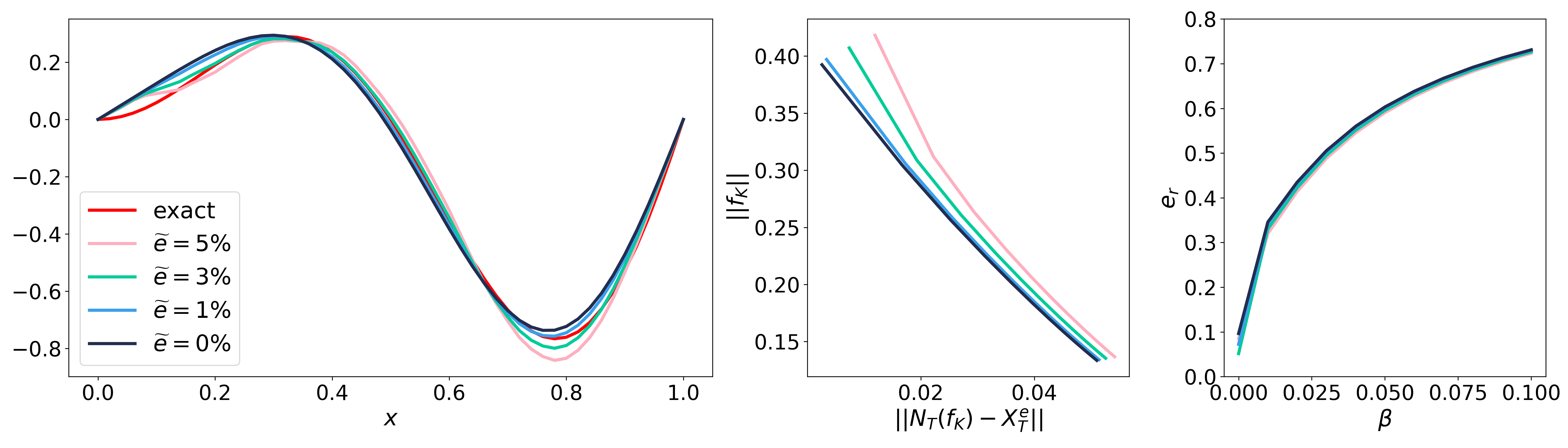}
    \caption{Reconstruction using the $\Leb^2$-gradient with the conjugate gradient algorithm for $f^0$ for $\beta=0$ (left), penalty versus data fitting (middle), relative errors in function of $\beta$ (right) for \textbf{ISP1.2}.} 
    \label{fig:L2CGM-f0}
\end{figure} 
\begin{figure}[t]
    \centering 
    \includegraphics[width=\linewidth]{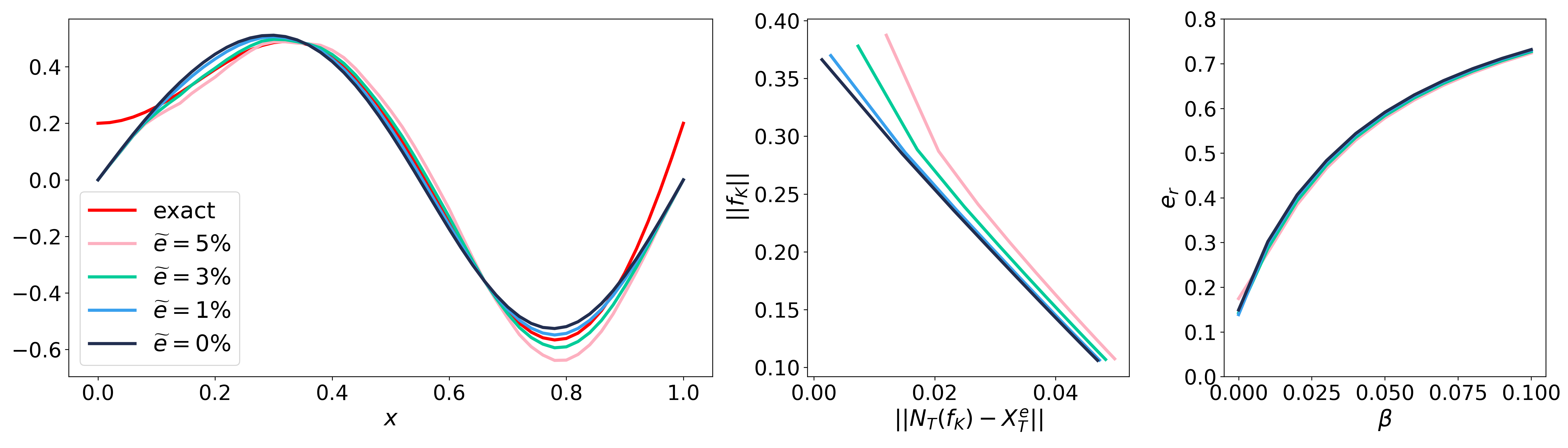}
    \caption{Reconstruction using the $\Leb^2$-gradient with the conjugate gradient algorithm for $f^1$ for $\beta=0$ (left), penalty versus data fitting (middle), relative errors in function of $\beta$ (right) for \textbf{ISP1.2}.} 
    \label{fig:L2CGM-f1}
\end{figure} 

\begin{table}[t]
    \centering 
    \caption{Results for minimisation of the functional \eqref{eq:IT12} using the $\Leb^2$-gradient \eqref{eq:gradientIL2}  and the conjugate gradient algorithm 
    for $f^0$ and $f^1$ for \textbf{ISP1.2}, where $e_r$ is the relative error, $P=\nrm{f_K}$ and  $DF = \nrm{N_T(f_K)-X_T^e}$ where $f_K$ is the obtained approximation $(K=3, \beta=0.0)$.}
    \label{tab:L2CGM}
    \footnotesize{ 
    \begin{tabular*}{\linewidth}{@{\extracolsep{\fill}}  c| c c  c| c c c }\hline 
    &$e_r$ & $P$ & $DF$ & $e_r$ & $P$ & $DF$ \\ \hline 
    $\widetilde{e}$& \multicolumn{3}{c|}{$0\%$} & \multicolumn{3}{c}{$1\%$}  \\ 
    $f^0$ & 0.0964 & 0.3924 & 2.5929e-03 & 0.0723 & 0.3971 & 3.3966e-03 \\
    $f^1$& 0.1492 & 0.3662 & 1.3347e-03 & 0.1388 & 0.3698 & 2.7948e-03 \\ \hline 
    $\widetilde{e}$& \multicolumn{3}{c|}{$3\%$} & \multicolumn{3}{c}{$5\%$} \\
    $f^0$ & 0.0515 & 0.4072 & 7.3974e-03 &0.0852 & 0.4183 & 1.1938e-02\\
    $f^1$  &0.1429 & 0.3779 & 7.3119e-03 &0.1747 & 0.3874 & 1.1986e-02 \\\hline
        \end{tabular*}
    }     
\end{table} 

Finally, the conjugate gradient method applied to the Sobolev version $\nabla_S\mathcal{I}_\beta$ is presented in \Cref{fig:SCGM-f0} and \Cref{fig:SCGM-f1}, and the numerical values are shown in \Cref{tab:SCGM}. The strength of this approach lies in the fact that only $K=2$ iterations are needed to produce an accurate reconstruction (as similar to the standard conjugate gradient approach), but without additional assumptions and under various noise levels. We observe that the relative errors are smaller for the experiment with $f^1.$ 

\begin{figure}[t]
    \centering 
    \includegraphics[width=\linewidth]{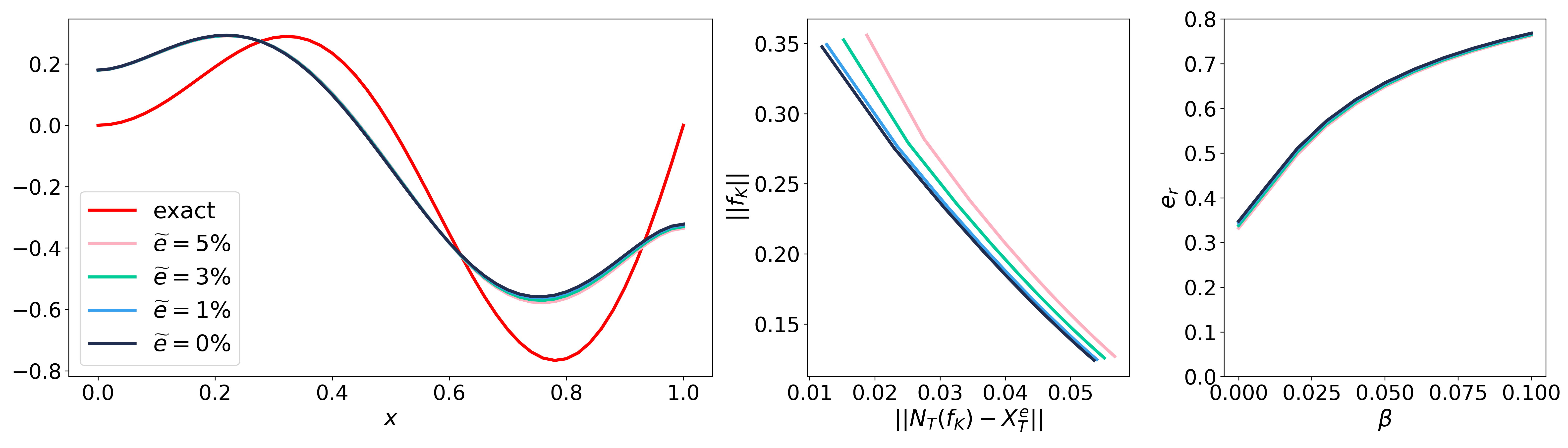}
    \caption{Reconstruction using the Sobolev gradient with the conjugate gradient algorithm for $f^0$ for $\beta=0$ (left), penalty versus data fitting (middle), relative errors in function of $\beta$ (right) for \textbf{ISP1.2}.} 
    \label{fig:SCGM-f0}
\end{figure} 
\begin{figure}[t]
    \centering 
    \includegraphics[width=\linewidth]{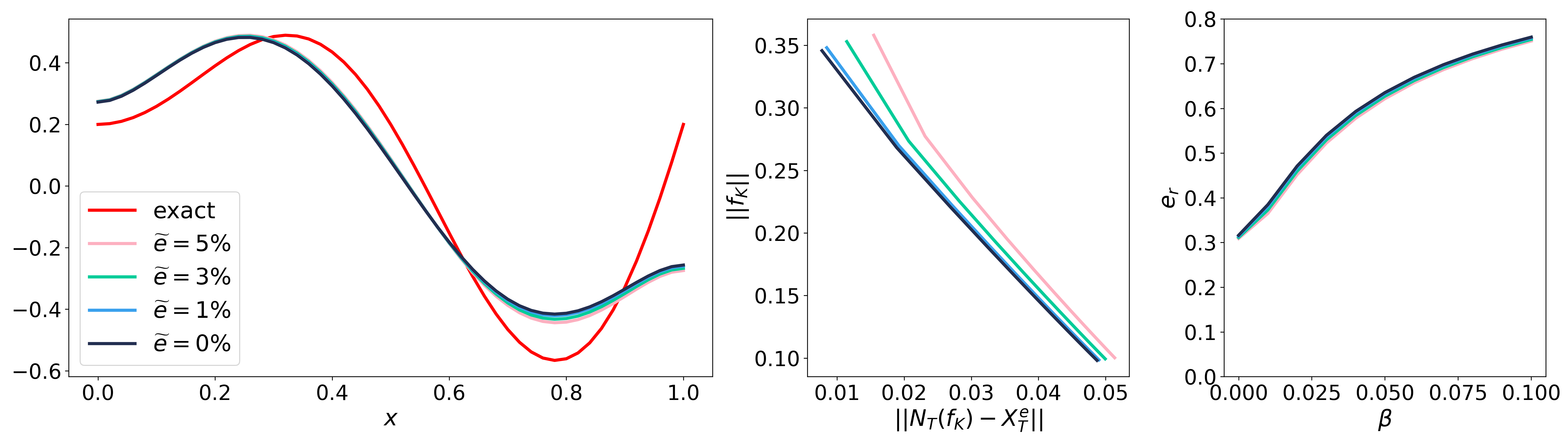}
    \caption{Reconstruction using the Sobolev gradient with the conjugate gradient algorithm for $f^1$ for $\beta=0$ (left), penalty versus data fitting (middle), relative errors in function of $\beta$ (right) for \textbf{ISP1.2}.}
    \label{fig:SCGM-f1}
\end{figure} 
\begin{table}[t]
 \centering 
\caption{Results for minimisation of the functional \eqref{eq:IT12} using the Sobolev gradient from \eqref{eq:gradientellitpicbvpIT}  and the conjugate gradient algorithm     for $f^0$ and $f^1$ for \textbf{ISP1.2}, where $e_r$ is the relative error, $P=\nrm{f_K}$ and  $DF = \nrm{N_T(f_K)-X_T^e}$ where $f_K$ is the obtained approximation ($K=2, \beta=0.0$).} 
   \footnotesize{ 
    \begin{tabular*}{\linewidth}{@{\extracolsep{\fill}}  c| c c  c| c c c }\hline 
    & $e_r$ & $P$ & $DF$ &  $e_r$ & $P$ & $DF$ \\ \hline 
    $\widetilde{e}$& \multicolumn{3}{c|}{$0\%$} & \multicolumn{3}{c}{$1\%$}  \\ 
    $f^0$&  0.3476 & 0.3478 & 1.1883e-02& 0.3444 & 0.3494 & 1.2608e-02 \\
    $f^1$& 0.3162 & 0.3456 & 7.7647e-03&0.3143 & 0.3480 & 8.4741e-03\\\hline
    $\widetilde{e}$& \multicolumn{3}{c|}{$3\%$} & \multicolumn{3}{c}{$5\%$}  \\ 
    $f^0$&  0.3381 & 0.3527 & 1.5203e-02 &0.3320 & 0.3560 & 1.8765e-02 \\
    $f^1$& 0.3113 & 0.3529 & 1.1460e-02 & 0.3090 & 0.3580 & 1.5477e-02 \\\hline 
        \end{tabular*}
   }     
    \label{tab:SCGM}
\end{table}

\section{Conclusion}\label{sec:conclusion}
In this work, three inverse problems related to the recovery of an unknown space-dependent source function in the load or heat force of a thermoelastic system of type-{III} were investigated. The numerical determination based on several iterative methods was discussed, and formulae for the (Sobolev) gradients of the cost functionals by means of appropriate adjoint problems were derived. Convergence results were shown for the Landweber approach, and the existence and uniqueness of minimisers for the (noisy) cost functionals were established. The proposed schemes were illustrated and discussed on a numerical example in \Cref{sec:numerical}. %
Based on the relative errors, we observe that for $f^0$ (see \eqref{eq:fchoices}), the Landweber method and the steepest descent with $\Leb^2$-gradient perform best for small noise levels, while the conjugate gradient method is better for larger noise levels. The Sobolev versions yield less good results, as these methods do not enforce the approximations to be zero at the boundary as the other methods impose. However, without the extra assumption on the data, the Sobolev methods can reasonably recover the shape of the target source. For the function $f^1$ (see \eqref{eq:fchoices}), the Landweber and $\Leb^2$-gradient methods are of similar quality, but less accurate compared to those for $f^0.$ One of the main reasons is the fixed values at the boundary. Based on the number of iterations, the conjugate gradient methods (both for $\Leb^2$-gradients and Sobolev gradients) need the fewest iterations compared to the Landweber and steepest descent methods. In this view, the conjugate gradient methods yield a fast recovery, and depending on the available knowledge on the boundary either the $\Leb^2$-gradient or the Sobolev version is preferred.

\begin{table}[t]
    \centering 
 \caption{Smallest obtained relative errors $e_r$ for the different used methods and different levels of noise for the functions $f^0$ and $f^1$ for \textbf{ISP1.2}. LW: Landweber method, L2-SD: steepest descent with $\Leb^2$-gradient, S-SD: steepest descent with Sobolev gradient, L2-CGM: conjugate gradient with $\Leb^2$-gradient, S-CGM: conjugate gradient method with Sobolev gradient.}
   
   \footnotesize{ 
    \begin{tabular*}{\linewidth}{@{\extracolsep{\fill}} c c c c c c  }\hline 
    $e_r$ & Method & $0\%$ &  $1\%$ & $3\%$ & $5\%$ \\\hline 
    \multirow{5}{*}{$f^0$}     & LW & 0.0241 & 0.0487 & 0.1020 & 0.1446 \\
     & L2-SD & 0.0197 & 0.0425& 0.0587 &0.0689 \\
     & S-SD & 0.1605  & 0.1718 & 0.2145 & 0.2236\\
     & L2-CGM & 0.0964 & 0.0723 & 0.0515& 0.0852\\
     & S-CGM & 0.3476 & 0.3444& 0.3381& 0.3320\\\hline 
     \multirow{5}{*}{$f^1$} & LW &0.1255 & 0.1381& 0.1661& 0.1956 \\
     & L2-SD & 0.1240 & 0.1253 & 0.1308& 0.1320 \\
     & S-SD & 0.2283 & 0.2258 & 0.2710  &0.2815  \\
     & L2-CGM & 0.1492 & 0.1388& 0.1429& 0.1747\\
     & S-CGM & 0.3162 & 0.3143 & 0.3113& 0.3090 \\\hline 
        \end{tabular*}
   }     
   
    \label{tab:summary}
\end{table}
\clearpage
\section*{Funding}
Dr.\ Karel Van Bockstal is supported by the Methusalem programme of the Ghent University Special Research Fund (BOF) grant number (01M01021).

\section*{Acknowledgements}
We would like to thank Prof.\ Daniel Lesnic from the University of Leeds for valuable discussions and pointing us in the direction of the Sobolev gradient method during the AIP2023 conference. Dr.\ Tim Raeymaekers is thanked for fruitful discussions.

 \bibliography{bibliography}
 \bibliographystyle{abbrv}
 \end{document}